\documentclass[11pt]{amsart}

\usepackage[usenames,dvipsnames,svgnames,table]{xcolor}
\usepackage[colorlinks=true, pdfstartview=FitV, linkcolor=blue, citecolor=blue, urlcolor=darkblue]{hyperref}

\usepackage{geometry}                
\usepackage{graphicx}
\usepackage{mathrsfs}
\usepackage{amssymb}
\usepackage{amsfonts}
\usepackage{mathrsfs}
\usepackage{epstopdf}
\usepackage{lscape}
\usepackage[utf8]{inputenc}
\usepackage{tikz,caption}
\usepackage{enumitem}
\setlist[itemize]{leftmargin=2em}
\setlist[enumerate]{leftmargin=2em}
\usepackage{booktabs}
\usepackage{multirow}
\usepackage{mathtools}
\usepackage[linesnumbered,ruled]{algorithm2e}
\usepackage{fontawesome5}

\usepackage{soul}
\usepackage{upgreek}

\definecolor{darkblue}{rgb}{0.0,0,0.7} 
\definecolor{darkred}{rgb}{0.7,0,0} 
\definecolor{darkgreen}{rgb}{0, .6, 0} 

\newcommand{\defncolor}{\color{darkred}}
\newcommand{\defn}[1]{{\defncolor\emph{#1}}} 

\newtheorem{theorem}{Theorem}[section]
\newtheorem{prop}[theorem]{Proposition}
\newtheorem{cor}[theorem]{Corollary}
\newtheorem{lemma}[theorem]{Lemma}

\theoremstyle{definition}
\newtheorem{definition}[theorem]{Definition}
\newtheorem{example}[theorem]{Example}
\newtheorem{remark}[theorem]{Remark}
\numberwithin{equation}{section}

\newcommand{\idiot}[1]{\vspace{5 mm}\par \noindent
\marginpar{\textsc{Note}}
\framebox{\begin{minipage}[c]{0.95 \textwidth}
#1 \end{minipage}}\vspace{5 mm}\par}

\renewcommand{\idiot}[1]{}

\def\la{{\lambda}}

\def\o{\overline}

\newcommand{\Ph}{\widehat{\mathsf{P}}}
\newcommand{\Dh}{\widehat{\mathsf{D}}}
\newcommand{\Pa}{\mathsf{P}}
\newcommand{\piok}{\pi^{\otimes k}}
\newcommand{\piokk}{\piok_{k+1}}
\newcommand{\stirling}{\genfrac{\{}{\}}{0pt}{}}

\usepackage[colorinlistoftodos]{todonotes}

\newcommand{\mike}[1]{\todo[size=\tiny,color=orange!30]{#1 \\ \hfill --- Mike}}

\newcommand{\rosa}[1]{\todo[size=\tiny,color=Cyan]{#1 \\ \hfill --- Rosa}}

\usepackage{tikz}
\usepackage{tikz-cd}
\usetikzlibrary{calc}
\usetikzlibrary{arrows}
\usepackage{tikzsymbols}
\usetikzlibrary{backgrounds}
\tikzstyle{partition-diagram}=[scale=0.4, thick, baseline={(0,-1ex/2)}, 
]
\tikzstyle{idempotent}=[background rectangle/.style={ultra thick, draw=ForestGreen, rounded corners}]
\tikzstyle{vertex} = [shape = circle, minimum size = 3pt, inner sep = 1pt]
\tikzstyle{representative}=[background rectangle/.style={ultra thick, draw=RoyalBlue, rounded corners}]

\tikzstyle{partition-diagram-small}=[scale=0.15, thin, baseline={(0,-1ex/2)}]
\tikzstyle{vertex-small} = [shape = circle, minimum size = 2pt, inner sep = 1pt]

\usepackage{ytableau}
\ytableausetup{smalltableaux, aligntableaux=center, textmode}

\def\unprotectedboldentry#1{\textcolor{Red}{\large{#1}}}
\def\boldentry{\protect\unprotectedboldentry}
\usepackage{tikz}
\usetikzlibrary{calc}
\usepackage{ifthen}
\newcommand{\tikztableauinternal}[1]{
    \def\newtableau{#1}
    \coordinate (x) at (-0.5,0.5);
    \coordinate (y) at (-0.5,0.5);
    \foreach \row in \newtableau {
        \foreach \entry in \row {
            \ifthenelse{\equal{\entry}{X} \OR \equal{\entry}{None}}
               {
                \node (y) at ($(y) + (1,0)$) {};
                \fill[color=gray!30] ($(y)-(0.5,0.5)$) rectangle +(1,1);
                \draw[color=gray, dotted] ($(y)-(0.5,0.5)$) rectangle +(1,1);
               }
               {
                \ifthenelse{\equal{\entry}{\boldentry X}}
                   {
                    \node (y) at ($(y) + (1,0)$) {};
                    \fill[color=gray] ($(y)-(0.5,0.5)$) rectangle +(1,1);
                    \draw ($(y)-(0.5,0.5)$) rectangle +(1,1);
                   }
                   {
                    \node (y) at ($(y) + (1,0)$) {\entry};
                    \draw ($(y)-(0.5,0.5)$) rectangle +(1,1);
                   }
               }
            }
        \coordinate (x) at ($(x)-(0,1)$);
        \coordinate (y) at (x);
        }
}

\newdimen\squaresize \squaresize=10pt
\newdimen\thickness \thickness=0.4pt

\def\square#1{\hbox{\vrule width \thickness
     \vbox to \squaresize{\hrule height \thickness\vss
        \hbox to \squaresize{\hss#1\hss}
     \vss\hrule height\thickness}
\unskip\vrule width \thickness}
\kern-\thickness}

\def\vsquare#1{\vbox{\square{$#1$}}\kern-\thickness}

\def\young#1{
\vbox{\smallskip\offinterlineskip
\halign{&\vsquare{##}\cr #1}}}

\def\thisbox#1{\kern-.09ex\fbox{#1}}
\def\downbox#1{\lower1.200em\hbox{#1}}
\newdimen\Squaresize \Squaresize=20pt
\newdimen\Thickness \Thickness=0.4pt

\def\Square#1{\hbox{\vrule width \Thickness
     \vbox to \Squaresize{\hrule height \Thickness\vss
        \hbox to \Squaresize{\hss#1\hss}
     \vss\hrule height\Thickness}
\unskip\vrule width \Thickness}
\kern-\Thickness}

\def\Vsquare#1{\vbox{\Square{$#1$}}\kern-\Thickness}

\title[Quasi-partition algebras]{Representations of the quasi-partition algebras}

\author[Orellana]{Rosa Orellana}
\address[R. Orellana]{Mathematics Department, Dartmouth College, 
Hanover, NH 03755, U.S.A.}
\email{Rosa.C.Orellana@dartmouth.edu}
\urladdr{\href{https://math.dartmouth.edu/~orellana/}{https://math.dartmouth.edu/~orellana/}}

\author[Wallace]{Nancy Wallace}
\address[N. Wallace]{Department of Mathematics and Statistics,  York University, 4700 Keele Street, Toronto,
Ontario M3J 1P3, Canada}
\email{nwallace@yorku.ca}

\author[Zabrocki]{Mike Zabrocki}
\address[M. Zabrocki]{Department of Mathematics and Statistics,  York University, 4700 Keele Street, Toronto,
Ontario M3J 1P3, Canada}
\email{zabrocki@mathstat.yorku.ca}
\urladdr{\href{http://garsia.math.yorku.ca/~zabrocki/}{http://garsia.math.yorku.ca/~zabrocki/}}

\allowdisplaybreaks

\begin{document}
\maketitle

\centerline{\it Dedicated to the memory of Georgia Benkart}
\begin{abstract}
The quasi-partition algebras were introduced in \cite{DO} as centralizers of the symmetric group.
In this article, we give a more general definition of these algebras and give a construction
of their simple modules. In addition, we introduce two new algebras, we give linear bases
and show that for specializations of their parameters, these new algebras are isomorphic
to centralizer algebras.  We provide a generalized Bratteli diagram that illustrates how
the representation theory of the three algebras discussed in this paper are related.
Moreover, we give combinatorial formulas for the dimensions of the simple modules of these algebras.
\end{abstract}

\section{Introduction}

The partition algebra was defined independently in the work of
Martin and his coauthors \cite{marbook, mar1, mar3, MS} and Jones \cite{J3} in the early 1990's
as a natural extension of centralizer algebras such as the Brauer and Temperley-Lieb
algebras.  It is of interest for combinatorial representation theory because
it provides a dual approach to resolving some of the open combinatorial problems
related to the representation theory of the symmetric group.

The partition algebras are indexed by an integer $k\geq0$;
however in \cite{mar3} Martin introduced an intermediate algebra
that was recognized to lie between each integer indexed partition algebra.  Halverson and Ram named them 
\emph{half-partition algebras} (indexing them by $k+\frac{1}{2}$ for each
$k\geq0$).  They used the inclusions and  Jones' basic construction  to construct the matrix units, and the irreducible representations of the algebras.

The partition algebras $\mathsf{P}_k(n)$
are the centralizer algebras of the symmetric groups $S_n$
when they act on the vector space $V_n^{\otimes k}$
where $V_n \cong \mathbb{C}\mathrm{-Span}\{v_1,v_2,\ldots, v_n\}$.
One way of viewing this vector space is to notice \cite[Section 7.5.5]{CST}
that as an $S_n$ module, $V_n^{\otimes k}$ can be realized as the
following iteration of induction and restrictions:
\[
V_n^{\otimes k} \cong
({\text{Ind}}^{S_n}_{S_{n-1}}
{\text{Res}}^{S_n}_{S_{n-1}})^k {\mathbb{S}^{(n)}},
\]
where ${\mathbb{S}}^{(n)}$ is the trivial representation of the
symmetric group $S_n$.  The half-partition algebras
$\mathsf{P}_{k+\frac{1}{2}}(n)$
fill in the chain of algebras as the centralizer
of the symmetric group $S_{n-1}$ acting
on ${\text{Res}}^{S_n}_{S_{n-1}} V_n^{\otimes k}$.

The partition algebra shows promise for applications in
developing a better understanding of the representation theory
of the symmetric groups.
Indeed, Bowman-De Vischer and the first author \cite[Equation (3.13)]{BDO}
identified the see-saw pair,
\begin{center}
\begin{tikzpicture}
\node (C) at (0,0)  {$V_n^{\otimes k + \ell}$} ;
\node (S2) at (2,-1)  {$S_n$} ;
\node (S1) at (2,1)  {$S_n \times S_n$} ;
\node (P2) at (-2,-1) {$\mathsf{P}_k(n) \otimes \mathsf{P}_\ell(n)$} ;
\node (P1) at (-2,1) {$\mathsf{P}_{k+\ell}(n)$} ;
\draw[{Hooks[right]}->] (S2.north) to (S1.south) ;
\draw[{Hooks[right]}->] (P2.north) to (P1.south) ;
\draw[-] (S1.south west) to (C.north east) ;
\draw[-] (S2.north west) to (C.south east) ;
\draw[-] (P1.south east) to (C.north west) ;
\draw[-] (-1.4,-.7) to (C.south west) ;
\end{tikzpicture}~,
\end{center}
and used it to show that for sufficiently large $n$,
the reduced Kronecker coefficients \cite{BOR, Mur, OZ2}
are the multiplicities of
an standard $\mathsf{P}_k(n) \otimes \mathsf{P}_\ell(n)$ module in the restriction of an irreducible $\mathsf{P}_{k+\ell}(n)$ module. 
They used that relationship to develop formulae relating
the reduced Kronecker coefficients, Littlewood-Richardson coefficients
and the usual Kronecker coefficients.

This approach seems like a difficult yet promising way to gain a better
combinatorial understanding of the reduced Kronecker coefficients.
Bowman, De Visher and Enyang \cite{BDE} used this connection and bases
for irreducible representations of the partition algebras to
develop combinatorial interpretations for certain classes of
the reduced Kronecker coefficients using the structure of the tower
of algebras.

One source of the complexity of this approach could be that
the space $V_n$ is not irreducible as an $S_n$ module since
$V_n \cong {\mathbb{S}}^{(n-1,1)} \oplus {\mathbb{S}}^{(n)}$.
Daugherty and the first author \cite{DO} introduced a subalgebra of the partition algebra
that is related to this picture by instead of being the centralizer of the action of the
symmetric group on $V_n^{\otimes k}$, is the centralizer of
the symmetric group when it acts on $(\mathbb{S}^{(n-1,1)})^{\otimes k}$.
They showed that the quasi-partition algebra $\mathsf{QP}_k(n)$
can also be realized as a projection of the partition algebra by an idempotent.

We note that there are other diagram algebras that are projections by
idempotents that can be found in the literature.  A recent paper by
Doty and Giaquinto \cite{DG} defines balanced elements of a planar
diagram algebra by conjugating by an idempotent of the partial Temperley-Lieb algebra.
The techniques we use to analyze the quasi-partition algebra
can apply to other algebras.

In this paper we develop the representation theory of the quasi-partition
algebras using a construction similar to the inclusion of the half-partition
algebras.  Let $\mathrm{proj}$ be the projection from $V_n$ to the
$S_n$ irreducible representation $\mathbb{S}^{(n-1,1)}$.  Then we remark
that
\[
(\mathbb{S}^{(n-1,1)})^{\otimes k} \cong
(\mathrm{proj} \circ {\text{Ind}}^{S_n}_{S_{n-1}}
{\text{Res}}^{S_n}_{S_{n-1}})^k {\mathbb{S}^{(n)}}~.
\]
Now in a similar construction to the partition algebra and the
half-partition algebra we have that there are two intermediate algebras
between the centralizer of $(\mathbb{S}^{(n-1,1)})^{\otimes k}$,
namely the centralizers of
\[
{\text{Res}}^{S_n}_{S_{n-1}} (\mathbb{S}^{(n-1,1)})^{\otimes k}\hbox{ and }
{\text{Ind}}^{S_n}_{S_{n-1}} {\text{Res}}^{S_n}_{S_{n-1}} (\mathbb{S}^{(n-1,1)})^{\otimes k}
\cong (\mathbb{S}^{(n-1,1)})^{\otimes k} \otimes V_n
\]
which we name $\mathsf{QP}_{k+\frac{1}{2}}(n)$ and $\widetilde{\mathsf{QP}}_{k+1}(n)$.
This allows us to develop a Bratteli-like
diagram (see Figure \ref{fig:Bdiagram}) to describe the relationship between the irreducible representations.
We say that this is `Bratteli-like' rather than exactly a Bratteli diagram
because the edges between the
irreducibles of $\widetilde{\mathsf{QP}}_{k+1}(n)$ and the irreducibles of
$\mathsf{QP}_{k+1}(n)$ represent a projection rather than an inclusion.

In analogy with the observations in \cite{BDO}, we remark that there
is a similar see-saw pair,
\begin{center}
\begin{tikzpicture}
\node (C) at (.3,0)  {$\left(\mathbb{S}^{(n-1,1)}\right)^{\otimes k + \ell}$} ;
\node (S2) at (2,-1)  {$S_n$} ;
\node (S1) at (2,1)  {$S_n \times S_n$} ;
\node (P2) at (-2,-1) {$\mathsf{QP}_k(n) \otimes \mathsf{QP}_\ell(n)$} ;
\node (P1) at (-2,1) {$\mathsf{QP}_{k+\ell}(n)$} ;
\draw[{Hooks[right]}->] (S2.north) to (S1.south) ;
\draw[{Hooks[right]}->] (P2.north) to (P1.south) ;
\draw[-] (1.4,.7) to (.5,.4) ;
\draw[-] (1.4,-.7) to (.5,-.4) ;
\draw[-] (-1.4,.7) to (-.5,.4) ;
\draw[-] (-1.4,-.7) to (-.5,-.4) ;
\end{tikzpicture}~.
\end{center}
Hence we have in analogy that the reduced Kronecker coefficients are the multiplicities of
an irreducible $\mathsf{QP}_k(n) \otimes \mathsf{QP}_\ell(n)$ module in the restriction of
an irreducible $\mathsf{QP}_{k+\ell}(n)$ module.

This paper develops the
quasi-partition algebras both as centralizer algebras (Theorems \ref{th:QPkcentralizer},
\ref{th:QPkhalfcentralizer} and \ref{th:QPktildecentralizer})
and as projections of the partition algebra conjugated
by an idempotent (Equation \eqref{eq:threealgebras}).
The main results are the construction of a tower of
quasi-partition algebras (Subsection \ref{subsec:quasitower}) and an explicit description of
bases of the irreducible representations of $\mathsf{QP}_k(n)$
(Section \ref{sec:QPirreps}).  The tower of algebras is used to
relate the dimensions of the irreducibles in the tower
using the inclusions and projections (Equations \eqref{eq:p1}, \eqref{eq:p2} and \eqref{eq:p3}).
This allows us to give a combinatorial interpretation of the dimensions of
the irreducibles (Corollary \ref{cor:combdims}) and formulae for the dimensions in terms of
the numbers of set partitions and numbers of standard tableaux (Equations \eqref{eq:fulldim1} and \eqref{eq:halfdim1}).

Section \ref{sec:partitionalg} of this paper begins with a review of the
partition algebras, half-partition algebras,
and in Section \ref{sec:repsPk}, we give a construction of their
irreducibles in terms of the standard modules.  Although the
standard modules for the partition algebras were constructed and expressed in terms of
bases indexed by set valued tableaux in \cite{HJ}, they do not include the standard modules
of the half-partition algebras and so we give a description of those modules
in Subsection \ref{sec:halfPkreps}.
In Section \ref{sec:quasipalgebra}, we present three families of quasi-partition algebras,
and use them to build a tower of algebras.
In Section \ref{sec:QPirreps}, we construct the irreducible representations of
quasi-partition algebras.
In Section \ref{sec:quasicentralizers}, we show that the partition algebras that we defined
as projections of the partition algebra are isomorphic to centralizers
when the parameter is specialized to values of $n$ that are sufficiently large.
Finally, in Section \ref{sec:formulaedim} we use the tower of algebras to give formulae for the
dimensions of the irreducible representations.

\subsection*{Acknowledgements}
R. Orellana was partially supported by NSF grant DMS-2153998
and both N. Wallace and M. Zabrocki are supported by NSERC. 

\section{The partition algebras}\label{sec:partitionalg}

The partition algebra was originally defined by Martin in \cite{marbook}.  All the results in this section are due to Martin and his collaborators, see \cite{mar1} and references therein.  For a nice survey on the partition algebra see \cite{HR}.  
\subsection{Definitions} \label{subsec:defs}
For $k\in \mathbb{Z}_{>0}$, $x\in\mathbb{C}$, we let $\mathsf{P}_k(x)$
denote the complex vector space with bases given by  all set partitions of
$[k] \cup [\o{k}] := \{1,2,\ldots, k, \o{1},\o{2}, \ldots, \o{k}\}.$
A part of a set partition is called a \emph{block}.
For a given block $B$, the set  $B \cap [\o{k}]$
denotes the subset of all barred elements of $B$ are referred to as the \defn{bottom} of $B$
and the set $B \cap [k]$ denotes the subset of all unbarred elements of $B$ and is referred to as the \defn{top} of $B$. 
Notice that for a given set partition $d$ on $2k$ elements,
then $d\cap[k]$ and $d\cap[\o{k}]$ are set partitions on $k$ elements.
We will let $\Ph_k$ denote the set of all set partitions
of $\{1,2,\ldots, k, \o{1},\o{2}, \ldots, \o{k}\}.$

Blocks with a single element will either be referred to as \defn{singletons} or an
\defn{isolated vertex}.  Blocks containing at least one element from $[k]$
and one element from $[\o{k}]$ will be called \defn{propagating blocks}; all other blocks will
be called \defn{non-propagating blocks}.

For example,
$$d=\{\{1, 2, 4, \o{2}, \o{5}\}, \{3\}, \{5, 6, 7, \o{3}, \o{4}, \o{6}, \o{7}\}, \{8, \o{8}\}, \{\o{1}\}\},$$
is a set partition (for $k=8$) with 5  blocks. The block $B=\{1, 2, 4, \o{2}, \o{5}\}$ is  propagating. The block $\{3\}$ is a singleton.

A  set partition in $\Ph_k$ can be represented
 by a \defn{partition diagram}
 consisting of a frame with $k$ distinguished points
 on the top and bottom boundaries, which we call vertices.
 We number the top vertices from left to right by $1,2,\ldots, k$ and the bottom vertices similarly by $\o{1},\o{2},\ldots, \o{k}$. We create a graph with connected components corresponding to the blocks of the set partition such that there is a path of edges between two vertices
 if they belong to the same block.  Note that such a diagram is not uniquely defined,
 two diagrams representing the set partition $d$ above are given in Figure \ref{2diag}. Hence, a partition diagram is an equivalence class of graphs, where the equivalence is given by having the same connected components. In addition, we often omit the numbering on the vertices in the interest of keeping the diagrams less cluttered.

We will use the word \defn{diagram} to refer to any element of $\Ph_k$
or equivalently its partition diagram.

\begin{figure}[ht]
\begin{equation*}
\begin{array}{c}
\scalebox{1.0}{\begin{tikzpicture}[scale=.55,line width=1.35pt]
\foreach \i in {1,...,8}
{ \path (\i,2) coordinate (T\i); \path (\i,0) coordinate (B\i); }
\filldraw[fill=gray!25,draw=gray!25,line width=4pt]  (T1) -- (T8) -- (B8) -- (B1) -- (T1);
\draw (T1) .. controls +(.1,-.6) and +(-.1,-.6) .. (T2);
\draw (T2) .. controls +(.1,-.6) and +(-.1,-.6) .. (T4);
\draw (T4) -- (B5);
\draw (T5) .. controls +(.1,-.6) and +(-.1,-.6) .. (T6);
\draw (T6) .. controls +(.1,-.6) and +(-.1,-.6) .. (T7);
\draw (T7) -- (B7);
\draw (T8) -- (B8);
\draw (B3) .. controls +(.1,.6) and +(-.1,.6) .. (B4);
\draw (B2) .. controls +(.1,1) and +(-.1,1) .. (B5);
\draw (B4) .. controls +(.1,.8) and +(-.1,.8) .. (B6);
\draw (B6) .. controls +(.1,.6) and +(-.1,.6) .. (B7);
\foreach \i in {1,...,8}
{\draw  (B\i)  node[below=0.05cm]{${\overline{\i}}$}; \draw  (T\i)  node[above=0.05cm]{${\i}$};}
\foreach \i in {1,...,8}
{ \fill (T\i) circle (4.5pt); \fill (B\i) circle (4.5pt); }
\end{tikzpicture}} \end{array}
\hskip .3in
\begin{array}{c}
\scalebox{1.0}{\begin{tikzpicture}[scale=.55,line width=1.35pt]
\foreach \i in {1,...,8}
{ \path (\i,2) coordinate (T\i); \path (\i,0) coordinate (B\i); }
\filldraw[fill=gray!25,draw=gray!25,line width=4pt]  (T1) -- (T8) -- (B8) -- (B1) -- (T1);
\draw (T1) .. controls +(.1,-.6) and +(-.1,-.6) .. (T2);
\draw (T1) -- (B2) ;%
\draw (B2) -- (T4) ;%
\draw (T4) -- (B5);%
\draw (T5) .. controls +(.1,-1) and +(-.1,-1) .. (T7);
\draw (T6) .. controls +(.1,-.6) and +(-.1,-.6) .. (T7);
\draw (T6) -- (B7);
\draw (T8) -- (B8);
\draw (B3) .. controls +(.1,.6) and +(-.1,.6) .. (B4);
\draw (B4) .. controls +(.1,.8) and +(-.1,.8) .. (B6);
\draw (B6) .. controls +(.1,.6) and +(-.1,.6) .. (B7);
\foreach \i in {1,...,8}
{\draw  (B\i)  node[below=0.05cm]{${\overline{\i}}$}; \draw  (T\i)  node[above=0.05cm]{${\i}$};}
\foreach \i in {1,...,8}
{ \fill (T\i) circle (4.5pt); \fill (B\i) circle (4.5pt); }
\end{tikzpicture}} \end{array}.
\end{equation*}

  \caption{Two representative graphs of the set partition $d$.}
\label{2diag}
\end{figure}
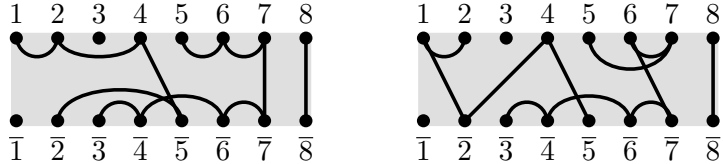

We define an internal product, $d_1 \cdot d_2$, of two diagrams $d_1$
and $d_2$ using the concatenation of $d_1$ above $d_2$, where we identify
the bottom vertices of $d_1$ with the top vertices of $d_2$.   
If there are $m$ connected components consisting only of  middle vertices, then 
\[d_1\cdot d_2 = x^m d_3\]
where $d_3$ is the diagram with the middle vertices components removed.

\begin{example} For example, consider the two diagrams
$d_1, d_2 \in \mathsf{P}_{6}(x)$ where
$d_1 =$ $\{\{1,3,\o4\},$ $\{2,\o1\},\{4,5,6,\o5\},\{\o2,\o3\},\{\o6\}\}$
and $d_2 = \{\{1\},\{2,3\},\{4,\o1,\o2,\o4\},\{5,\o6\},\{6\},\{\o3,\o5\}\}$.
These elements are represented by the diagrams
\begin{equation*} 
d_1 =
\begin{array}{c}
\scalebox{1.0}{\begin{tikzpicture}[scale=.55,line width=1.35pt]
\foreach \i in {1,...,6}
{ \path (\i,1) coordinate (T\i); \path (\i,0) coordinate (B\i); }
\filldraw[fill=gray!25,draw=gray!25,line width=4pt]  (T1) -- (T6) -- (B6) -- (B1) -- (T1);
\draw (T1) .. controls +(.1,-.4) and +(-.1,-.4) .. (T3);
\draw (T4) .. controls +(.1,-.4) and +(-.1,-.4) .. (T5);
\draw (T5) .. controls +(.1,-.4) and +(-.1,-.4) .. (T6);
\draw (T6) -- (B5);
\draw (T5) .. controls +(.1,-.4) and +(-.1,-.4) .. (T6);
\draw (T2) -- (B1);
\draw (T3) -- (B4);
\draw (B2) .. controls +(.1,.4) and +(-.1,.4) .. (B3);
\foreach \i in {1,...,6}
{ \filldraw (T\i) circle (2pt); \filldraw (B\i) circle (2pt); }
\end{tikzpicture}} \end{array}\quad \text{ and } \quad
d_2 = 
\begin{array}{c}
\scalebox{1.0}{\begin{tikzpicture}[scale=.55,line width=1.35pt]
\foreach \i in {1,...,6}
{ \path (\i,1) coordinate (T\i); \path (\i,0) coordinate (B\i); }
\filldraw[fill=gray!25,draw=gray!25,line width=4pt]  (T1) -- (T6) -- (B6) -- (B1) -- (T1);
\draw (T2) .. controls +(.1,-.4) and +(-.1,-.4) .. (T3);
\draw (T4) -- (B4);
\draw (T5) -- (B6);
\draw (B1) .. controls +(.1,.4) and +(-.1,.4) .. (B2);
\draw (B2) .. controls +(.1,.4) and +(-.1,.4) .. (B4);
\draw (B3) .. controls +(.1,.4) and +(-.1,.4) .. (B5);
\foreach \i in {1,...,6}
{ \filldraw (T\i) circle (2pt); \filldraw (B\i) circle (2pt); }
\end{tikzpicture}} \end{array}~.
\end{equation*}
The product in the algebra may be calculated by
stacking the two diagrams and removing the middle row
of vertices, so that
\begin{equation*}
d_1d_2 =
\begin{array}{c}
\scalebox{1.0}{\begin{tikzpicture}[scale=.55,line width=1.35pt]
\foreach \i in {1,...,6}
{ \path (\i,1) coordinate (T\i); \path (\i,0) coordinate (B\i); }
\filldraw[fill=gray!25,draw=gray!25,line width=4pt]  (T1) -- (T6) -- (B6) -- (B1) -- (T1);
\draw (T1) .. controls +(.1,-.4) and +(-.1,-.4) .. (T3);
\draw (T4) .. controls +(.1,-.4) and +(-.1,-.4) .. (T5);
\draw (T5) .. controls +(.1,-.4) and +(-.1,-.4) .. (T6);
\draw (T6) -- (B5);
\draw (T5) .. controls +(.1,-.4) and +(-.1,-.4) .. (T6);
\draw (T2) -- (B1);
\draw (T3) -- (B4);
\draw (B2) .. controls +(.1,.4) and +(-.1,.4) .. (B3);
\foreach \i in {1,...,6}
{ \filldraw (T\i) circle (2pt); \filldraw (B\i) circle (2pt); }
\end{tikzpicture}}\\
\scalebox{1.0}{\begin{tikzpicture}[scale=.55,line width=1.35pt]
\foreach \i in {1,...,6}
{ \path (\i,1) coordinate (T\i); \path (\i,0) coordinate (B\i); }
\filldraw[fill=gray!25,draw=gray!25,line width=4pt]  (T1) -- (T6) -- (B6) -- (B1) -- (T1);
\draw (T2) .. controls +(.1,-.4) and +(-.1,-.4) .. (T3);
\draw (T4) -- (B4);
\draw (T5) -- (B6);
\draw (B1) .. controls +(.1,.4) and +(-.1,.4) .. (B2);
\draw (B2) .. controls +(.1,.4) and +(-.1,.4) .. (B4);
\draw (B3) .. controls +(.1,.4) and +(-.1,.4) .. (B5);
\foreach \i in {1,...,6}
{ \filldraw (T\i) circle (2pt); \filldraw (B\i) circle (2pt); }
\end{tikzpicture}}
\end{array}
= x^2
\begin{array}{c}
\scalebox{1.0}{\begin{tikzpicture}[scale=.55,line width=1.35pt]
\foreach \i in {1,...,6}
{ \path (\i,1) coordinate (T\i); \path (\i,0) coordinate (B\i); }
\filldraw[fill=gray!25,draw=gray!25,line width=4pt]  (T1) -- (T6) -- (B6) -- (B1) -- (T1);
\draw (T1) .. controls +(.1,-.4) and +(-.1,-.4) .. (T3);
\draw (T4) .. controls +(.1,-.4) and +(-.1,-.4) .. (T5);
\draw (T5) .. controls +(.1,-.4) and +(-.1,-.4) .. (T6);
\draw (T1) -- (B1);
\draw (T6) -- (B6);
\draw (B1) .. controls +(.1,.4) and +(-.1,.4) .. (B2);
\draw (B2) .. controls +(.1,.4) and +(-.1,.4) .. (B4);
\draw (B3) .. controls +(.1,.4) and +(-.1,.4) .. (B5);
\foreach \i in {1,...,6}
{ \filldraw (T\i) circle (2pt); \filldraw (B\i) circle (2pt); }
\end{tikzpicture}}~. \end{array}
\end{equation*}
\end{example}

Extending this by linearity defines a multiplication on $\Pa_k(x)$.
With this product, $\Pa_k(x)$ becomes an associative algebra with unit of dimension $B(2k)$,
the \defn{Bell number} which enumerates the number of set partitions of a set with $2k$ elements.


Given diagrams $d_1\in \Ph_{k_1}$ and $d_2\in \Ph_{k_2}$, we denote by $d_1\otimes d_2$
the diagram in $\Ph_{k_1+k_2}$ obtained by placing $d_2$ to the right of $d_1$.
Alternatively, in terms of set partition notation
$$d_1 \otimes d_2 = d_1 \cup \{ \{ b+k_1 : b \in B \} : B \in d_2\}~.$$
Although we will use the notation only rarely, this external product is extended linearly
to a product of elements from $\mathsf{P}_k(x)$ and $\mathsf{P}_\ell(x)$
with the result being an element in $\mathsf{P}_{k + \ell}(x)$.

All partition diagrams can be expressed by combining the following building blocks:
\[ \mathbf{1} :=
\raisebox{-.1in}{
\begin{tikzpicture}[scale=.55, line width=1.35]
	\filldraw[fill=gray!25,draw=gray!25,line width=4pt]  (-.1,0) -- (-.1,1) -- (.1,1) -- (.1,0) -- (-.1,0);
	 \filldraw [black] (0, 0) circle (2pt);
	 \filldraw [black] (0, 1) circle (2pt);
\draw (0,0)--(0,1);
\end{tikzpicture}}, \qquad
\mathsf{p} :=
\raisebox{-.1in}{
\begin{tikzpicture}[scale=.55, line width=1.35]
	\filldraw[fill=gray!25,draw=gray!25,line width=4pt]  (-.1,0) -- (-.1,1) -- (.1,1) -- (.1,0) -- (-.1,0);
	 \filldraw [black] (0, 0) circle (2pt);
	 \filldraw [black] (0, 1) circle (2pt);
\end{tikzpicture}}, \qquad
\mathsf{b} :=
\raisebox{-.1in}{
\begin{tikzpicture}[scale=.55, line width=1.35]
	\filldraw[fill=gray!25,draw=gray!25,line width=4pt]  (0,0) -- (0,1) -- (1,1) -- (1,0) -- (0,0);
	 \foreach \x in {0,1} {
	 \filldraw [black] (\x, 0) circle (2pt);
	 \filldraw [black] (\x, 1) circle (2pt);
	 }
	\draw (0,0)--(1,0)--(1,1)--(0,1)--(0,0);
\end{tikzpicture}}, \qquad
\mathsf{s} :=
\raisebox{-.1in}{
\begin{tikzpicture}[scale=.55, line width=1.35]
	\filldraw[fill=gray!25,draw=gray!25,line width=4pt]  (0,0) -- (0,1) -- (1,1) -- (1,0) -- (0,0);
	 \foreach \x in {0,1} {
	 \filldraw [black] (\x, 0) circle (2pt);
	 \filldraw [black] (\x, 1) circle (2pt);
	 }
	\draw (0,0)--(1,1) (0,1)--(1,0);
\end{tikzpicture}}~.
\]

It is known that $\Pa_k(x)$ is generated by the elements $\mathsf{s}_i$, $\mathsf{b}_i$  ($1\leq i< k$)  and $\mathsf{p}_j$ ($1\leq j\leq k$) depicted below (see \cite[Proposition 1]{mar1}).  The symbol used for
the elements does not contain an implicit value of $k$ (the maximum value in the diagram)
which may need to be determined from which algebra the element lies in.

\[ \mathsf{p}_j :=\mathbf{1}^{\otimes j -1}\otimes \mathsf{p} \otimes \mathbf{1}^{\otimes k-j} =
\raisebox{-.1in}{
\begin{tikzpicture}[scale=.55, line width=1.35]
	\filldraw[fill=gray!25,draw=gray!25,line width=4pt]  (0,0) -- (0,1) -- (6,1) -- (6,0) -- (0,0);
	 \foreach \x in {0,2,3,4,6} {
	 \filldraw [black] (\x, 0) circle (2pt);
	 \filldraw [black] (\x, 1) circle (2pt);
	 }
    \draw (1,.5) node {$\cdots$};
    \draw (5,.5) node {$\cdots$};
	\draw (0,0)--(0,1) (2,0)--(2,1) (4,0)--(4,1) (6,0)--(6,1);
\end{tikzpicture}}
\]
\[
\mathsf{b}_i := \mathbf{1}^{\otimes i -1}\otimes \mathsf{b} \otimes \mathbf{1}^{\otimes k-i-1} =
\raisebox{-.1in}{
\begin{tikzpicture}[scale=.55, line width=1.35]
	\filldraw[fill=gray!25,draw=gray!25,line width=4pt]  (0,0) -- (0,1) -- (7,1) -- (7,0) -- (0,0);
	 \foreach \x in {0,2,3,4,5,7} {
	 \filldraw [black] (\x, 0) circle (2pt);
	 \filldraw [black] (\x, 1) circle (2pt);
	 }
    \draw (1,.5) node {$\cdots$};
    \draw (6,.5) node {$\cdots$};
	\draw (0,0)--(0,1) (2,0)--(2,1) (3,0)--(3,1)--(4,1)--(4,0)--(3,0) (5,0)--(5,1) (7,0)--(7,1);
\end{tikzpicture}},
\]
\[ \mathbf{s}_i := \mathbf{1}^{\otimes i -1}\otimes \mathsf{s} \otimes \mathbf{1}^{\otimes k-i-1} =
\raisebox{-.1in}{
\begin{tikzpicture}[scale=.55, line width=1.35]
	\filldraw[fill=gray!25,draw=gray!25,line width=4pt]  (0,0) -- (0,1) -- (7,1) -- (7,0) -- (0,0);
	 \foreach \x in {0,2,3,4,5,7} {
	 \filldraw [black] (\x, 0) circle (2pt);
	 \filldraw [black] (\x, 1) circle (2pt);
	 }
    \draw (1,.5) node {$\cdots$};
    \draw (6,.5) node {$\cdots$};
	\draw (0,0)--(0,1) (2,0)--(2,1) (3,0)--(4,1) (4,0)--(3,1) (5,0)--(5,1) (7,0)--(7,1);
\end{tikzpicture}}
.
 \]
Another element that plays an important role in diagram algebras is the following: 
\[\mathsf{e} :=
\raisebox{-.1in}{
\begin{tikzpicture}[scale=.55, line width=1.35]
	\filldraw[fill=gray!25,draw=gray!25,line width=4pt]  (0,0) -- (0,1) -- (1,1) -- (1,0) -- (0,0);
	 \foreach \x in {0,1} {
	 \filldraw [black] (\x, 0) circle (2pt);
	 \filldraw [black] (\x, 1) circle (2pt);
	 }
	\draw (0,0)--(1,0) (0,1)--(1,1);
\end{tikzpicture}}
, \quad \text{ and } \quad
\mathsf{e}_i := \mathbf{1}^{\otimes i -1}\otimes \mathsf{e} \otimes \mathbf{1}^{\otimes k-i-1} =
\raisebox{-.1in}{
\begin{tikzpicture}[scale=.55, line width=1.35]
	\filldraw[fill=gray!25,draw=gray!25,line width=4pt]  (0,0) -- (0,1) -- (7,1) -- (7,0) -- (0,0);
	 \foreach \x in {0,2,3,4,5,7} {
	 \filldraw [black] (\x, 0) circle (2pt);
	 \filldraw [black] (\x, 1) circle (2pt);
	 }
    \draw (1,.5) node {$\cdots$};
    \draw (6,.5) node {$\cdots$};
	\draw (0,0)--(0,1) (2,0)--(2,1) (3,0)--(4,0) (3,1)--(4,1) (5,0)--(5,1) (7,0)--(7,1);
\end{tikzpicture}}.
\]
It can be easily checked that $\mathsf{e}_i = \mathsf{b}_i \mathsf{p}_i \mathsf{p}_{i+1} \mathsf{b}_i$ in $\Pa_k(x)$.

We remark that the generators $\mathsf{s}_i$ in $\Pa_k(x)$ generate the symmetric group $S_k$.  In addition, $\mathsf{p}_i^2= x\mathsf{p}_i$ and $\mathsf{b}_i^2 = \mathsf{b}_i$.  Which implies that $\mathsf{b}_i$s are idempotent and the $\mathsf{p}_i$s are multiples of idempotents. For a complete presentation of the partition algebra see \cite{HR} Theorem 1.11.

\vskip .1in

\noindent \textbf{Assumption:} The algebra $\mathsf{P}_k(x)$ is defined for any $x\in \mathbb{C}$.
However, we will be dividing by $x$; therefore, \emph{we assume throughout the paper that $x\neq 0$.}

\subsection{The half-partition algebra}
For $k\in \mathbb{Z}_{\ge 0}$, we define $\Pa_{k+\frac{1}{2}}(x)$
to be the subalgebra of $\Pa_{k+1}(x)$ spanned by diagrams $d \in \Ph_{k+1}$
such that $k+1$ and $\o{k+1}$ are in the same block.
We also set $\Pa_0(x)= \mathbb{C}(x)$ and $\Pa_{\frac{1}{2}}(x)$ as the span of $\mathbf{1}$.
The set of the set partitions, where $k+1$ and $\o{k+1}$
are in the same block, is closed under the diagram product.
These subalgebras were originally studied by Martin \cite{mar3}
(the notation for these algebras in that paper is $P_n^1(Q)$)
but here we follow the notation of in \cite{HR} and denote these subalgebras
$\Pa_{k+\frac{1}{2}}(x)$.

The algebra $\Pa_{k+\frac{1}{2}}(x)$ is generated by
$\mathsf{p}_i$, $\mathsf{b}_i$ and $\mathsf{s}_j$, where $1\leq i\leq k$ and $1 \leq j <k$
and the elements are set partitions with maximum value $k+1, \o{k+1}$.
Notice that $\Pa_{k+\frac{1}{2}}(x)$ is a subalgebra of
$\Pa_{k+1}(x)$ and has all of the same generators except for $\mathsf{p}_{k+1}$
and $\mathsf{s}_k$.
The algebra $\Pa_{k+\frac{1}{2}}(x)$ has  dimension $B(2k+1)$,
the number of set partitions of a set with $2k+1$ elements.
There is an injection of $\Pa_k(x)$ into $\Pa_{k+\frac{1}{2}}(x)$
which sends each element $d \in \Pa_k(x)$ to $d \otimes \mathbf{1}$.
This gives the tower of algebras:

$$\Pa_0(x) \hookrightarrow \Pa_{\frac{1}{2}}(x) \subset \Pa_1(x)
\hookrightarrow \Pa_{1+\frac{1}{2}}(x) \subset \Pa_2(x)
\hookrightarrow \Pa_{2+\frac{1}{2}}(x) \subset \Pa_3(x)
\hookrightarrow \cdots~.$$

\begin{example}
Let $r \in \frac{1}{2} \mathbb{Z}_{\geq 0}$, the sequence of dimensions of the algebras
$\mathsf{P}_{r}(x)$ for $r = 0,\frac{1}{2},1,\frac{3}{2},2,\ldots$ is given by
the number of set partitions of a set of size $2r$ and is equal to the sequence
$$1,1,2,5,15,52, 203, 877, 4140, 21147, \ldots~.$$
\end{example}
\subsection{Partition algebras as centralizers} Let $V_n = \mathbb{C}^n$, the symmetric group acts on $V_n$ via the permutation matrices
\[ \sigma \cdot v_i = v_{\sigma(i)}, \quad \text{ for } \sigma \in S_n.\]
Thus, $S_n$ acts diagonally on a basis of simple tensors of $V_n^{\otimes k}$, 
\[ \sigma \cdot (v_{i_1} \otimes \cdots \otimes v_{i_k}) = v_{\sigma(i_1)} \otimes \cdots \otimes v_{\sigma(i_k)} \]

The action of $\mathsf{P}_{r}(x)$ on an element of $V_n^{\otimes k}$ is given in Equation \eqref{eq:daction}.
Using this action one can show the following isomorphisms. 

For $n \geq 2k$, $k\in \mathbb{Z}_{\geq 0}$,
\[ \mathsf{P}_k(n) \cong \text{End}_{S_n}(V_n^{\otimes k}) \]
and for $n \geq 2k +1$ and $k\in \mathbb{Z}_{\geq 0}$,
\[ \mathsf{P}_{k+\frac{1}{2}}(n) \cong \text{End}_{S_{n-1}}(\text{Res}^{S_n}_{S_{n-1}}V_n^{\otimes k})~. \]
For details and proofs see \cite{CST, HR, J3}.

\section{The representations of $\Pa_k(x)$}
\label{sec:repsPk}
The representation theory of the partition algebra was first studied by Martin and his collaborators,
the main references are \cite{mar1, mar3, MS} and references therein.
The algebra $\Pa_{k}(x)$ is semisimple \cite{MS} whenever $x \notin \{0, 1, 2, \ldots, 2k-2\}$. We will assume that $\mathsf{P}_k(x)$ is semisimple and in this case we will construct ``standard modules''
which form a complete set of simple modules.
In addition, we give the correspondence to set value tableaux \cite{HJ}, when doing this we further assume that $x=n$ is an integer such that $n\geq 2k$.

Recall that a block $B$ is called propagating if it contains at least one element from each $[k]$ and $[\overline{k}]$.  Notice that the product does not increase the number of propagating blocks,
in fact, if $p(d)$ denotes the number of propagating blocks and $d_1$ and $d_2$ are two diagrams, then
\[p(d_1d_2)\leq \min\{p(d_1), p(d_2)\}.\]
This gives a filtration of the algebra $\mathsf{P}_k(x)$ by number of propagating blocks.
For $m\in [k]$, set $I_{k,m} = \{m+1, m+2, \ldots k\}$, we define the following idempotents

\[ \tilde{\mathsf{p}}_{I_{k,m}} := \frac{1}{x^i}\mathsf{p}_{m+1}\mathsf{p}_{m+2}\cdots \mathsf{p}_k =
\raisebox{-.1in}{
\begin{tikzpicture}[scale=.55, line width=1.35]
	\filldraw[fill=gray!25,draw=gray!25,line width=4pt]  (0,0) -- (0,1) -- (6,1) -- (6,0) -- (0,0);
	 \foreach \x in {0,2,3,4,6} {
	 \filldraw [black] (\x, 0) circle (2pt);
	 \filldraw [black] (\x, 1) circle (2pt);
	 }
    \draw (1,.5) node {$\cdots$};
    \draw (5,.5) node {$\cdots$};
	\draw (0,0)--(0,1) (2,0)--(2,1); 
\end{tikzpicture}} ~.
\]
This is the diagram  where the vertices $m+1, m+2, \ldots, k$ and
$\overline{m+1}, \overline{m+2}, \ldots \overline{k}$ are isolated.
The propagation number of $\tilde{\mathsf{p}}_{I_{k,m}}$ is $m$.

Let $\mathbf{J}_m = \mathsf{P}_k(x) \tilde{\mathsf{p}}_{I_{k,m}} \mathsf{P}_k(x)$ is the two-sided
ideal spanned by all the diagrams with at most $m$ propagation blocks. This gives the filtration
\[ \mathbf{J}_k \subset \mathbf{J}_{k-1} \subset \cdots \subset \mathbf{J}_1 \subset \mathsf{P}_k(x).\]
And it can be shown that 
\begin{equation}\label{isoPk}
\tilde{\mathsf{p}}_{I_{k,m}} \mathsf{P}_k(x) \tilde{\mathsf{p}}_{I_{k,m}} \cong \mathsf{P}_{m}(x)
\end{equation}
and that 
\begin{equation} \label{isotoSk}
\mathsf{P}_k(x) /\mathbf{J}_1 \cong \mathbb{C}S_k.
\end{equation}
We note that the diagram elements of $\mathsf{P}_{m}(x)$ that have propagating number equal to $m$
are the permutations of $[m]$ and the other diagrams of $\mathsf{P}_{m}(x)$ have propagating
number strictly smaller than $m$.

This implies that an $S_{m}$ module can be \emph{inflated} to a $\mathsf{P}_k(x)$-module and by induction 
using \eqref{isoPk} and \eqref{isotoSk} and the following construction.
Given $\nu\in \mathsf{Par}_{\leq k}$ and $\nu\vdash m$, let
$\mathbb{S}^\nu$ be the irreducible representation of the symmetric
group $S_{m}$. We will identify this representation with a representation
of $\mathsf{P}_{m}(x)$ such that diagrams with propagation number less than $m$ act by zero.
We define a $\mathsf{P}_k(x)$-module,
$\Delta_k(\nu)$, by
\begin{equation}\label{standardmod}
\Delta_k(\nu) = \mathsf{P}_k(x) \tilde{\mathsf{p}}_{I_{k,m}} \otimes_{\mathsf{P}_{m}(x)} \mathbb{S}^\nu,
\end{equation}
where we have identified $\mathsf{P}_{m}(x)$ with
$\tilde{\mathsf{p}}_{I_{k,m}} \mathsf{P}_k(x) \tilde{\mathsf{p}}_{I_{k,m}}$ using the isomorphism in
\eqref{isoPk}.

It then follows that the simple $\mathsf{P}_k(x)$ modules are
indexed by the  set $\mathsf{Par}_{\leq k}$ containing all partitions of all non-negative integers
less than or equal to $k$.

We now describe a basis for the modules $\Delta_k(\nu)$ for any $m \leq k\in \mathbb{Z}_{\geq 0}$ and $\nu \vdash m$. Define $V(k, m)$ to be the vector space spanned by the diagrams corresponding to set partitions  of $[k]\cup[\overline{k}]$ with $\overline{m+1}, \ldots, \overline{k}$ in singleton blocks, and all other $\overline{j}$ are in propagating blocks where $\overline{j}$ is the only barred element in its block. We call these \defn{$(k,m)$-diagrams}. 

\begin{example} An example of two diagrams in $V(6,3)$:
\begin{center}
\scalebox{1.0}{\begin{tikzpicture}[scale=.55,line width=1.35pt]
\foreach \i in {1,...,6}
{ \path (\i,1) coordinate (T\i); \path (\i,0) coordinate (B\i); }
\filldraw[fill=gray!25,draw=gray!25,line width=4pt]  (T1) -- (T6) -- (B6) -- (B1) -- (T1);
\draw (T1) .. controls +(.1,-.4) and +(-.1,-.4) .. (T3);
\draw (T4) .. controls +(.1,-.4) and +(-.1,-.4) .. (T5);
\draw (T4) -- (B1);
\draw (T2) -- (B3);
\draw (T1) -- (B2);
\foreach \i in {1,...,6}
{ \filldraw (T\i) circle (2pt); \filldraw (B\i) circle (2pt); }
\end{tikzpicture}}, \qquad
{\begin{tikzpicture}[scale=.55,line width=1.35pt]
\foreach \i in {1,...,6}
{ \path (\i,1) coordinate (T\i); \path (\i,0) coordinate (B\i); }
\filldraw[fill=gray!25,draw=gray!25,line width=4pt]  (T1) -- (T6) -- (B6) -- (B1) -- (T1);
\draw (T1) .. controls +(.1,-.4) and +(-.1,-.4) .. (T3);
\draw (T4) .. controls +(.1,-.4) and +(-.1,-.4) .. (T6);
\draw (T5) -- (B3);
\draw (T2) -- (B1);
\draw (T3) -- (B2);
\foreach \i in {1,...,6}
{ \filldraw (T\i) circle (2pt); \filldraw (B\i) circle (2pt); }
\end{tikzpicture}}~.
\end{center}
\end{example}
In what follows, we identify $\sigma \in S_{m}$ with an element of the subgroup of $S_k$
contained in $\mathsf{P}_k(x)$ which
contains blocks $\{1, \o{\sigma(1)}\}, \ldots, \{ m, \o{\sigma(m)} \},
\{m+1, \overline{m+1}\}, \ldots \{k, \overline{k}\}$.

\begin{lemma}\label{V-module} If $d\in V(k,m)$, there exists $\sigma \in S_{m}$ such that $d = d' \sigma$ where $d'$ satisfies the condition that
if $B_{\overline{j}}$ denotes the propagating block containing $\overline{j}$, then $\text{max}(B_{\overline{j-1}}\cap [k]) < \text{max}(B_{\overline{j}}\cap [k])$ for all $1\leq j \leq m$.
\end{lemma}
\begin{proof}
Let $A_1$, $A_2, \ldots, A_{m}$ be the propagating blocks of $d \in V(k,m)$
ordered by largest element in $A_i \cap [k]$.
Define a permutation $\sigma \in S_m$, where for any $i\in [m]$, $\sigma(i) = j$ if
$\overline{j} \in A_i$.
\end{proof}

Because of Lemma \ref{V-module}, $V(k,m)$ has a natural structure of a $(\mathsf{P}_k(x), S_{m})$-bimodule, where $\mathsf{P}_k(x)$ acts on the left by diagram multiplication, with the understanding that the product is zero if the resulting diagram in not in $V(k,m)$ and $S_{m}$ acts on the right with diagram multiplication that permutes $\overline{1}, \ldots, \overline{m}$.

As vector spaces, we have
\begin{equation}
\Delta_k(\nu) \cong V(k,m)\otimes_{S_{m}} \mathbb{S}(\nu)~.
\end{equation}
A $(k,m)$-diagram is called \defn{$(k,m)$-standard} if its propagating blocks $B_{\overline{1}}, \ldots, B_{\overline{m}}$ satisfy $\text{max}(B_{\overline{j-1}}\cap [k]) < \text{max}(B_{\overline{j}}\cap [k])$ for all $1\leq j \leq m$.

\begin{example} If $i = 1$ and $k=3$, then the $(3,2)$-standard diagrams are:
\[
\scalebox{1}{{\begin{tikzpicture}[scale=.55,line width=1.35]
\foreach \i in {1,...,3}
{ \path (\i,1) coordinate (T\i); \path (\i,0) coordinate (B\i); }
\filldraw[fill=gray!25,draw=gray!25,line width=4pt]  
(T1) -- (T3) -- (B3) -- (B1) -- (T1);
\draw (T1) -- (B1);
\draw (T2) -- (B2);
\foreach \i in {1,...,3}
{ \filldraw (T\i) circle (2pt); \filldraw (B\i) circle (2pt); }
\end{tikzpicture}}, \quad
{\begin{tikzpicture}[scale=.55,line width=1.35pt]
\foreach \i in {1,...,3}
{ \path (\i,1) coordinate (T\i); \path (\i,0) coordinate (B\i); }
\filldraw[fill=gray!25,draw=gray!25,line width=4pt]  
(T1) -- (T3) -- (B3) -- (B1) -- (T1);
\draw (T1) -- (B1);
\draw (T3) -- (B2);
\foreach \i in {1,...,3}
{ \filldraw (T\i) circle (2pt); \filldraw (B\i) circle (2pt); }
\end{tikzpicture}}, \quad
{\begin{tikzpicture}[scale=.55,line width=1.35pt]
\foreach \i in {1,...,3}
{ \path (\i,1) coordinate (T\i); \path (\i,0) coordinate (B\i); }
\filldraw[fill=gray!25,draw=gray!25,line width=4pt]  
(T1) -- (T3) -- (B3) -- (B1) -- (T1);
\draw (T2) -- (B1);
\draw (T3) -- (B2);
\foreach \i in {1,...,3}
{ \filldraw (T\i) circle (2pt); \filldraw (B\i) circle (2pt); }
\end{tikzpicture}}, \quad 
{\begin{tikzpicture}[scale=.55,line width=1.35pt]
\foreach \i in {1,...,3}
{ \path (\i,1) coordinate (T\i); \path (\i,0) coordinate (B\i); }
\filldraw[fill=gray!25,draw=gray!25,line width=4pt]  
(T1) -- (T3) -- (B3) -- (B1) -- (T1);
\draw (T1) .. controls +(.1,-.4) and +(-.1,-.4) .. (T2);
\draw (T1) -- (B1);
\draw (T3) -- (B2);
\foreach \i in {1,...,3}
{ \filldraw (T\i) circle (2pt); \filldraw (B\i) circle (2pt); }
\end{tikzpicture}}, \quad 
{\begin{tikzpicture}[scale=.55,line width=1.35pt]
\foreach \i in {1,...,3}
{ \path (\i,1) coordinate (T\i); \path (\i,0) coordinate (B\i); }
\filldraw[fill=gray!25,draw=gray!25,line width=4pt]  
(T1) -- (T3) -- (B3) -- (B1) -- (T1);
\draw (T1) .. controls +(.1,-.4) and +(-.1,-.4) .. (T3);
\draw (T2) -- (B1);
\draw (T3) -- (B2);
\foreach \i in {1,...,3}
{ \filldraw (T\i) circle (2pt); \filldraw (B\i) circle (2pt); }
\end{tikzpicture}}, \quad 
{\begin{tikzpicture}[scale=.55,line width=1.35pt]
\foreach \i in {1,...,3}
{ \path (\i,1) coordinate (T\i); \path (\i,0) coordinate (B\i); }
\filldraw[fill=gray!25,draw=gray!25,line width=4pt]  
(T1) -- (T3) -- (B3) -- (B1) -- (T1);
\draw (T2) .. controls +(.1,-.4) and +(-.1,-.4) .. (T3);
\draw (T1) -- (B1);
\draw (T2) -- (B2);
\foreach \i in {1,...,3}
{ \filldraw (T\i) circle (2pt); \filldraw (B\i) circle (2pt); }
\end{tikzpicture}}}.
\]
\end{example}

Recall that the irreducible representations of $S_{m}$ have a basis in bijection with standard tableaux of size $m$.
For $\nu\vdash m$, the Young diagram of $\nu$ is the left-justified array of boxes,
with $\nu_i$ boxes in the $i$-th row, we use the French notation which counts rows from bottom to top.
A \emph{standard tableau} of $\nu$ is a filling of the boxes of the Young diagram of $\nu$ with numbers
$1, \ldots, m$ so that the numbers increase along rows from left to right and along columns from bottom
to top.  The \emph{shape} of a tableau is the partition obtained by the sizes of the rows.

For $0\leq m\leq k$ and $\nu\vdash m$, a basis of $\Delta_k(\nu)$ is defined by
\begin{equation}
\mathcal{B}_{k}(\nu) = \{ d\otimes T\, |\, d \text{ is a } (k,m)\text{-standard and }
T \text{ is a standard tableau of shape } \nu\}.
\end{equation}

A diagram $d\in \mathsf{P}_k(x)$ acts on a basis element $d'\otimes T$ of $\Delta_k(\nu)$ by left multiplication,

\begin{equation}\label{standard-action}
d\cdot d'\otimes T = \begin{cases} dd' \otimes T & \text{ if } p(dd') = m  \\
0 & \text{otherwise}\end{cases},
\end{equation}
in the case that $p(dd') = m$,  we use the product of diagrams and Lemma \ref{V-module}
to get $dd' = x^a d_1\tau$ where $d_1$ is a $(k,m)$-standard diagram and $\tau \in S_{m}$.
Hence, $dd' \otimes T  = x^a d_1 \otimes \tau \cdot T$, where $\tau\cdot T$ might not be
standard, but can be written as a linear combination of standard tableaux using the
Garnir straightening algorithm for Specht modules (see for instance \cite{Sagan}).

\begin{remark}
Notice that $d$ acts by zero if after multiplying $dd'=x^a d_1$, we have
(1) two propagating blocks in $d'$ become connected in $d_1$; or
(2) $\overline{j}$ for $1\leq j \leq m$ does not propagate in $d_1$.
\end{remark}


These modules coincide with the modules constructed by Martin \cite{mar1} in Definition 7.  When $\mathsf{P}_k(x)$ is semisimple, the set $\{ \Delta_k(\nu) \, |\, \nu \in \mathsf{Par}_{\leq k}\}$ is a complete set of non-isomorphic simple modules. We refer to these modules as the \emph{standard} modules of the partition algebra.

By counting the elements in $\mathcal{B}_k(\nu)$ we find the dimensions of the standard modules. 
\begin{cor}\label{cor:dimDeltaknu} (\cite{BHH} Proposition 5.10) 
For $k\in \mathbb{Z}_{\geq 0}$ and $0 \leq m \leq k$, we have that 
    \[ \text{dim}(\Delta_k(\nu)) = \sum_{i = m} ^k
    \stirling{k}{i} \binom{i}{m} f^\nu\]
where $\stirling{k}{i}$ are the Stirling numbers of the second kind,
which enumerate the number of set partitions of $[k]$ into $i$ parts and $f^\nu$ is the dimension
of $\mathbb{S}^\nu$, or the number of standard tableaux of shape $\nu$.
\end{cor}

\subsection{A tableau model for the standard modules} In this 
section we show that the basis $\mathcal{B}_{k}(\nu)$ is in bijection with set-filled tableaux.  We will set $x=n\in \mathbb{Z}_{\geq 0}$ and $n\geq 2k$.

\begin{definition}\label{def:settableaux}
For $k \in \mathbb{Z}_{\geq 0}$, $n\geq 2k$ and $0\leq i \leq k$, let $\lambda$
be a partition of $n$, a \defn{$[k]$-set valued tableau $T$ of shape $\lambda$}
satisfies the following conditions:
\begin{enumerate}
    \item The sets filling the boxes of the Young diagram of $\lambda$ form a set partition $\alpha$ of $[k]$, the sets in $\alpha$ are called blocks.
    \item Every box in rows $\lambda_2, \ldots, \lambda_\ell$ is filled with a block in $\alpha$. 
    \item Boxes at the end of the first row of $\lambda$ could contain blocks of $\alpha$ and, because of the
    condition that $n \geq 2k$, there are at least $k$ empty boxes preceding the boxes containing sets.
\end{enumerate}
\end{definition}
Let $\mathcal{T}_k(\lambda)$ denote the set of all $[k]$-set valued tableaux of shape $\lambda$.

\begin{example}\squaresize=15pt The standard $[3]$-set valued tableaux in $\mathcal{T}_3((n-2,2))$ are:
\[
\young{1&2\cr&&&\cdots&&3\cr}~,\quad
\young{1&3\cr&&&\cdots&&2\cr}~,\quad
\young{2&3\cr&&&\cdots&&1\cr}~,\]
\[\young{12&3\cr&&&\cdots&&\cr}~,\quad
\young{2&13\cr&&&\cdots&&\cr}~,\quad
\young{1&23\cr&&&\cdots&&\cr}~,
\]
where each of the tableaux listed above have a first row with $n-2$ boxes.
\end{example}

\begin{definition}\label{def:bijectionrho}
Let $k\in \mathbb{Z}_{\geq 0}$, $n\geq 2k$, and $0\leq m \leq k$, and $\nu\vdash m$.  We know define a map
$\rho:\mathcal{B}_{k}(\nu)\rightarrow \mathcal{T}_k((n-|\nu|,\nu))$ for $d\otimes T \in \mathcal{B}_{k}(\nu)$
assign a $[k]$-set valued tableau as follows:
\begin{enumerate}\label{alg:tableau}
\item For each propagating block in $d$,
replace $j$ in $T$ by the block $B_{\overline{j}}$ containing $\overline{j}$.
\item Create a set valued tableau of shape $(n-|\nu|, \nu)$
by adding a first row of length $n-|\nu|$ to the tableaux obtained in step (1).
\item Place all non-propagating blocks in the top of $d$ in last-letter order in the boxes at the end of the row $n-|\nu|$ added in step (2).
\end{enumerate}
\end{definition}

It is relatively easy to see that the map $\rho$ is invertible.
The sets in the first row of the tableau correspond to non-propagating blocks
in the $(k,m)$-standard diagram.
The propagating blocks in the $(k,m)$-standard diagram
are determined by connecting $\overline{j}$ to the $j^{th}$ largest set (by maximum value in the set)
for $1 \leq j \leq m$.
The standard tableau is determined once the $(k,m)$-standard diagram is known.  Therefore, $\rho$ is a bijection. 

\begin{example}\label{ex:settableau}
\squaresize=18pt
Correspondence between a basis element $d\otimes T\in \Delta_9((2,1))$ and a $[9]$-set valued
tableau of shape $(n-3,2,1)$.
\[
\scalebox{1.0}{{\begin{tikzpicture}[scale=.55,line width=1.35pt]
\foreach \i in {1,...,9}
{ \path (\i,1) coordinate (T\i); \path (\i,0) coordinate (B\i); }
\filldraw[fill=gray!25,draw=gray!25,line width=4pt]  
(T1) -- (T9) -- (B9) -- (B1) -- (T1);
\draw (T1) .. controls +(.3,-.4) and +(-.3,-.4) .. (T3);
\draw (T2) -- (B1);
\draw (T3) -- (B2);
\draw (T4) -- (B3);
\draw (T4) .. controls +(.1,-.4) and +(-.1,-.4) .. (T5);
\draw (T5) .. controls +(.1,-.4) and +(-1,-.4) .. (T9);
\draw (T6) .. controls +(.1,-.2) and +(-.1,-.2) .. (T7);
\foreach \i in {1,...,9}
{ \filldraw (T\i) circle (2pt); \filldraw (B\i) circle (2pt); }
\foreach \i in {1,...,9}
{\draw  (B\i)  node[below=0.05cm]{${\overline{\i}}$}; \draw  (T\i)  node[above=0.05cm]{${\i}$};}
\end{tikzpicture} }\hskip -.1in
\raisebox{-.05in}{\begin{tikzpicture}[scale=.55,line width=1.35pt]
\draw[thin] (0,0)--(2,0)--(2,1)--(1,1)--(1,2)--(0,2)--(0,0);
\draw[thin] (1,0)--(1,1);
\draw[thin]  (0,1)--(1,1);
\node(11) at (0.5,0.5){$1$};
\node(12) at (1.5,0.5){$3$};
\node(21) at (0.5,1.5){$2$};
\node(e) at (1,-.75){};
\node(f) at (-0.5,1){$\otimes$};
\node(g) at (4,1){$\stackrel{\rho}{\longmapsto}$};
\end{tikzpicture}}}\quad
\underbrace{\young{13\cr2&459\cr&&&\cdots&&67&8\cr}}_{n-3 \text{ boxes }}
\]
\end{example}

Using the notation from \cite{HJ}, we let for any $\nu\vdash m$, with $0\leq m\leq k$
\[ \mathsf{P}_k^\nu = \mathbb{C}\text{-Span}\{ N_S : S\in \mathcal{T}_k((n-|\nu|,\nu))\}~. \]

In \cite{HJ} Definition 4.7, the authors define an action of the diagrams in $\mathsf{P}_k(n)$ on $\mathsf{P}_k^\nu$.  In Section 3 of \cite{HJ}, the authors show that $\mathsf{P}_k^\nu \cong \Delta_k(\nu)$ as $\mathsf{P}_k(n)$-modules, in particular they called the $(k,m)$-standard diagrams ``noncrossing $m$-factors". We skip these details here as we are interested in tableaux to serve as indexing sets for bases to help describe the dimensions of the modules. For example, we can see that Corollary \ref{cor:dimDeltaknu} is counting the number of $[k]$-set valued tableaux of shape $(n-|\nu|,\nu)$.

\begin{remark}
      Let $k\in \mathbb{Z}_{\geq 0}$, $n\geq 2k$, and $0\leq m \leq k$, and $\nu\vdash m$.
    The map $\rho: \mathcal{B}_{k}(\nu)\rightarrow \mathcal{T}_k((n-|\nu|,\nu))$ is a bijection.  In Section \ref{sec:QPirreps}, we will use the linear extension of $\rho$ and apply it to linear combinations. 
\end{remark}

\subsection{Representations of the half-partition algebra}\label{sec:halfPkreps}
The irreducible representations of the half-algebras using tableaux was not considered in \cite{HJ}
but Halverson-Ram \cite{HR} constructed the matrix units for the half-partition algebras using Jones' basic construction.

In this section we describe the standard modules and then give a correspondence to tableaux modules.  The construction of the simple standard modules of $\mathsf{P}_{k+\frac{1}{2}}(x)$ is similar to that of $\mathsf{P}_k(x)$. In this section we assume $x\geq 2k+2$. The main difference in the construction is that we use the idempotent
$\tilde{\mathsf{p}}_{I_{k,m}}\otimes \mathbf{1}$. We define left modules for any $\nu\vdash m$ as follows
\[
\Delta_{k+\frac{1}{2}}(\nu) =
\mathsf{P}_{k+\frac{1}{2}}(x)(\tilde{\mathsf{p}}_{I_{k,m}}\otimes \mathbf{1})
\otimes_{\mathsf{P}_{{m}+\frac{1}{2}}(x)} \mathbb{S}^\nu
\]
where we have used the corresponding identifications as we did in the construction of the standard modules for $\mathsf{P}_k(x)$.

To define a linear basis for this module, we first define a vector space
$V_{\frac{1}{2}}(k+1, m)$ as the span of all diagrams with exactly $m+1$ propagating blocks,
each containing exactly one bottom element in the set
$\{\overline{1}, \ldots, \overline{m}, \overline{k+1}\}$,
and the vertices $\{\overline{m+1}, \ldots, \overline{k}\}$ are singletons.

\begin{example} Elements in $V_{\frac{1}{2}}(8, 3)$:
\begin{center}
\scalebox{1.0}{\begin{tikzpicture}[scale=.55,line width=1.35pt]
\foreach \i in {1,...,8}
{ \path (\i,1) coordinate (T\i); \path (\i,0) coordinate (B\i); }
\filldraw[fill=gray!25,draw=gray!25,line width=4pt]  (T1) -- (T8) -- (B8) -- (B1) -- (T1);
\draw (T1) .. controls +(.1,-.4) and +(-.1,-.4) .. (T3);
\draw (T4) .. controls +(.1,-.4) and +(-.1,-.4) .. (T5);
\draw (T7) .. controls +(.1,-.4) and +(-.1,-.4) .. (T8);
\draw (T4) -- (B1);
\draw (T2) -- (B3);
\draw (T1) -- (B2);
\draw (T8) -- (B8);
\foreach \i in {1,...,8}
{ \filldraw (T\i) circle (2pt); \filldraw (B\i) circle (2pt); }
\end{tikzpicture}}, \qquad
{\begin{tikzpicture}[scale=.55,line width=1.35pt]
\foreach \i in {1,...,8}
{ \path (\i,1) coordinate (T\i); \path (\i,0) coordinate (B\i); }
\filldraw[fill=gray!25,draw=gray!25,line width=4pt]  (T1) -- (T8) -- (B8) -- (B1) -- (T1);
\draw (T1) .. controls +(.1,-.4) and +(-.1,-.4) .. (T3);
\draw (T4) .. controls +(.1,-.4) and +(-.1,-.4) .. (T6);
\draw (T6) .. controls +(.1,-.4) and +(-.1,-.4) .. (T7);
\draw (T5) -- (B3);
\draw (T2) -- (B1);
\draw (T3) -- (B2);
\draw (T8) -- (B8);
\foreach \i in {1,...,8}
{ \filldraw (T\i) circle (2pt); \filldraw (B\i) circle (2pt); }
\end{tikzpicture}}~.
\end{center}
\end{example}
The \defn{half-$(k+1,m)$-standard diagrams} are defined to be the diagrams where if $B_j$ and $B_{j+1}$ are propagating blocks containing $\overline{j}$ and $\overline{j+1}$, respectively, then $\text{max}(B_j\cap [k])<\text{max}(B_{j+1}\cap [k])$.

\begin{example}
The five half-$(3,1)$-standard diagrams are:
\[
\scalebox{1}{{\begin{tikzpicture}[scale=.55,line width=1.35]
\foreach \i in {1,...,3}
{ \path (\i,1) coordinate (T\i); \path (\i,0) coordinate (B\i); }
\filldraw[fill=gray!25,draw=gray!25,line width=4pt]  
(T1) -- (T3) -- (B3) -- (B1) -- (T1);
\draw (T1) -- (B1);
\draw (T3) -- (B3);
\foreach \i in {1,...,3}
{ \filldraw (T\i) circle (2pt); \filldraw (B\i) circle (2pt); }
\end{tikzpicture}}, \quad
{\begin{tikzpicture}[scale=.55,line width=1.35pt]
\foreach \i in {1,...,3}
{ \path (\i,1) coordinate (T\i); \path (\i,0) coordinate (B\i); }
\filldraw[fill=gray!25,draw=gray!25,line width=4pt]  
(T1) -- (T3) -- (B3) -- (B1) -- (T1);
\draw (T2) -- (B1);
\draw (T3) -- (B3);
\foreach \i in {1,...,3}
{ \filldraw (T\i) circle (2pt); \filldraw (B\i) circle (2pt); }
\end{tikzpicture}}, \quad 
{\begin{tikzpicture}[scale=.55,line width=1.35pt]
\foreach \i in {1,...,3}
{ \path (\i,1) coordinate (T\i); \path (\i,0) coordinate (B\i); }
\filldraw[fill=gray!25,draw=gray!25,line width=4pt]  
(T1) -- (T3) -- (B3) -- (B1) -- (T1);
\draw (T1) .. controls +(.1,-.4) and +(-.1,-.4) .. (T2);
\draw (T1) -- (B1);
\draw (T3) -- (B3);
\foreach \i in {1,...,3}
{ \filldraw (T\i) circle (2pt); \filldraw (B\i) circle (2pt); }
\end{tikzpicture}}, \quad 
{\begin{tikzpicture}[scale=.55,line width=1.35pt]
\foreach \i in {1,...,3}
{ \path (\i,1) coordinate (T\i); \path (\i,0) coordinate (B\i); }
\filldraw[fill=gray!25,draw=gray!25,line width=4pt]  
(T1) -- (T3) -- (B3) -- (B1) -- (T1);
\draw (T1) .. controls +(.1,-.4) and +(-.1,-.4) .. (T3);
\draw (T2) -- (B1);
\draw (T3) -- (B3);
\foreach \i in {1,...,3}
{ \filldraw (T\i) circle (2pt); \filldraw (B\i) circle (2pt); }
\end{tikzpicture}}, \quad 
{\begin{tikzpicture}[scale=.55,line width=1.35pt]
\foreach \i in {1,...,3}
{ \path (\i,1) coordinate (T\i); \path (\i,0) coordinate (B\i); }
\filldraw[fill=gray!25,draw=gray!25,line width=4pt]  
(T1) -- (T3) -- (B3) -- (B1) -- (T1);
\draw (T2) .. controls +(.1,-.4) and +(-.1,-.4) .. (T3);
\draw (T1) -- (B1);
\draw (T3) -- (B3);
\foreach \i in {1,...,3}
{ \filldraw (T\i) circle (2pt); \filldraw (B\i) circle (2pt); }
\end{tikzpicture}}}.
\] 
\end{example}

Similarly as in Lemma \ref{V-module}, we have that every $d\in V_{\frac{1}{2}}(k+1,m)$ can be written as
\[d = d'\sigma, \] 
where $\sigma \in S_{m}$ and $d'$ is a half-$(k+1, m)$-standard diagram.
With these definitions, let $\nu\vdash m$, then as vector spaces, we have
\[ \Delta_{k+\frac{1}{2}}(\nu)\cong V_{\frac{1}{2}}(k+1,m) \otimes_{S_{m}} \mathbb{S}^\nu.\]
A linear basis for this module is 
$\mathcal{B}_{k+\frac{1}{2}} (\nu)$  consisting of elements $d\otimes T$ where $d$
is a half-$(k+1, m)$-standard diagram and  $T$ is a standard tableau of shape $\nu$.
The action of $d\in \mathsf{P}_{k+\frac{1}{2}}(x)$ is 
\begin{equation}\label{half-standard-action}
d\cdot d'\otimes T = \begin{cases} dd' \otimes T & \text{ if } p(dd') = m+1  \\
0 & \text{otherwise}\end{cases}, 
\end{equation}
Here, if $p(dd') = m+1$, then $dd'\in V_{\frac{1}{2}}(k+1, m)$ and it can be written as $dd' = d_1\tau$, where $d_1$ is half-$(k+1,m)$-standard diagram and $\tau\in S_{m}$. In this case, $dd' \otimes T = d_1 \otimes \tau \cdot T$, and as before we apply Garnir straightening algorithm to write $\tau\cdot T$ as a linear combination of standard tableaux, if needed.

\begin{example} The bases for the four representations of $\mathsf{P}_{2+\frac{1}{2}}(x)$ are the following sets:
\squaresize=14pt
\[
\left\{
\raisebox{-0.1in}{
{\begin{tikzpicture}[scale=.55,line width=1.35]
\foreach \i in {1,...,3}
{ \path (\i,1) coordinate (T\i); \path (\i,0) coordinate (B\i); }
\filldraw[fill=gray!25,draw=gray!25,line width=4pt]  
(T1) -- (T3) -- (B3) -- (B1) -- (T1);
\draw (T3) -- (B3);
\foreach \i in {1,...,3}
{ \filldraw (T\i) circle (2pt); \filldraw (B\i) circle (2pt); }
\end{tikzpicture}
\raisebox{0.09in}{$\otimes\ \emptyset~,$}}
\quad
{\begin{tikzpicture}[scale=.55,line width=1.35pt]
\foreach \i in {1,...,3}
{ \path (\i,1) coordinate (T\i); \path (\i,0) coordinate (B\i);}
\filldraw[fill=gray!25,draw=gray!25,line width=4pt]  
(T1) -- (T3) -- (B3) -- (B1) -- (T1);
\draw (T3) -- (B3);
\draw (T2) .. controls +(.1,-.4) and +(-.1,-.4) .. (T3);
\foreach \i in {1,...,3}
{ \filldraw (T\i) circle (2pt); \filldraw (B\i) circle (2pt); }
\end{tikzpicture}
\raisebox{0.09in}{$\otimes\ \emptyset~,$}}
\quad
{\begin{tikzpicture}[scale=.55,line width=1.35pt]
\foreach \i in {1,...,3}
{ \path (\i,1) coordinate (T\i); \path (\i,0) coordinate (B\i); }
\filldraw[fill=gray!25,draw=gray!25,line width=4pt]  
(T1) -- (T3) -- (B3) -- (B1) -- (T1);
\draw (T1) .. controls +(.1,-.4) and +(-.1,-.4) .. (T2);
\draw (T3) -- (B3);
\foreach \i in {1,...,3}
{ \filldraw (T\i) circle (2pt); \filldraw (B\i) circle (2pt); }
\end{tikzpicture}
\raisebox{0.09in}{$\otimes\ \emptyset~,$}}
\quad
{\begin{tikzpicture}[scale=.55,line width=1.35pt]
\foreach \i in {1,...,3}
{ \path (\i,1) coordinate (T\i); \path (\i,0) coordinate (B\i); }
\filldraw[fill=gray!25,draw=gray!25,line width=4pt]  
(T1) -- (T3) -- (B3) -- (B1) -- (T1);
\draw (T1) .. controls +(.1,-.4) and +(-.1,-.4) .. (T3);
\draw (T3) -- (B3);
\foreach \i in {1,...,3}
{ \filldraw (T\i) circle (2pt); \filldraw (B\i) circle (2pt); }
\end{tikzpicture}
\raisebox{0.09in}{$\otimes\ \emptyset~,$}}
\quad
{\begin{tikzpicture}[scale=.55,line width=1.35pt]
\foreach \i in {1,...,3}
{ \path (\i,1) coordinate (T\i); \path (\i,0) coordinate (B\i); }
\filldraw[fill=gray!25,draw=gray!25,line width=4pt]  
(T1) -- (T3) -- (B3) -- (B1) -- (T1);
\draw (T2) .. controls +(.1,-.4) and +(-.1,-.4) .. (T3);
\draw (T1) .. controls +(.1,-.4) and +(-.1,-.4) .. (T2);
\draw (T3) -- (B3);
\foreach \i in {1,...,3}
{ \filldraw (T\i) circle (2pt); \filldraw (B\i) circle (2pt); }
\end{tikzpicture}
\raisebox{0.09in}{$\otimes\ \emptyset$}}}
\right\}
\] 
\[
\left\{
\raisebox{-0.1in}{
{\begin{tikzpicture}[scale=.55,line width=1.35]
\foreach \i in {1,...,3}
{ \path (\i,1) coordinate (T\i); \path (\i,0) coordinate (B\i); }
\filldraw[fill=gray!25,draw=gray!25,line width=4pt]  
(T1) -- (T3) -- (B3) -- (B1) -- (T1);
\draw (T1) -- (B1);
\draw (T3) -- (B3);
\foreach \i in {1,...,3}
{ \filldraw (T\i) circle (2pt); \filldraw (B\i) circle (2pt); }
\end{tikzpicture}
\raisebox{0.09in}{$\otimes$}
\raisebox{0.03in}{\young{1\cr}~,}}
\ 
{\begin{tikzpicture}[scale=.55,line width=1.35pt]
\foreach \i in {1,...,3}
{ \path (\i,1) coordinate (T\i); \path (\i,0) coordinate (B\i); }
\filldraw[fill=gray!25,draw=gray!25,line width=4pt]  
(T1) -- (T3) -- (B3) -- (B1) -- (T1);
\draw (T2) -- (B1);
\draw (T3) -- (B3);
\foreach \i in {1,...,3}
{ \filldraw (T\i) circle (2pt); \filldraw (B\i) circle (2pt); }
\end{tikzpicture}
\raisebox{0.09in}{$\otimes$}
\raisebox{0.03in}{\young{1\cr}~,}}
\ 
{\begin{tikzpicture}[scale=.55,line width=1.35pt]
\foreach \i in {1,...,3}
{ \path (\i,1) coordinate (T\i); \path (\i,0) coordinate (B\i); }
\filldraw[fill=gray!25,draw=gray!25,line width=4pt]  
(T1) -- (T3) -- (B3) -- (B1) -- (T1);
\draw (T1) .. controls +(.1,-.4) and +(-.1,-.4) .. (T2);
\draw (T1) -- (B1);
\draw (T3) -- (B3);
\foreach \i in {1,...,3}
{ \filldraw (T\i) circle (2pt); \filldraw (B\i) circle (2pt); }
\end{tikzpicture}
\raisebox{0.09in}{$\otimes$}
\raisebox{0.03in}{\young{1\cr}~,}}
\ 
{\begin{tikzpicture}[scale=.55,line width=1.35pt]
\foreach \i in {1,...,3}
{ \path (\i,1) coordinate (T\i); \path (\i,0) coordinate (B\i); }
\filldraw[fill=gray!25,draw=gray!25,line width=4pt]  
(T1) -- (T3) -- (B3) -- (B1) -- (T1);
\draw (T1) .. controls +(.1,-.4) and +(-.1,-.4) .. (T3);
\draw (T2) -- (B1);
\draw (T3) -- (B3);
\foreach \i in {1,...,3}
{ \filldraw (T\i) circle (2pt); \filldraw (B\i) circle (2pt); }
\end{tikzpicture}
\raisebox{0.09in}{$\otimes$}
\raisebox{0.03in}{\young{1\cr}~,}}
\ 
{\begin{tikzpicture}[scale=.55,line width=1.35pt]
\foreach \i in {1,...,3}
{ \path (\i,1) coordinate (T\i); \path (\i,0) coordinate (B\i); }
\filldraw[fill=gray!25,draw=gray!25,line width=4pt]  
(T1) -- (T3) -- (B3) -- (B1) -- (T1);
\draw (T2) .. controls +(.1,-.4) and +(-.1,-.4) .. (T3);
\draw (T1) -- (B1);
\draw (T3) -- (B3);
\foreach \i in {1,...,3}
{ \filldraw (T\i) circle (2pt); \filldraw (B\i) circle (2pt); }
\end{tikzpicture}
\raisebox{0.09in}{$\otimes$}
\raisebox{0.03in}{\young{1\cr}}}}
\right\}
\]
\[
\left\{
\raisebox{-.1in}{
\begin{tikzpicture}[scale=.55,line width=1.35pt]
\foreach \i in {1,...,3}
{ \path (\i,1) coordinate (T\i); \path (\i,0) coordinate (B\i); }
\filldraw[fill=gray!25,draw=gray!25,line width=4pt]  
(T1) -- (T3) -- (B3) -- (B1) -- (T1);
\draw (T1) -- (B1);
\draw (T2) -- (B2);
\draw (T3) -- (B3);
\foreach \i in {1,...,3}
{ \filldraw (T\i) circle (2pt); \filldraw (B\i) circle (2pt); }
\end{tikzpicture}
\raisebox{0.09in}{$~~\otimes$}
\raisebox{0.03in}{\young{1&2\cr}}}
\right\}
\]
\[
\left\{
\raisebox{-.1in}{
\begin{tikzpicture}[scale=.55,line width=1.35pt]
\foreach \i in {1,...,3}
{ \path (\i,1) coordinate (T\i); \path (\i,0) coordinate (B\i); }
\filldraw[fill=gray!25,draw=gray!25,line width=4pt]  
(T1) -- (T3) -- (B3) -- (B1) -- (T1);
\draw (T1) -- (B1);
\draw (T2) -- (B2);
\draw (T3) -- (B3);
\foreach \i in {1,...,3}
{ \filldraw (T\i) circle (2pt); \filldraw (B\i) circle (2pt); }
\end{tikzpicture}
\raisebox{0.09in}{$~~\otimes$}
\raisebox{-0.06in}{\young{2\cr1\cr}}~}
\right\}~.
\]
Corresponding, respectively to 
$\Delta_{2+\frac{1}{2}}(\emptyset)$, $\Delta_{2+\frac{1}{2}}((1))$,
$\Delta_{2+\frac{1}{2}}((2))$, and $\Delta_{2+\frac{1}{2}}((1,1))$.
\end{example}

The elements in $\mathcal{B}_{k+\frac{1}{2}}(\nu)$ are in bijection with pairs of set valued tableaux of shape $((n-1-|\nu|, \nu), (1))$,  where $d\otimes T \in \mathcal{B}_{\frac{1}{2}}(\nu)$ corresponds to the pair constructed as follows: 
\begin{enumerate}
\item The propagating blocks of $d$ containing $\overline{1}, \ldots, \overline{m}$ are inserted in the boxes of $T$ where $j$ is replaced by $B_{\overline{j}}\cap [k]$.
\item The nonpropagating blocks in $d$ are inserted in last letter order at the end of the first row. 
\item $B_{\overline{k+1}}\cap [k]$ is inserted in the one box in the second component. 
\end{enumerate}

\begin{example}\squaresize=18pt
Correspondence between a basis element $d\otimes T\in \Delta_{8+\frac{1}{2}}((2,1))$ and a  pair of set valued
tableaux of shape $((n-3,1,1),(1))$.
\[
\scalebox{1.0}{{\begin{tikzpicture}[scale=.55,line width=1.35pt]
\foreach \i in {1,...,9}
{ \path (\i,1) coordinate (T\i); \path (\i,0) coordinate (B\i); }
\filldraw[fill=gray!25,draw=gray!25,line width=4pt]  
(T1) -- (T9) -- (B9) -- (B1) -- (T1);
\draw (T1) .. controls +(.3,-.4) and +(-.3,-.4) .. (T3);
\draw (T2) -- (B1);
\draw (T3) -- (B2);
\draw (T9) -- (B9);
\draw (T4) .. controls +(.1,-.4) and +(-.1,-.4) .. (T5);
\draw (T5) .. controls +(.1,-.4) and +(-1,-.4) .. (T9);
\draw (T6) .. controls +(.1,-.2) and +(-.1,-.2) .. (T7);
\foreach \i in {1,...,9}
{ \filldraw (T\i) circle (2pt); \filldraw (B\i) circle (2pt); }
\foreach \i in {1,...,9}
{\draw  (B\i)  node[below=0.05cm]{${\overline{\i}}$}; \draw  (T\i)  node[above=0.05cm]{${\i}$};}
\end{tikzpicture} }\hskip -.1in
\raisebox{-.05in}{\begin{tikzpicture}[scale=.55,line width=1.35pt]
\draw[thin] (0,0)--(1,0)--(1,2)--(0,2)--(0,0);
\draw[thin] (1,0)--(1,1);
\draw[thin]  (0,1)--(1,1);
\node(11) at (0.5,0.5){$1$};
\node(21) at (0.5,1.5){$2$};
\node(e) at (1,-.75){};
\node(f) at (-0.5,1){$\otimes$};
\node(g) at (4,1){$\stackrel{\rho}{\longmapsto}$};
\end{tikzpicture}}}\quad
\underbrace{\young{13\cr2\cr&\cdots&&67&8\cr}}_{n-3 \text{ boxes }}
\quad, \quad\young{459\cr}
\]
\end{example}

\begin{remark}
The Bratteli diagram of the tower
$\mathsf{P}_0(x) \hookrightarrow \mathsf{P}_1(x) \hookrightarrow \mathsf{P}_2(x) \hookrightarrow \cdots$
has in general multiple edges between the vertices since tensoring representations of the
symmetric group by the permutation module is not multiplicity free.
The introduction of the half algebras helps decompose the tensor product of the permutation
module into two operations, induction and restriction.
The Bratteli diagram for the tower
$\mathsf{P}_0(x) \hookrightarrow \mathsf{P}_{\frac{1}{2}}(x)
\subset \mathsf{P}_1(x) \hookrightarrow \mathsf{P}_{1+\frac{1}{2}}(x) \subset \cdots$
does not have multiple edges between the vertices.
The Bratteli diagram for partition algebras interlaced
with half-partition algebras is in \cite{HR}, page 893. 
\end{remark}

\section{Quasi-partition algebras and half quasi-partition algebras}\label{sec:quasipalgebra}
For $k \in \mathbb{Z}_{k > 0}$, the quasi-partition algebra $\mathsf{QP}_k(n)$ was introduced in \cite{DO}
as the centralizer algebra
$\text{End}_{S_n}((\mathbb{S}^{(n-1,1)})^{\otimes k})$,
where $\mathbb{S}^{(n-1,1)}$ is the irreducible representation of the symmetric group, $S_n$.
In this section, we give a more general definition and introduce
the half quasi-partition algebras, $\mathsf{QP}_{k+\frac{1}{2}}(x)$.

We define $\pi := \mathbf{1} - \frac{1}{x}\mathsf{p}$ and using tensor notation we define an idempotent $\pi^{\otimes k}$ in $\mathsf{P}_k(x)$ as follows
\begin{equation}\label{eq:projector}
\piok := \left(\mathbf{1}^{\otimes k} - \frac{1}{x}\mathsf{p}_1\right)
\left(\mathbf{1}^{\otimes k} - \frac{1}{x}\mathsf{p}_2\right)\cdots
\left(\mathbf{1}^{\otimes k} - \frac{1}{x}\mathsf{p}_k\right).
\end{equation}
Let $k \in \mathbb{Z}_{\geq0}$ and let $J$ be any subset of
$[k] = \{1,2,\ldots, k \}$,
we set $\mathsf{p}_\emptyset := \mathbf{1}^{\otimes k}$
and $\mathsf{p}_J: = \prod_{j\in J} \mathsf{p}_j$.  Using this notation we have
\begin{equation}\label{eq:exprpir}
\piok = \sum_{J\subseteq [k]} \frac{1}{(-x)^{|J|}} \mathsf{p}_J~.
\end{equation}

The corresponding idempotent in 
$\mathsf{P}_{k+\frac{1}{2}}(x) \subseteq \mathsf{P}_{k+1}(x)$ will 
be denoted by $\piokk := \piok \otimes \bf{1}$ to indicate that it is contained in the larger algebra. 

\begin{example} The idempotent in $\mathsf{P}_3(x)$ is
$\pi^{\otimes 3} = \mathbf{1}^{\otimes 3} - \frac{1}{x} \mathsf{p}_1 - \frac{1}{x} \mathsf{p}_2 - \frac{1}{x} \mathsf{p}_3 + \frac{1}{x^2} \mathsf{p}_1 \mathsf{p}_2 +\frac{1}{x^2} \mathsf{p}_1 \mathsf{p}_3+\frac{1}{x^2} \mathsf{p}_2 \mathsf{p}_3 - \frac{1}{x^3} \mathsf{p}_1 \mathsf{p}_2 \mathsf{p}_3$.
This element expressed using diagrams is

\begin{align*}
\pi^{\otimes 3} = &
\raisebox{-.1in}{
\begin{tikzpicture}[scale=.55, line width=1.35]
	\filldraw[fill=gray!25,draw=gray!25,line width=4pt]  (0,0) -- (0,1) -- (2,1) -- (2,0) -- (0,0);
	 \foreach \x in {0,1,2} {
	 \filldraw [black] (\x, 0) circle (2pt);
	 \filldraw [black] (\x, 1) circle (2pt);
	 }
	\draw  (0,1)--(0,0) (1,1)--(1,0) (2,1)--(2,0) ;
\end{tikzpicture}}
-\frac{1}{x}
\raisebox{-.1in}{
\begin{tikzpicture}[scale=.55, line width=1.35]
	\filldraw[fill=gray!25,draw=gray!25,line width=4pt]  (0,0) -- (0,1) -- (2,1) -- (2,0) -- (0,0);
	 \foreach \x in {0,1,2} {
	 \filldraw [black] (\x, 0) circle (2pt);
	 \filldraw [black] (\x, 1) circle (2pt);
	 }
	\draw (1,1)--(1,0) (2,1)--(2,0) ;
\end{tikzpicture}}
-\frac{1}{x}
\raisebox{-.1in}{
\begin{tikzpicture}[scale=.55, line width=1.35]
	\filldraw[fill=gray!25,draw=gray!25,line width=4pt]  (0,0) -- (0,1) -- (2,1) -- (2,0) -- (0,0);
	 \foreach \x in {0,1,2} {
	 \filldraw [black] (\x, 0) circle (2pt);
	 \filldraw [black] (\x, 1) circle (2pt);
	 }
	\draw (0,1)--(0,0) (2,1)--(2,0) ;
\end{tikzpicture}}
-\frac{1}{x}
\raisebox{-.1in}{
\begin{tikzpicture}[scale=.55, line width=1.35]
	\filldraw[fill=gray!25,draw=gray!25,line width=4pt]  (0,0) -- (0,1) -- (2,1) -- (2,0) -- (0,0);
	 \foreach \x in {0,1,2} {
	 \filldraw [black] (\x, 0) circle (2pt);
	 \filldraw [black] (\x, 1) circle (2pt);
	 }
	\draw (1,1)--(1,0) (0,1)--(0,0) ;
\end{tikzpicture}}
+\frac{1}{x^2}
\raisebox{-.1in}{
\begin{tikzpicture}[scale=.55, line width=1.35]
	\filldraw[fill=gray!25,draw=gray!25,line width=4pt]  (0,0) -- (0,1) -- (2,1) -- (2,0) -- (0,0);
	 \foreach \x in {0,1,2} {
	 \filldraw [black] (\x, 0) circle (2pt);
	 \filldraw [black] (\x, 1) circle (2pt);
	 }
	\draw (0,1)--(0,0) ;
\end{tikzpicture}}
+\frac{1}{x^2}
\raisebox{-.1in}{
\begin{tikzpicture}[scale=.55, line width=1.35]
	\filldraw[fill=gray!25,draw=gray!25,line width=4pt]  (0,0) -- (0,1) -- (2,1) -- (2,0) -- (0,0);
	 \foreach \x in {0,1,2} {
	 \filldraw [black] (\x, 0) circle (2pt);
	 \filldraw [black] (\x, 1) circle (2pt);
	 }
	\draw (1,1)--(1,0) ;
\end{tikzpicture}}\\
&+\frac{1}{x^2}
\raisebox{-.1in}{
\begin{tikzpicture}[scale=.55, line width=1.35]
	\filldraw[fill=gray!25,draw=gray!25,line width=4pt]  (0,0) -- (0,1) -- (2,1) -- (2,0) -- (0,0);
	 \foreach \x in {0,1,2} {
	 \filldraw [black] (\x, 0) circle (2pt);
	 \filldraw [black] (\x, 1) circle (2pt);
	 }
	\draw (2,1)--(2,0) ;
\end{tikzpicture}}
-\frac{1}{x^3}
\raisebox{-.1in}{
\begin{tikzpicture}[scale=.55, line width=1.35]
	\filldraw[fill=gray!25,draw=gray!25,line width=4pt]  (0,0) -- (0,1) -- (2,1) -- (2,0) -- (0,0);
	 \foreach \x in {0,1,2} {
	 \filldraw [black] (\x, 0) circle (2pt);
	 \filldraw [black] (\x, 1) circle (2pt);
	 }
\end{tikzpicture}}~.
\end{align*}

The idempotent in $\mathsf{P}_{2 + \frac{1}{2}}(x)$ is
$\pi^{\otimes 2}_3 = \mathbf{1}^{\otimes 3} - \frac{1}{x} \mathsf{p}_1 - \frac{1}{x} \mathsf{p}_2 + \frac{1}{x^2} \mathsf{p}_1\mathsf{p}_2$ and this expression in terms of diagrams is
\begin{align*}
\pi^{\otimes 2}_{3}
&=
\raisebox{-.1in}{
\begin{tikzpicture}[scale=.55, line width=1.35]
	\filldraw[fill=gray!25,draw=gray!25,line width=4pt]  (0,0) -- (0,1) -- (2,1) -- (2,0) -- (0,0);
	 \foreach \x in {0,1,2} {
	 \filldraw [black] (\x, 0) circle (2pt);
	 \filldraw [black] (\x, 1) circle (2pt);
	 }
	\draw  (0,1)--(0,0) (1,1)--(1,0) (2,1)--(2,0) ;
\end{tikzpicture}}
- \frac{1}{x}
\raisebox{-.1in}{
\begin{tikzpicture}[scale=.55, line width=1.35]
	\filldraw[fill=gray!25,draw=gray!25,line width=4pt]  (0,0) -- (0,1) -- (2,1) -- (2,0) -- (0,0);
	 \foreach \x in {0,1,2} {
	 \filldraw [black] (\x, 0) circle (2pt);
	 \filldraw [black] (\x, 1) circle (2pt);
	 }
	\draw  (1,1)--(1,0) (2,1)--(2,0) ;
\end{tikzpicture}}
- \frac{1}{x}
\raisebox{-.1in}{
\begin{tikzpicture}[scale=.55, line width=1.35]
	\filldraw[fill=gray!25,draw=gray!25,line width=4pt]  (0,0) -- (0,1) -- (2,1) -- (2,0) -- (0,0);
	 \foreach \x in {0,1,2} {
	 \filldraw [black] (\x, 0) circle (2pt);
	 \filldraw [black] (\x, 1) circle (2pt);
	 }
	\draw  (0,1)--(0,0) (2,1)--(2,0) ;
\end{tikzpicture}}
+\frac{1}{x^2}
\raisebox{-.1in}{
\begin{tikzpicture}[scale=.55, line width=1.35]
	\filldraw[fill=gray!25,draw=gray!25,line width=4pt]  (0,0) -- (0,1) -- (2,1) -- (2,0) -- (0,0);
	 \foreach \x in {0,1,2} {
	 \filldraw [black] (\x, 0) circle (2pt);
	 \filldraw [black] (\x, 1) circle (2pt);
	 }
	\draw (2,1)--(2,0) ;
\end{tikzpicture}}~.
\end{align*}
\end{example}

For $k \in {\mathbb Z}_{\geq 0}$ and any diagram $d \in \Ph_k$, we define
\[ \overline{d} = \piok d \piok.\]
And similarly, for $d \in \Ph_{k+\frac{1}{2}}$, we define
$\overline{d} = \piokk d \piokk$.

\begin{lemma}\label{lem:odzero}
For $r \in \frac{1}{2}{\mathbb Z}_{\geq 0}$ if
$d\in \Ph_r$ is a diagram with one or more singletons, then $\overline{d} = 0$.
\end{lemma}
\begin{proof}
Observe that $\mathsf{p}_i\piok = \piok\mathsf{p}_i =
(\mathbf{1}^{\otimes k} - \frac{1}{x}\mathsf{p}_1)\cdots
(\mathsf{p}_i -\mathsf{p}_i)\cdots
(\mathbf{1}^{\otimes k} - \frac{1}{x}\mathsf{p}_{k}) = 0$,
since the $\mathsf{p}_i$s commute and $\mathsf{p}_i^2 = x \mathsf{p}_i$.
If $d$ has a singleton $\{i\}$ on the top row,
then $d = \frac{1}{x} \mathsf{p}_i d$ and if $d$ has a singleton $\{\overline{i}\}$ in the bottom row,
then $d = \frac{1}{x} d\mathsf{p}_i $.
Hence $\overline{d} =0$ if $d$ has any singletons, since $\Pa_r(x)$ is associative.

The same proof also applies to $\piokk$ since it is equal to $\piok \otimes \mathbf{1}$
and $k+1$, $\o{k+1}$ are not isolated in $\mathsf{P}_{k+\frac{1}{2}}(x)$.
\end{proof}

In the following proposition, we give an expression for the elements $\overline{d}$, where $d$ is a diagram without singletons.  For this we need some notation and definitions. 
For $r \in \frac{1}{2} \mathbb{Z}_{\geq 0}$ and $d,d' \in \Ph_r$, we say that $d'$ is
\defn{finer} than $d$ (alternatively, $d$ is \defn{coarser} than $d'$)
and use the notation $d'\leq d$ if every block of $d'$ is contained in a block of $d$.
Furthermore for a block $B \in d$,
we set $sp_{d'}(B) = \{ B \cap B' : B' \in d' \}$, the intersections of a single block of $d$ with all the blocks of $d'$. We write $d' \leq_\ast d$ if $d' \leq d$ and
$sp_{d'}(B)$ contains at most one set of size greater or equal to two for each $B \in d$. 
For every block $B$ of $d$, we define the \defn{length} of $B$ in $d'$, denoted by $\ell_{d'}(B)$,
to be the number of blocks of size $1$ in $sp_{d'}(B)$.  We note that if $d'\leq d$, then
$\ell_{d'}(B) = |B|$ if and only if all elements $i \in B$ are isolated vertices
in $d'$. 

\begin{example}\label{Ex : lenght blocks}
Let $d,d',d'' \in \Ph_{2+\frac{1}{2}}$ be the diagrams:
\[
d=
\raisebox{-.35in}{\begin{tikzpicture}[scale=.55,line width=1.35pt]
\foreach \i in {1,...,3}
{ \path (\i,1) coordinate (T\i); \path (\i,0) coordinate (B\i); }
\filldraw[fill=gray!25,draw=gray!25,line width=4pt]  (T1) -- (T3) -- (B3) -- (B1) -- (T1);
\draw (T1) -- (B1);
\draw (T1) .. controls +(.1,-.4) and +(-.1,-.4) .. (T2);
\draw (T3)--(B3);
\draw (B2) .. controls +(.1,.4) and +(-.1,.4) .. (B3);
\foreach \i in {1,...,3}
{\draw  (B\i)  node[below=0.05cm]{${\overline{\i}}$}; \draw  (T\i)  node[above=0.05cm]{${\i}$};}
\foreach \i in {1,...,3}
{ \filldraw (T\i) circle (2pt); \filldraw (B\i) circle (2pt); }
\end{tikzpicture}},
\quad\quad
d'=
\raisebox{-.35in}{\begin{tikzpicture}[scale=.55,line width=1.35pt]
\foreach \i in {1,...,3}
{ \path (\i,1) coordinate (T\i); \path (\i,0) coordinate (B\i); }
\filldraw[fill=gray!25,draw=gray!25,line width=4pt]  (T1) -- (T3) -- (B3) -- (B1) -- (T1);
\draw (T3)--(B3);
\foreach \i in {1,...,3}
{\draw  (B\i)  node[below=0.05cm]{${\overline{\i}}$}; \draw  (T\i)  node[above=0.05cm]{${\i}$};}
\foreach \i in {1,...,3}
{ \filldraw (T\i) circle (2pt); \filldraw (B\i) circle (2pt); }
\end{tikzpicture}}
\quad\hbox{ and }\quad
d''=
\raisebox{-.35in}{\begin{tikzpicture}[scale=.55,line width=1.35pt]
\foreach \i in {1,...,3}
{ \path (\i,1) coordinate (T\i); \path (\i,0) coordinate (B\i); }
\filldraw[fill=gray!25,draw=gray!25,line width=4pt]  (T1) -- (T3) -- (B3) -- (B1) -- (T1);
\draw (B1)--(T2);
\draw (T3)--(B3);
\draw (B2) .. controls +(.1,.4) and +(-.1,.4) .. (B3);
\foreach \i in {1,...,3}
{\draw  (B\i)  node[below=0.05cm]{${\overline{\i}}$}; \draw  (T\i)  node[above=0.05cm]{${\i}$};}
\foreach \i in {1,...,3}
{ \filldraw (T\i) circle (2pt); \filldraw (B\i) circle (2pt); }
\end{tikzpicture}}~.
\]
We note that
$d'\leq_\ast d$ and $d'' \leq_\ast d$.  Let $B_1 = \{1,2,\o{1}\}$ and $B_2 = \{3,\o{2},\o3\}$ so that
$d = \{ B_1, B_2 \}$, then
$\ell_{d'}(B_1)=|B_1|=3$, $\ell_{d'}(B_2)=1$, $\ell_{d''}(B_1)=1$ and $\ell_{d''}(B_2)=0$.
\end{example}

\begin{prop}\label{Prop : coeff d bar}
Let $r \in \frac{1}{2} \mathbb{Z}_{\geq0}$ and
let $d \in \Ph_r$ be a diagram without singletons, then
\begin{equation*}
\overline{d} = d + \sum_{d': d'<_\ast d} a_{d,d'}(x) d' ~, 
\end{equation*}
where
\begin{equation*}
a_{d,d'}(x) =
\prod_{B \in d}
\begin{cases}
\frac{|B|-1}{(-x)^{\ell_{d'}(B)-1}}&\hbox{ if }\ell_{d'}(B) = |B|\\
\frac{1}{(-x)^{\ell_{d'}(B)}}&\hbox{ otherwise}
\end{cases}~.
\end{equation*}
\end{prop}

\begin{proof}
Let $d \in \Ph_r$ be a diagram without singletons, then by Equation \eqref{eq:exprpir},
$\overline{d}$ is a sum of terms involving diagrams
in the expression
$$\frac{1}{(-x)^{|J|+|J'|}} \mathsf{p}_J d \mathsf{p}_{J'}$$
for some subsets $J, J' \subseteq [k]$.
Let $d'$ be the diagram that occurs when we multiply $\mathsf{p}_J d \mathsf{p}_{J'}$. Multiplying $d$ on the top by $\mathsf{p}_J$ isolates all vertices in $J$ and multiplying on the bottom by $\mathsf{p}_{J'}$ isolated all vertices in $J'$ in the bottom of $d$. The resulting $d'$ is finer than $d$ and results from isolating some vertices of blocks in $d$. Hence,  for each block $B \in d$, all elements in $B$ that also are either in $J$ or
in $J'$ are singletons in $d'$.  The only set in $sp_{d'}(B)$ that may contain more than one element is $B \backslash (J \cup J')$.
Therefore by definition, $d' \leq_\ast d$.

Now fix a $d'$ such that $d' <_\ast d$.
Assume that $d'$ is the diagram of the term in the expression $\mathsf{p}_J d \mathsf{p}_{J'}$.
For each block $B$, we let $J_B = J \cap B$ (and respectively $J'_B = J' \cap B$).
We have therefore
$\mathsf{p}_J =\left(\prod_{B \in d} \mathsf{p}_{J_B}\right)$ (with a similar expression for $\mathsf{p}_{J'}$).
In determining the coefficient of $d'$ in $\o{d}$, we will calculate the contribution
to the coefficient grouped by possible pairs of sets $(J_B, J'_B)$.

For each block of $B \in d$, either $sp_{d'}(B)$ contains
a block of size greater than one or it contains only singletons.
If $sp_{d'}(B)$ contains
a block of size greater than one, then there are unique $J_B$ and $J'_B$
that will appear in every term where $d'$ is the diagram in the expression $\mathsf{p}_J d \mathsf{p}_{J'}$.
Those terms will have a contribution of $\frac{1}{(-x)^{|J_B|+|J_{B'}|}} = \frac{1}{(-x)^{\ell_{d'}(B)}}$
from the part corresponding to this block.

If $sp_{d'}(B)$ contains only singletons, then there are $|B|+1$ pairs of sets $(J_B, J_B')$
such that $J_B \subseteq B \cap [k]$ and $J' \subseteq B \cap [\o{k}]$ where
$d'$ is the diagram in the expression $\mathsf{p}_J d \mathsf{p}_{J'}$.
Those pairs of sets are: $((B \cap [k]) \backslash \{a\}, B \cap [\o{k}])$
for $a \in B\cap [k]$,
$(B \cap [k] , (B \cap [\o{k}])\backslash \{a'\})$
for $a' \in B \cap [\o{k}]$, and $(B \cap [k],B \cap [\o{k}])$.
In the first two cases, the contribution to the coefficient will be
$\frac{1}{(-x)^{|J_B|+|J_{B'}|}} = \frac{1}{(-x)^{\ell_{d'}(B)-1}}$.
In the latter case $\frac{1}{(-x)^{|J_B|+|J_{B'}|}} = \frac{1}{(-x)^{\ell_{d'}(B)}}$, however
$\mathsf{p}_J d \mathsf{p}_{J'}$ will contribute a power of $x$ due
to the diagram product that will have a disconnected component in the middle row
and hence the contribution from that term will be $\frac{x}{(-x)^{\ell_{d'}(B)}}
= \frac{-1}{(-x)^{\ell_{d'}(B)-1}}$.
The sum over all of the $|B|+1$ possible pairs $(J_B, J_B')$ will contribute
a total of $\frac{|B|-1}{(-x)^{\ell_{d'}(B)-1}}$ to the coefficient
of $d'$ in $\o{d}$.

Because the pairs $(J_B, J'_B)$ can be chosen for each block $B \in d$, we
conclude that the coefficient of $d'$ in $\o{d}$ is equal to
the expression stated in the proposition.
\end{proof}

\begin{example} For $d = \{\{1,2,\o1\},\{3,\o2,\o3\}\}$ a diagram in $\mathsf{P}_{2+\frac{1}{2}}(x)$,
we compute using Proposition \ref{Prop : coeff d bar}:
\begin{align*}
\o{d} =
\raisebox{-.1in}{
\begin{tikzpicture}[scale=.55, line width=1.35]
	\filldraw[fill=gray!25,draw=gray!25,line width=4pt]  (0,0) -- (0,1) -- (2,1) -- (2,0) -- (0,0);
	 \foreach \x in {0,1,2} {
	 \filldraw [black] (\x, 0) circle (2pt);
	 \filldraw [black] (\x, 1) circle (2pt);
	 }
	\draw (0,0)--(0,1);
	\draw (0,1) .. controls +(.1,-.4) and +(-.1,-.4) .. (1,1);
	\draw (2,0)--(2,1);
	\draw (1,0) .. controls +(.1,.4) and +(-.1,.4) .. (2,0);
\end{tikzpicture}}
&-\frac{1}{x}
\raisebox{-.1in}{
\begin{tikzpicture}[scale=.55, line width=1.35]
	\filldraw[fill=gray!25,draw=gray!25,line width=4pt]  (0,0) -- (0,1) -- (2,1) -- (2,0) -- (0,0);
	 \foreach \x in {0,1,2} {
	 \filldraw [black] (\x, 0) circle (2pt);
	 \filldraw [black] (\x, 1) circle (2pt);
	 }
	\draw (0,1) .. controls +(.1,-.4) and +(-.1,-.4) .. (1,1);
	\draw (2,0)--(2,1);
	\draw (1,0) .. controls +(.1,.4) and +(-.1,.4) .. (2,0);
\end{tikzpicture}}
-\frac{1}{x}
\raisebox{-.1in}{
\begin{tikzpicture}[scale=.55, line width=1.35]
	\filldraw[fill=gray!25,draw=gray!25,line width=4pt]  (0,0) -- (0,1) -- (2,1) -- (2,0) -- (0,0);
	 \foreach \x in {0,1,2} {
	 \filldraw [black] (\x, 0) circle (2pt);
	 \filldraw [black] (\x, 1) circle (2pt);
	 }
	\draw  (0,0)--(0,1);
	\draw (2,0)--(2,1);
	\draw (1,0) .. controls +(.1,.4) and +(-.1,.4) .. (2,0);
\end{tikzpicture}}
-\frac{1}{x}
\raisebox{-.1in}{
\begin{tikzpicture}[scale=.55, line width=1.35]
	\filldraw[fill=gray!25,draw=gray!25,line width=4pt]  (0,0) -- (0,1) -- (2,1) -- (2,0) -- (0,0);
	 \foreach \x in {0,1,2} {
	 \filldraw [black] (\x, 0) circle (2pt);
	 \filldraw [black] (\x, 1) circle (2pt);
	 }
	\draw  (1,1)--(0,0);
	\draw (2,0)--(2,1);
	\draw (1,0) .. controls +(.1,.4) and +(-.1,.4) .. (2,0);
\end{tikzpicture}}
- \frac{1}{x}
\raisebox{-.1in}{
\begin{tikzpicture}[scale=.55, line width=1.35]
	\filldraw[fill=gray!25,draw=gray!25,line width=4pt]  (0,0) -- (0,1) -- (2,1) -- (2,0) -- (0,0);
	 \foreach \x in {0,1,2} {
	 \filldraw [black] (\x, 0) circle (2pt);
	 \filldraw [black] (\x, 1) circle (2pt);
	 }
	\draw (0,0)--(0,1);
	\draw (0,1) .. controls +(.1,-.4) and +(-.1,-.4) .. (1,1);
	\draw (2,0)--(2,1);
\end{tikzpicture}}
+ \frac{1}{x^2}
\raisebox{-.1in}{
\begin{tikzpicture}[scale=.55, line width=1.35]
	\filldraw[fill=gray!25,draw=gray!25,line width=4pt]  (0,0) -- (0,1) -- (2,1) -- (2,0) -- (0,0);
	 \foreach \x in {0,1,2} {
	 \filldraw [black] (\x, 0) circle (2pt);
	 \filldraw [black] (\x, 1) circle (2pt);
	 }
	\draw (2,0)--(2,1);
	\draw (0,1) .. controls +(.1,-.4) and +(-.1,-.4) .. (1,1);
\end{tikzpicture}}\\
&+ \frac{1}{x^2}
\raisebox{-.1in}{
\begin{tikzpicture}[scale=.55, line width=1.35]
	\filldraw[fill=gray!25,draw=gray!25,line width=4pt]  (0,0) -- (0,1) -- (2,1) -- (2,0) -- (0,0);
	 \foreach \x in {0,1,2} {
	 \filldraw [black] (\x, 0) circle (2pt);
	 \filldraw [black] (\x, 1) circle (2pt);
	 }
	\draw  (0,0)--(0,1) (2,0)--(2,1);
\end{tikzpicture}}
+ \frac{1}{x^2}
\raisebox{-.1in}{
\begin{tikzpicture}[scale=.55, line width=1.35]
	\filldraw[fill=gray!25,draw=gray!25,line width=4pt]  (0,0) -- (0,1) -- (2,1) -- (2,0) -- (0,0);
	 \foreach \x in {0,1,2} {
	 \filldraw [black] (\x, 0) circle (2pt);
	 \filldraw [black] (\x, 1) circle (2pt);
	 }
	\draw  (1,1)--(0,0) (2,0)--(2,1);
\end{tikzpicture}}
+ \frac{2}{x^2}
\raisebox{-.1in}{
\begin{tikzpicture}[scale=.55, line width=1.35]
	\filldraw[fill=gray!25,draw=gray!25,line width=4pt]  (0,0) -- (0,1) -- (2,1) -- (2,0) -- (0,0);
	 \foreach \x in {0,1,2} {
	 \filldraw [black] (\x, 0) circle (2pt);
	 \filldraw [black] (\x, 1) circle (2pt);
	 }
	\draw (2,0)--(2,1);
	\draw (1,0) .. controls +(.1,.4) and +(-.1,.4) .. (2,0);
\end{tikzpicture}}
- \frac{2}{x^3}
\raisebox{-.1in}{
\begin{tikzpicture}[scale=.55, line width=1.35]
	\filldraw[fill=gray!25,draw=gray!25,line width=4pt]  (0,0) -- (0,1) -- (2,1) -- (2,0) -- (0,0);
	 \foreach \x in {0,1,2} {
	 \filldraw [black] (\x, 0) circle (2pt);
	 \filldraw [black] (\x, 1) circle (2pt);
	 }
	\draw (2,0)--(2,1);
\end{tikzpicture}}~.
\end{align*}
Referring to Example \ref{Ex : lenght blocks} the diagram corresponding to $d'$ has coefficients $\frac{2}{(-x)^2}\frac{1}{(-x)^1}$ and the coefficient of the diagram $d''$ is $\frac{1}{(-x)^1}\frac{1}{(-x)^0}$.
\end{example}

\subsection{The product in quasi-partition algebras} \label{sec:product}
For $r \in \frac{1}{2} {\mathbb Z}_{\geq 0}$, we set
\[\Dh_{r} = \{ d : d \in \Ph_{r}\hbox{ without singletons}\},\] and  define
\begin{equation}\label{eq:basisQP}
\mathsf{QP}_r(x) = \mathbb{C}(x)\text{-Span}\{ \overline{d}\, |\, d \in \Dh_r\}.
\end{equation}
If $r$ is an integer, we call $\mathsf{QP}_r(x)$ the \defn{quasi-partition algebra} and if $r$ is half an integer, the \defn{half quasi-partition algebra}.  

We now show that $\mathsf{QP}_{r}(x)$
for $r \in {\frac{1}{2}} \mathbb{Z}_{\geq 0}$ is closed under the product. 


\begin{prop}
    Let $d_1, d_2 \in \Dh_k$ and suppose $d_1d_2 = x^a d_3$ in $\mathsf{P}_k(x)$, where $a\in \mathbb{Z}_{\geq 0}$. Then, we have either $\o{d_1}\ \o{d_2} =0$ or 
    \[
    \overline{d_1}\ \overline{d_2} = \sum_{d: d\leq d_3} c_{d_1,d_2}^d(x) \overline{d},\qquad c_{d_1,d_2}^d(x)\in \mathbb{C}(x), 
    \]
where $d\leq d'$ means that $d$ is finer than or equal to $d'$, and $d\in \Dh_r$.    
\end{prop}

\begin{proof}
    Suppose that $r=k\in \mathbb{Z}_{\geq 0}$. 
    The product $\o{d_1} \ \o{d_2} = \pi^{\otimes k} (d_1 \pi^{\otimes k} d_2) \pi^{\otimes k} 
    = \o{d_1\pi^{\otimes k} d_2}$, the first equality follows because $\pi^{\otimes k}$ is an idempotent and the last expression simply means that we multiply every element in the linear combination $d_1\pi^{\otimes k} d_2$ on the right and on the left by $\pi^{\otimes k}$.

     We have $d_1 \pi^{\otimes k} d_2 = \sum_{J\subseteq [k]} \frac{1}{(-x)^{|J|}}d_1\mathsf{p}_J d_2$.  Each summand $d_1\mathsf{p}_Jd_2$ can be computed by drawing a three layered graph with $d_1$ on top, followed by $\mathsf{p}_J$, and $d_2$ in the bottom.  For any $J\neq \emptyset$, we get that the three layered graph always has fewer edges than the graph of $d_1\mathsf{p}_\emptyset d_2 = d_1 d_2$ (since $\mathsf{p}_\emptyset = \mathbf{1}^{\otimes k}$, the identity in $\mathsf{P}_k(x)$). Hence after removing middle blocks we get $d_1\mathsf{p}_J d_2 = x^m d'$, where $d'$ is finer or equal to the diagram that occurs in the product $d_1d_2$, i.e. $d_3$. After combining like terms, we have  
    \[d_1\pi^{\otimes k} d_2 = \sum_{d'\leq d_3} \gamma_{d_1,d_2}^{d'}(x) d'.\]

     Now multiplying this last expression on the left and on the right by $\pi^{\otimes k}$ we get 
    \[\o{d_1} \ \o{d_2} = \o{d_1\pi^{\otimes k} d_2} = \sum_{d'\leq d_3} \gamma_{d_1,d_2}^{d'}(x) \o{d'}.\]
    It is possible for the last expression to be equal to zero in the case that all terms cancel out or when all resulting diagrams have singletons.  If a term $\gamma_{d_1,d_2}^{d'}(x)\o{d'}$ is nonzero, the diagram $d'$ has no singletons by Lemma \ref{lem:odzero}.

    For $r = k +\frac{1}{2}$ the above argument also works when we use the idempotent $\pi_{k+1}^{\otimes k}$ to show that we get finer elements in the summand; however, we have to show that we get $\o{d'}$ where $d'$ has $k+1$ and $\o{k+1}$ in the same block.  But this should be clear as in this case $d_1, d_2, \pi_{k+1}^{\otimes k}, \in \mathsf{P}_{k+\frac{1}{2}}(x)$ and  $\mathsf{P}_{k+\frac{1}{2}}(x)$ is closed under the product, therefore every resulting summand satisfies this condition.
\end{proof}

\begin{example} Consider the following example in $\mathsf{QP}_{2+\frac{1}{2}}(x)$.  Let
$$\mathsf{t}_1 = \raisebox{-.1in}{
\begin{tikzpicture}[scale=.55, line width=1.35]
	\filldraw[fill=gray!25,draw=gray!25,line width=4pt]  (0,0) -- (0,1) -- (2,1) -- (2,0) -- (0,0);
	 \foreach \x in {0,1,2} {
	 \filldraw [black] (\x, 0) circle (2pt);
	 \filldraw [black] (\x, 1) circle (2pt);
	 }
	\draw  (0,0)--(0,1) (2,0)--(2,1);
	\draw (0,1) .. controls +(.1,-.4) and +(-.1,-.4) .. (1,1);
	\draw (1,0) .. controls +(.1,.4) and +(-.1,.4) .. (2,0);
\end{tikzpicture}},
$$
then $\o{\mathsf{e}_1}~\o{\mathsf{t}_1}=0$, while
$$
\mathsf{t}_1~\mathsf{e}_1 = 
\raisebox{-.1in}{
\begin{tikzpicture}[scale=.55, line width=1.35]
	\filldraw[fill=gray!25,draw=gray!25,line width=4pt]  (0,0) -- (0,1) -- (2,1) -- (2,0) -- (0,0);
	 \foreach \x in {0,1,2} {
	 \filldraw [black] (\x, 0) circle (2pt);
	 \filldraw [black] (\x, 1) circle (2pt);
	 }
	\draw (2,1)--(2,0);
	\draw (0,1) .. controls +(.1,-.4) and +(-.1,-.4) .. (1,1);
	\draw (1,1) .. controls +(.1,-.4) and +(-.1,-.4) .. (2,1);
	\draw (0,0) .. controls +(.1,.4) and +(-.1,.4) .. (1,0);
\end{tikzpicture}}
\qquad\hbox{and}\qquad
\o{\mathsf{t}_1}~\o{\mathsf{e}_1} = 
\overline{\raisebox{-.1in}{
\begin{tikzpicture}[scale=.55, line width=1.35]
	\filldraw[fill=gray!25,draw=gray!25,line width=4pt]  (0,0) -- (0,1) -- (2,1) -- (2,0) -- (0,0);
	 \foreach \x in {0,1,2} {
	 \filldraw [black] (\x, 0) circle (2pt);
	 \filldraw [black] (\x, 1) circle (2pt);
	 }
	\draw (2,1)--(2,0);
	\draw (0,1) .. controls +(.1,-.4) and +(-.1,-.4) .. (1,1);
	\draw (1,1) .. controls +(.1,-.4) and +(-.1,-.4) .. (2,1);
	\draw (0,0) .. controls +(.1,.4) and +(-.1,.4) .. (1,0);
\end{tikzpicture}}}
- \frac{1}{x} \overline{\raisebox{-.1in}{
\begin{tikzpicture}[scale=.55, line width=1.35]
	\filldraw[fill=gray!25,draw=gray!25,line width=4pt]  (0,0) -- (0,1) -- (2,1) -- (2,0) -- (0,0);
	 \foreach \x in {0,1,2} {
	 \filldraw [black] (\x, 0) circle (2pt);
	 \filldraw [black] (\x, 1) circle (2pt);
	 }
	\draw (2,0)--(2,1);
	\draw (0,1) .. controls +(.1,-.4) and +(-.1,-.4) .. (1,1);
	\draw (0,0) .. controls +(.1,.4) and +(-.1,.4) .. (1,0);
\end{tikzpicture}}}~.$$
\end{example}

\medskip
We note that $\mathsf{QP}_k(x)$ is an algebra with unity  $\pi^{\otimes k}$ and $\mathsf{QP}_{k+\frac{1}{2}}(x)$ with unity $\pi^{\otimes k}_{k+1}$.

For any permutation $\tau$ of $[k]$, and any $J\subseteq [k]$, let $\tau(J) = \{ \tau(i)\, |\, i\in J\}$.
Any permutation $\sigma\in S_k$ is identified with a diagram in $\mathsf{P}_k(x)$ corresponding to the
set partition with blocks $\{ i, \o{\sigma(i)}\}$.  For any $\sigma \in S_k\subseteq \mathsf{P}_k(x)$
a permutation and $J\subseteq [k]$, we have 
\[ \mathsf{p}_J \sigma = \sigma \mathsf{p}_{\sigma(J)}
\quad \text{ and } \quad \piok = \sum_{J\subseteq [k]} \frac{1}{(-x)^{|\sigma(J)|}} \mathsf{p}_{\sigma(J)}.\]
From these two identities, it follows that
\begin{equation}\label{piksigma}
 \pi^{\otimes k} \sigma = \sigma \pi^{\otimes k}
 \end{equation}
Hence, $\o{\sigma} = \pi^{\otimes k} \sigma \pi^{\otimes k} = \pi^{\otimes k} \sigma$ as an element in $\mathsf{QP}_k(x)$.

\begin{prop}\label{prop:symmetricgroup}
    The elements $\o{\mathsf{s}_i}$ for $1\leq i \leq k-1$ satisfy the following relations: 
    \[ \o{\mathsf{s}_i}^2 = \pi^{\otimes k}, \quad \o{\mathsf{s}_i}\ \o{\mathsf{s}_{i+1}}\  \o{\mathsf{s}_i} = \o{\mathsf{s}_{i+1}}\ \o{\mathsf{s}_i} \ \o{\mathsf{s}_{i+1}}, \quad \text{and } \quad\o{\mathsf{s}_i} \ \o{\mathsf{s}_j}= \o{\mathsf{s}_j}\ \o{\mathsf{s}_i}~,\]
    where $|i-j|>1$. Hence, the $\o{\mathsf{s}_i}$ generate a group isomorphic to $S_k$ with identity $\pi^{\otimes k}$ in $\mathsf{QP}_k(x)$.
\end{prop}

\begin{proof}
    By Equation \eqref{piksigma}, if $\sigma, \tau \in S_k\subset \mathsf{P}_k(x)$, then $\o{\sigma}\ \o{\tau} = \pi^{\otimes k}  \sigma\tau = \o{\sigma \tau}$. From this and the relations of the $s_i$'s in $\mathsf{P}_k(x)$. 
\end{proof}

\subsection{A tower of quasi-partition algebras}\label{subsec:quasitower}


As in the case of the half-partition algebra, for $k \in \mathbb{Z}_{\geq 0}$
there is a natural embedding of $\mathsf{QP}_k(x)$ into $\mathsf{QP}_{k+\frac{1}{2}}(x)$
by the embedding $\overline{d} \mapsto \overline{d} \otimes \mathbf{1}$.
However for $k$ an integer it is not the case that $\mathsf{QP}_{k+\frac{1}{2}}(x)$ is a subalgebra
of $\mathsf{QP}_{k+1}(x)$ (see Example \ref{ex:elementsQP}).
Instead we describe a two-step process which consists of
an inclusion followed by a projection.

For integers $k \geq 0$, and $d \in \Pa_{k+1}(x)$, we define\footnote{Note the difference between
the definition of $\tilde{d}$ and $\o{d} = \pi^{\otimes k+1} d \pi^{\otimes k+1}$.}

$$\tilde{d} = \piokk d \piokk~.$$

Now consider the subalgebra of $\Pa_{k+1}(x)$,
$$\widetilde{\mathsf{QP}}_{k+1}(x) = \mathbb{C}(x)\text{-Span}\{ \tilde{d} \, |\, d\in \Pa_{k+1}(x)\}~.$$
We note that the basis of $\widetilde{\mathsf{QP}}_{k+1}(x)$ is:
\begin{equation}\label{eq:basistildeQP}
\{ \widetilde{d} : d \in \Ph_{k+1}\hbox{ has no singletons in } [k] \cup [\o{k}] \}~.
\end{equation}
The index set are the diagrams which have no singletons in the first $k$ positions
but that may have singletons in the last position.  Equivalently it could be expressed
as neither $\{i\}$ nor $\{\o{i}\}$ is a block for $1 \leq i \leq k$.

Hence, the first step is the natural inclusion of $\mathsf{QP}_{k+\frac{1}{2}}(x)$ in
$\widetilde{\mathsf{QP}}_{k+1}(x)$.  It should be made clear that
$\widetilde{\mathsf{QP}}_{k+1}(x)$ is larger than
both $\mathsf{QP}_{k+\frac{1}{2}}(x)$ and $\mathsf{QP}_{k+1}(x)$ since, for instance,
$\piokk \mathsf{p}_{k+1} \in \widetilde{\mathsf{QP}}_{k+1}(x)$
but it is not an element of either $\mathsf{QP}_{k+\frac{1}{2}}(x)$ or $\mathsf{QP}_{k+1}(x)$.

Thus far we have introduced algebras in our tower so that for each $k \in {\mathbb{Z}}_{\geq0}$,
\begin{align}
\mathsf{QP}_k(x) &= \piok \mathsf{P}_k(x) \piok ~,\nonumber\\
\mathsf{QP}_{k+\frac{1}{2}}(x) &= \piokk \mathsf{P}_{k+\frac{1}{2}}(x) \piokk~,
\label{eq:threealgebras}\\
\widetilde{\mathsf{QP}}_{k+1}(x) &= \piokk \mathsf{P}_{k+1}(x) \piokk~.\nonumber
\end{align}

Then for the second step there is a projection from $\widetilde{\mathsf{QP}}_{k+1}(x)$
to $\mathsf{QP}_{k+1}(x)$ which, for each $d \in \Ph_{k+1}$,
$\tilde{d} \in \widetilde{\mathsf{QP}}_{k+1}(x)$
is sent to $\overline{d} =
(\mathbf{1}^{\otimes k+1} - \frac{1}{x} \mathsf{p}_{k+1})
\tilde{d}
(\mathbf{1}^{\otimes k+1} - \frac{1}{x} \mathsf{p}_{k+1})
\in \mathsf{QP}_{k+1}(x)$.

Therefore we have the following chain of inclusions and projections:
\begin{equation}\label{eq:QPtower}
\mathsf{QP}_0(x) \hookrightarrow \mathsf{QP}_{\frac{1}{2}}(x) \subseteq \widetilde{\mathsf{QP}}_{1}(x)\twoheadrightarrow
\mathsf{QP}_1(x) \hookrightarrow \mathsf{QP}_{1+\frac{1}{2}}(x) \subseteq \widetilde{\mathsf{QP}}_{2}(x)\twoheadrightarrow
\mathsf{QP}_2(x) \hookrightarrow  \cdots~.
\end{equation}

While there is an embedding of $\mathsf{QP}_{k}(x)$ in $\mathsf{QP}_{k+\frac{1}{2}}(x)$
by tensoring by $\mathbf{1}$, the
following example makes it clearer that $\mathsf{QP}_{k+\frac{1}{2}}(x)$
is not a subalgebra of $\mathsf{QP}_{k+1}(x)$.

\begin{example} \label{ex:elementsQP}
Notice that $\mathsf{QP}_{1+\frac{1}{2}}(x)$ is an algebra of dimension $2$ and is spanned by
the following two elements:

\begin{align*}
\raisebox{-.1in}{
\begin{tikzpicture}[scale=.55, line width=1.35]
	\filldraw[fill=gray!25,draw=gray!25,line width=4pt]  (0,0) -- (0,1) -- (1,1) -- (1,0) -- (0,0);
	 \foreach \x in {0,1} {
	 \filldraw [black] (\x, 0) circle (2pt);
	 \filldraw [black] (\x, 1) circle (2pt);
	 }
	\draw  (0,1)--(0,0) (1,1)--(1,0);
\end{tikzpicture}}
- \frac{1}{x}
\raisebox{-.1in}{
\begin{tikzpicture}[scale=.55, line width=1.35]
	\filldraw[fill=gray!25,draw=gray!25,line width=4pt]  (0,0) -- (0,1) -- (1,1) -- (1,0) -- (0,0);
	 \foreach \x in {0,1} {
	 \filldraw [black] (\x, 0) circle (2pt);
	 \filldraw [black] (\x, 1) circle (2pt);
	 }
	\draw  (1,1)--(1,0);
\end{tikzpicture}}~,\qquad\qquad
\raisebox{-.1in}{
\begin{tikzpicture}[scale=.55, line width=1.35]
	\filldraw[fill=gray!25,draw=gray!25,line width=4pt]  (0,0) -- (0,1) -- (1,1) -- (1,0) -- (0,0);
	 \foreach \x in {0,1} {
	 \filldraw [black] (\x, 0) circle (2pt);
	 \filldraw [black] (\x, 1) circle (2pt);
	 }
	\draw  (0,0)--(0,1) (1,0)--(1,1) ;
	\draw (0,0) .. controls +(.1,.4) and +(-.1,.4) .. (1,0);
	\draw (0,1) .. controls +(.1,-.4) and +(-.1,-.4) .. (1,1);
\end{tikzpicture}}
&-\frac{1}{x}
\raisebox{-.1in}{
\begin{tikzpicture}[scale=.55, line width=1.35]
	\filldraw[fill=gray!25,draw=gray!25,line width=4pt]  (0,0) -- (0,1) -- (1,1) -- (1,0) -- (0,0);
	 \foreach \x in {0,1} {
	 \filldraw [black] (\x, 0) circle (2pt);
	 \filldraw [black] (\x, 1) circle (2pt);
	 }
	\draw (1,0)--(1,1);
	\draw (0,1) .. controls +(.1,-.4) and +(-.1,-.4) .. (1,1);
\end{tikzpicture}}
-\frac{1}{x}
\raisebox{-.1in}{
\begin{tikzpicture}[scale=.55, line width=1.35]
	\filldraw[fill=gray!25,draw=gray!25,line width=4pt]  (0,0) -- (0,1) -- (1,1) -- (1,0) -- (0,0);
	 \foreach \x in {0,1} {
	 \filldraw [black] (\x, 0) circle (2pt);
	 \filldraw [black] (\x, 1) circle (2pt);
	 }
	\draw  (1,0)--(1,1);
	\draw (0,0) .. controls +(.1,.4) and +(-.1,.4) .. (1,0);
\end{tikzpicture}}
+\frac{1}{x^2}
\raisebox{-.1in}{
\begin{tikzpicture}[scale=.55, line width=1.35]
	\filldraw[fill=gray!25,draw=gray!25,line width=4pt]  (0,0) -- (0,1) -- (1,1) -- (1,0) -- (0,0);
	 \foreach \x in {0,1} {
	 \filldraw [black] (\x, 0) circle (2pt);
	 \filldraw [black] (\x, 1) circle (2pt);
	 }
	\draw  (1,0)--(1,1);
\end{tikzpicture}}~.
\end{align*}

A basis for the algebra $\widetilde{\mathsf{QP}}_2(x)$ contains seven elements.
Two of those elements are the same as those that are in $\mathsf{QP}_{1+\frac{1}{2}}(x)$,
then in addition there are the following five elements:

\begin{align*}
&\raisebox{-.1in}{
\begin{tikzpicture}[scale=.55, line width=1.35]
	\filldraw[fill=gray!25,draw=gray!25,line width=4pt]  (0,0) -- (0,1) -- (1,1) -- (1,0) -- (0,0);
	 \foreach \x in {0,1} {
	 \filldraw [black] (\x, 0) circle (2pt);
	 \filldraw [black] (\x, 1) circle (2pt);
	 }
	\draw  (0,0)--(0,1);
	\draw (0,1) .. controls +(.1,-.4) and +(-.1,-.4) .. (1,1);
\end{tikzpicture}}
-\frac{1}{x}
\raisebox{-.1in}{
\begin{tikzpicture}[scale=.55, line width=1.35]
	\filldraw[fill=gray!25,draw=gray!25,line width=4pt]  (0,0) -- (0,1) -- (1,1) -- (1,0) -- (0,0);
	 \foreach \x in {0,1} {
	 \filldraw [black] (\x, 0) circle (2pt);
	 \filldraw [black] (\x, 1) circle (2pt);
	 }
	\draw (0,0)--(1,1) ;
\end{tikzpicture}}
-\frac{1}{x}
\raisebox{-.1in}{
\begin{tikzpicture}[scale=.55, line width=1.35]
	\filldraw[fill=gray!25,draw=gray!25,line width=4pt]  (0,0) -- (0,1) -- (1,1) -- (1,0) -- (0,0);
	 \foreach \x in {0,1} {
	 \filldraw [black] (\x, 0) circle (2pt);
	 \filldraw [black] (\x, 1) circle (2pt);
	 }
	\draw (0,1) .. controls +(.1,-.4) and +(-.1,-.4) .. (1,1);
\end{tikzpicture}}
+\frac{1}{x^2}
\raisebox{-.1in}{
\begin{tikzpicture}[scale=.55, line width=1.35]
	\filldraw[fill=gray!25,draw=gray!25,line width=4pt]  (0,0) -- (0,1) -- (1,1) -- (1,0) -- (0,0);
	 \foreach \x in {0,1} {
	 \filldraw [black] (\x, 0) circle (2pt);
	 \filldraw [black] (\x, 1) circle (2pt);
	 }
\end{tikzpicture}}~,
\quad
&\raisebox{-.1in}{
\begin{tikzpicture}[scale=.55, line width=1.35]
	\filldraw[fill=gray!25,draw=gray!25,line width=4pt]  (0,0) -- (0,1) -- (1,1) -- (1,0) -- (0,0);
	 \foreach \x in {0,1} {
	 \filldraw [black] (\x, 0) circle (2pt);
	 \filldraw [black] (\x, 1) circle (2pt);
	 }
	\draw  (0,0)--(0,1);
	\draw (0,0) .. controls +(.1,.4) and +(-.1,.4) .. (1,0);
\end{tikzpicture}}
-\frac{1}{x}
\raisebox{-.1in}{
\begin{tikzpicture}[scale=.55, line width=1.35]
	\filldraw[fill=gray!25,draw=gray!25,line width=4pt]  (0,0) -- (0,1) -- (1,1) -- (1,0) -- (0,0);
	 \foreach \x in {0,1} {
	 \filldraw [black] (\x, 0) circle (2pt);
	 \filldraw [black] (\x, 1) circle (2pt);
	 }
	\draw (1,0)--(0,1) ;
\end{tikzpicture}}
-\frac{1}{x}
\raisebox{-.1in}{
\begin{tikzpicture}[scale=.55, line width=1.35]
	\filldraw[fill=gray!25,draw=gray!25,line width=4pt]  (0,0) -- (0,1) -- (1,1) -- (1,0) -- (0,0);
	 \foreach \x in {0,1} {
	 \filldraw [black] (\x, 0) circle (2pt);
	 \filldraw [black] (\x, 1) circle (2pt);
	 }
	\draw (0,0) .. controls +(.1,.4) and +(-.1,.4) .. (1,0);
\end{tikzpicture}}
+\frac{1}{x^2}
\raisebox{-.1in}{
\begin{tikzpicture}[scale=.55, line width=1.35]
	\filldraw[fill=gray!25,draw=gray!25,line width=4pt]  (0,0) -- (0,1) -- (1,1) -- (1,0) -- (0,0);
	 \foreach \x in {0,1} {
	 \filldraw [black] (\x, 0) circle (2pt);
	 \filldraw [black] (\x, 1) circle (2pt);
	 }
\end{tikzpicture}}~,\\
&\raisebox{-.1in}{
\begin{tikzpicture}[scale=.55, line width=1.35]
	\filldraw[fill=gray!25,draw=gray!25,line width=4pt]  (0,0) -- (0,1) -- (1,1) -- (1,0) -- (0,0);
	 \foreach \x in {0,1} {
	 \filldraw [black] (\x, 0) circle (2pt);
	 \filldraw [black] (\x, 1) circle (2pt);
	 }
	\draw (0,0) .. controls +(.1,.4) and +(-.1,.4) .. (1,0);
	\draw (0,1) .. controls +(.1,-.4) and +(-.1,-.4) .. (1,1);
\end{tikzpicture}}
-\frac{1}{x}
\raisebox{-.1in}{
\begin{tikzpicture}[scale=.55, line width=1.35]
	\filldraw[fill=gray!25,draw=gray!25,line width=4pt]  (0,0) -- (0,1) -- (1,1) -- (1,0) -- (0,0);
	 \foreach \x in {0,1} {
	 \filldraw [black] (\x, 0) circle (2pt);
	 \filldraw [black] (\x, 1) circle (2pt);
	 }
	\draw (0,0) .. controls +(.1,.4) and +(-.1,.4) .. (1,0);
\end{tikzpicture}}
-\frac{1}{x}
\raisebox{-.1in}{
\begin{tikzpicture}[scale=.55, line width=1.35]
	\filldraw[fill=gray!25,draw=gray!25,line width=4pt]  (0,0) -- (0,1) -- (1,1) -- (1,0) -- (0,0);
	 \foreach \x in {0,1} {
	 \filldraw [black] (\x, 0) circle (2pt);
	 \filldraw [black] (\x, 1) circle (2pt);
	 }
	\draw (0,1) .. controls +(.1,-.4) and +(-.1,-.4) .. (1,1);
\end{tikzpicture}}
+\frac{1}{x^2}
\raisebox{-.1in}{
\begin{tikzpicture}[scale=.55, line width=1.35]
	\filldraw[fill=gray!25,draw=gray!25,line width=4pt]  (0,0) -- (0,1) -- (1,1) -- (1,0) -- (0,0);
	 \foreach \x in {0,1} {
	 \filldraw [black] (\x, 0) circle (2pt);
	 \filldraw [black] (\x, 1) circle (2pt);
	 }
\end{tikzpicture}}~,
&\raisebox{-.1in}{
\begin{tikzpicture}[scale=.55, line width=1.35]
	\filldraw[fill=gray!25,draw=gray!25,line width=4pt]  (0,0) -- (0,1) -- (1,1) -- (1,0) -- (0,0);
	 \foreach \x in {0,1} {
	 \filldraw [black] (\x, 0) circle (2pt);
	 \filldraw [black] (\x, 1) circle (2pt);
	 }
	\draw  (0,0)--(1,1)  (1,0)--(0,1) ;
\end{tikzpicture}}
-\frac{1}{x}
\raisebox{-.1in}{
\begin{tikzpicture}[scale=.55, line width=1.35]
	\filldraw[fill=gray!25,draw=gray!25,line width=4pt]  (0,0) -- (0,1) -- (1,1) -- (1,0) -- (0,0);
	 \foreach \x in {0,1} {
	 \filldraw [black] (\x, 0) circle (2pt);
	 \filldraw [black] (\x, 1) circle (2pt);
	 }
	\draw (0,0)--(1,1) ;
\end{tikzpicture}}
-\frac{1}{x}
\raisebox{-.1in}{
\begin{tikzpicture}[scale=.55, line width=1.35]
	\filldraw[fill=gray!25,draw=gray!25,line width=4pt]  (0,0) -- (0,1) -- (1,1) -- (1,0) -- (0,0);
	 \foreach \x in {0,1} {
	 \filldraw [black] (\x, 0) circle (2pt);
	 \filldraw [black] (\x, 1) circle (2pt);
	 }
	\draw  (1,0)--(0,1) ;
\end{tikzpicture}}
+\frac{1}{x^2}
\raisebox{-.1in}{
\begin{tikzpicture}[scale=.55, line width=1.35]
	\filldraw[fill=gray!25,draw=gray!25,line width=4pt]  (0,0) -- (0,1) -- (1,1) -- (1,0) -- (0,0);
	 \foreach \x in {0,1} {
	 \filldraw [black] (\x, 0) circle (2pt);
	 \filldraw [black] (\x, 1) circle (2pt);
	 }
\end{tikzpicture}}~,
\end{align*}

\begin{equation*}
\raisebox{-.1in}{
\begin{tikzpicture}[scale=.55, line width=1.35]
	\filldraw[fill=gray!25,draw=gray!25,line width=4pt]  (0,0) -- (0,1) -- (1,1) -- (1,0) -- (0,0);
	 \foreach \x in {0,1} {
	 \filldraw [black] (\x, 0) circle (2pt);
	 \filldraw [black] (\x, 1) circle (2pt);
	 }
	\draw  (0,1)--(0,0);
\end{tikzpicture}}
- \frac{1}{x}
\raisebox{-.1in}{
\begin{tikzpicture}[scale=.55, line width=1.35]
	\filldraw[fill=gray!25,draw=gray!25,line width=4pt]  (0,0) -- (0,1) -- (1,1) -- (1,0) -- (0,0);
	 \foreach \x in {0,1} {
	 \filldraw [black] (\x, 0) circle (2pt);
	 \filldraw [black] (\x, 1) circle (2pt);
	 }
\end{tikzpicture}}~.
\end{equation*}

The algebra
$\mathsf{QP}_2(x)$ is of dimension $4$
and the following elements are the $\overline{d}\in \mathsf{QP}_2(x)$:

\begin{align*}
\raisebox{-.1in}{
\begin{tikzpicture}[scale=.55, line width=1.35]
	\filldraw[fill=gray!25,draw=gray!25,line width=4pt]  (0,0) -- (0,1) -- (1,1) -- (1,0) -- (0,0);
	 \foreach \x in {0,1} {
	 \filldraw [black] (\x, 0) circle (2pt);
	 \filldraw [black] (\x, 1) circle (2pt);
	 }
	\draw  (0,0)--(0,1) (1,0)--(1,1) ;
	\draw (0,0) .. controls +(.1,.4) and +(-.1,.4) .. (1,0);
	\draw (0,1) .. controls +(.1,-.4) and +(-.1,-.4) .. (1,1);
\end{tikzpicture}}
&-\frac{1}{x}
\raisebox{-.1in}{
\begin{tikzpicture}[scale=.55, line width=1.35]
	\filldraw[fill=gray!25,draw=gray!25,line width=4pt]  (0,0) -- (0,1) -- (1,1) -- (1,0) -- (0,0);
	 \foreach \x in {0,1} {
	 \filldraw [black] (\x, 0) circle (2pt);
	 \filldraw [black] (\x, 1) circle (2pt);
	 }
	\draw  (0,0)--(0,1);
	\draw (0,0) .. controls +(.1,.4) and +(-.1,.4) .. (1,0);
\end{tikzpicture}}
-\frac{1}{x}
\raisebox{-.1in}{
\begin{tikzpicture}[scale=.55, line width=1.35]
	\filldraw[fill=gray!25,draw=gray!25,line width=4pt]  (0,0) -- (0,1) -- (1,1) -- (1,0) -- (0,0);
	 \foreach \x in {0,1} {
	 \filldraw [black] (\x, 0) circle (2pt);
	 \filldraw [black] (\x, 1) circle (2pt);
	 }
	\draw  (0,0)--(0,1);
	\draw (0,1) .. controls +(.1,-.4) and +(-.1,-.4) .. (1,1);
\end{tikzpicture}}
-\frac{1}{x}
\raisebox{-.1in}{
\begin{tikzpicture}[scale=.55, line width=1.35]
	\filldraw[fill=gray!25,draw=gray!25,line width=4pt]  (0,0) -- (0,1) -- (1,1) -- (1,0) -- (0,0);
	 \foreach \x in {0,1} {
	 \filldraw [black] (\x, 0) circle (2pt);
	 \filldraw [black] (\x, 1) circle (2pt);
	 }
	\draw (1,0)--(1,1);
	\draw (0,1) .. controls +(.1,-.4) and +(-.1,-.4) .. (1,1);
\end{tikzpicture}}
-\frac{1}{x}
\raisebox{-.1in}{
\begin{tikzpicture}[scale=.55, line width=1.35]
	\filldraw[fill=gray!25,draw=gray!25,line width=4pt]  (0,0) -- (0,1) -- (1,1) -- (1,0) -- (0,0);
	 \foreach \x in {0,1} {
	 \filldraw [black] (\x, 0) circle (2pt);
	 \filldraw [black] (\x, 1) circle (2pt);
	 }
	\draw  (1,0)--(1,1);
	\draw (0,0) .. controls +(.1,.4) and +(-.1,.4) .. (1,0);
\end{tikzpicture}}
+\frac{1}{x^2}
\raisebox{-.1in}{
\begin{tikzpicture}[scale=.55, line width=1.35]
	\filldraw[fill=gray!25,draw=gray!25,line width=4pt]  (0,0) -- (0,1) -- (1,1) -- (1,0) -- (0,0);
	 \foreach \x in {0,1} {
	 \filldraw [black] (\x, 0) circle (2pt);
	 \filldraw [black] (\x, 1) circle (2pt);
	 }
	\draw  (0,0)--(0,1);
\end{tikzpicture}}
+\frac{1}{x^2}
\raisebox{-.1in}{
\begin{tikzpicture}[scale=.55, line width=1.35]
	\filldraw[fill=gray!25,draw=gray!25,line width=4pt]  (0,0) -- (0,1) -- (1,1) -- (1,0) -- (0,0);
	 \foreach \x in {0,1} {
	 \filldraw [black] (\x, 0) circle (2pt);
	 \filldraw [black] (\x, 1) circle (2pt);
	 }
	\draw  (1,0)--(1,1);
\end{tikzpicture}}
+\frac{1}{x^2}
\raisebox{-.1in}{
\begin{tikzpicture}[scale=.55, line width=1.35]
	\filldraw[fill=gray!25,draw=gray!25,line width=4pt]  (0,0) -- (0,1) -- (1,1) -- (1,0) -- (0,0);
	 \foreach \x in {0,1} {
	 \filldraw [black] (\x, 0) circle (2pt);
	 \filldraw [black] (\x, 1) circle (2pt);
	 }
	\draw  (1,0)--(0,1);
\end{tikzpicture}}\\
&+\frac{1}{x^2}
\raisebox{-.1in}{
\begin{tikzpicture}[scale=.55, line width=1.35]
	\filldraw[fill=gray!25,draw=gray!25,line width=4pt]  (0,0) -- (0,1) -- (1,1) -- (1,0) -- (0,0);
	 \foreach \x in {0,1} {
	 \filldraw [black] (\x, 0) circle (2pt);
	 \filldraw [black] (\x, 1) circle (2pt);
	 }
	\draw  (0,0)--(1,1);
\end{tikzpicture}}
+\frac{1}{x^2}
\raisebox{-.1in}{
\begin{tikzpicture}[scale=.55, line width=1.35]
	\filldraw[fill=gray!25,draw=gray!25,line width=4pt]  (0,0) -- (0,1) -- (1,1) -- (1,0) -- (0,0);
	 \foreach \x in {0,1} {
	 \filldraw [black] (\x, 0) circle (2pt);
	 \filldraw [black] (\x, 1) circle (2pt);
	 }
	\draw (0,0) .. controls +(.1,.4) and +(-.1,.4) .. (1,0);
\end{tikzpicture}}
+\frac{1}{x^2}
\raisebox{-.1in}{
\begin{tikzpicture}[scale=.55, line width=1.35]
	\filldraw[fill=gray!25,draw=gray!25,line width=4pt]  (0,0) -- (0,1) -- (1,1) -- (1,0) -- (0,0);
	 \foreach \x in {0,1} {
	 \filldraw [black] (\x, 0) circle (2pt);
	 \filldraw [black] (\x, 1) circle (2pt);
	 }
	\draw (0,1) .. controls +(.1,-.4) and +(-.1,-.4) .. (1,1);
\end{tikzpicture}}
-\frac{3}{x^3}
\raisebox{-.1in}{
\begin{tikzpicture}[scale=.55, line width=1.35]
	\filldraw[fill=gray!25,draw=gray!25,line width=4pt]  (0,0) -- (0,1) -- (1,1) -- (1,0) -- (0,0);
	 \foreach \x in {0,1} {
	 \filldraw [black] (\x, 0) circle (2pt);
	 \filldraw [black] (\x, 1) circle (2pt);
	 }
\end{tikzpicture}}~,
\end{align*}

\begin{align*}
\raisebox{-.1in}{
\begin{tikzpicture}[scale=.55, line width=1.35]
	\filldraw[fill=gray!25,draw=gray!25,line width=4pt]  (0,0) -- (0,1) -- (1,1) -- (1,0) -- (0,0);
	 \foreach \x in {0,1} {
	 \filldraw [black] (\x, 0) circle (2pt);
	 \filldraw [black] (\x, 1) circle (2pt);
	 }
	\draw  (0,1)--(0,0) (1,1)--(1,0);
\end{tikzpicture}}
- \frac{1}{x}
\raisebox{-.1in}{
\begin{tikzpicture}[scale=.55, line width=1.35]
	\filldraw[fill=gray!25,draw=gray!25,line width=4pt]  (0,0) -- (0,1) -- (1,1) -- (1,0) -- (0,0);
	 \foreach \x in {0,1} {
	 \filldraw [black] (\x, 0) circle (2pt);
	 \filldraw [black] (\x, 1) circle (2pt);
	 }
	\draw  (1,1)--(1,0);
\end{tikzpicture}}
- \frac{1}{x}
\raisebox{-.1in}{
\begin{tikzpicture}[scale=.55, line width=1.35]
	\filldraw[fill=gray!25,draw=gray!25,line width=4pt]  (0,0) -- (0,1) -- (1,1) -- (1,0) -- (0,0);
	 \foreach \x in {0,1} {
	 \filldraw [black] (\x, 0) circle (2pt);
	 \filldraw [black] (\x, 1) circle (2pt);
	 }
	\draw  (0,1)--(0,0);
\end{tikzpicture}}
+\frac{1}{x^2}
\raisebox{-.1in}{
\begin{tikzpicture}[scale=.55, line width=1.35]
	\filldraw[fill=gray!25,draw=gray!25,line width=4pt]  (0,0) -- (0,1) -- (1,1) -- (1,0) -- (0,0);
	 \foreach \x in {0,1} {
	 \filldraw [black] (\x, 0) circle (2pt);
	 \filldraw [black] (\x, 1) circle (2pt);
	 }
\end{tikzpicture}}~,
\qquad
\raisebox{-.1in}{
\begin{tikzpicture}[scale=.55, line width=1.35]
	\filldraw[fill=gray!25,draw=gray!25,line width=4pt]  (0,0) -- (0,1) -- (1,1) -- (1,0) -- (0,0);
	 \foreach \x in {0,1} {
	 \filldraw [black] (\x, 0) circle (2pt);
	 \filldraw [black] (\x, 1) circle (2pt);
	 }
	\draw (0,0) .. controls +(.1,.4) and +(-.1,.4) .. (1,0);
	\draw (0,1) .. controls +(.1,-.4) and +(-.1,-.4) .. (1,1);
\end{tikzpicture}}
- \frac{1}{x}
\raisebox{-.1in}{
\begin{tikzpicture}[scale=.55, line width=1.35]
	\filldraw[fill=gray!25,draw=gray!25,line width=4pt]  (0,0) -- (0,1) -- (1,1) -- (1,0) -- (0,0);
	 \foreach \x in {0,1} {
	 \filldraw [black] (\x, 0) circle (2pt);
	 \filldraw [black] (\x, 1) circle (2pt);
	 }
	\draw (0,0) .. controls +(.1,.4) and +(-.1,.4) .. (1,0);
\end{tikzpicture}}
- \frac{1}{x}
\raisebox{-.1in}{
\begin{tikzpicture}[scale=.55, line width=1.35]
	\filldraw[fill=gray!25,draw=gray!25,line width=4pt]  (0,0) -- (0,1) -- (1,1) -- (1,0) -- (0,0);
	 \foreach \x in {0,1} {
	 \filldraw [black] (\x, 0) circle (2pt);
	 \filldraw [black] (\x, 1) circle (2pt);
	 }
	\draw (0,1) .. controls +(.1,-.4) and +(-.1,-.4) .. (1,1);
\end{tikzpicture}}
+\frac{1}{x^2}
\raisebox{-.1in}{
\begin{tikzpicture}[scale=.55, line width=1.35]
	\filldraw[fill=gray!25,draw=gray!25,line width=4pt]  (0,0) -- (0,1) -- (1,1) -- (1,0) -- (0,0);
	 \foreach \x in {0,1} {
	 \filldraw [black] (\x, 0) circle (2pt);
	 \filldraw [black] (\x, 1) circle (2pt);
	 }
\end{tikzpicture}}~,
\end{align*}

\begin{align*}
\raisebox{-.1in}{
\begin{tikzpicture}[scale=.55, line width=1.35]
	\filldraw[fill=gray!25,draw=gray!25,line width=4pt]  (0,0) -- (0,1) -- (1,1) -- (1,0) -- (0,0);
	 \foreach \x in {0,1} {
	 \filldraw [black] (\x, 0) circle (2pt);
	 \filldraw [black] (\x, 1) circle (2pt);
	 }
	\draw  (0,1)--(1,0) (0,0)--(1,1);
\end{tikzpicture}}
- \frac{1}{x}
\raisebox{-.1in}{
\begin{tikzpicture}[scale=.55, line width=1.35]
	\filldraw[fill=gray!25,draw=gray!25,line width=4pt]  (0,0) -- (0,1) -- (1,1) -- (1,0) -- (0,0);
	 \foreach \x in {0,1} {
	 \filldraw [black] (\x, 0) circle (2pt);
	 \filldraw [black] (\x, 1) circle (2pt);
	 }
	\draw  (0,1)--(1,0);
\end{tikzpicture}}
- \frac{1}{x}
\raisebox{-.1in}{
\begin{tikzpicture}[scale=.55, line width=1.35]
	\filldraw[fill=gray!25,draw=gray!25,line width=4pt]  (0,0) -- (0,1) -- (1,1) -- (1,0) -- (0,0);
	 \foreach \x in {0,1} {
	 \filldraw [black] (\x, 0) circle (2pt);
	 \filldraw [black] (\x, 1) circle (2pt);
	 }
	\draw  (0,0)--(1,1);
\end{tikzpicture}}
+\frac{1}{x^2}
\raisebox{-.1in}{
\begin{tikzpicture}[scale=.55, line width=1.35]
	\filldraw[fill=gray!25,draw=gray!25,line width=4pt]  (0,0) -- (0,1) -- (1,1) -- (1,0) -- (0,0);
	 \foreach \x in {0,1} {
	 \filldraw [black] (\x, 0) circle (2pt);
	 \filldraw [black] (\x, 1) circle (2pt);
	 }
\end{tikzpicture}}~.
\end{align*}
\medskip

In displaying these elements, we see that $\mathsf{QP}_{1+\frac{1}{2}}(x)$ is not
naturally a subalgebra of $\mathsf{QP}_{2}(x)$ as it was in the case of the
partition algebra.
\end{example}

The dimensions of these algebras are determined by counting the elements
in Equations \eqref{eq:basisQP} and \eqref{eq:basistildeQP}.

The sequence of dimensions of $\dim(\mathsf{QP}_k(x))$ for $k$ an integer was shown in
\cite{DO} to be
\[ \text{dim}(\mathsf{QP}_k(x)) = \sum_{j=1}^{2k} (-1)^{j-1} B(2k-j) + 1 \]
and are every other term in the \cite{OEIS} sequence \href{https://oeis.org/A000296}{A000296}.
The dimension is equal to the number of set partitions of
$[k]\cup[\overline{k}]$ without blocks of size one.

The next corollary shows that the sequence of dimensions of
$\dim(\mathsf{QP}_{k+\frac{1}{2}}(x))$ for $k$ an integer
is $B(2k)$.
This sequence is every other term in the \cite{OEIS}
sequence \href{https://oeis.org/A000110}{A000110}.

The sequence of dimensions of $\widetilde{\mathsf{QP}}_{k+1}(x)$ is
given by \cite{OEIS} sequence \href{https://oeis.org/A207978}{A207978}.
Using the standard counting techniques of inclusion-exclusion we deduce that
$${\mathrm{dim}}(\widetilde{\mathsf{QP}}_{k+1}(x))
= \sum_{s=0}^{2k} (-1)^s \binom{2k}{s} B(2k+2-s)~.$$
The expression $\binom{2k}{s} B(2k+2-s)$ represents the number of set partitions of
$[k+1] \cup [\o{k+1}]$ where a subset of $s$ of the vertices in
$[k] \cup [\o{k}]$ are isolated
and a there is a set partition on the remaining vertices.

\begin{example}
The sequence of dimensions of the algebras for $0 \leq k \leq 6$ is given in the table below.

\begin{center}
\begin{tabular}{|c||c|c|c|c|c|c|c|}
\hline
$k$&0&1&2&3&4&5&6\\
\hline\hline
&&&&&&&\\
$\dim(\mathsf{QP}_k(x))$&1&1&4&41&715&17722&580317\\&&&&&&&\\
\hline
&&&&&&&\\
$\dim(\mathsf{QP}_{k+\frac{1}{2}}(x))$&1&2&15&203&4140&115975&4213597\\&&&&&&&\\
\hline
&&&&&&&\\
$\dim(\widetilde{\mathsf{QP}}_{k+1}(x))$&2&7&67&1080&25287&794545&31858034\\&&&&&&&\\
\hline
\end{tabular}
\end{center}
\end{example}

\begin{cor}\label{Cor : dimensions} For $k \in \mathbb{Z}_{\geq 0}$,
the set $\{\overline{d} \, |\, d\in \Dh_{k+\frac{1}{2}} \}$
is a basis for $\mathsf{QP}_{k+\frac{1}{2}}(x)$ and the
dimension of $\mathsf{QP}_{k+\frac{1}{2}}(x)$ is equal to $B(2k)$,
the number of set partitions of a set of size $2k$.
\end{cor}
\begin{proof}  We will show that $\left|\Dh_{k+\frac{1}{2}}\right| = \left|\Ph_k\right|$
by establishing a bijection between these two sets.

Fix $k \in \mathbb{Z}_{\geq 0}$, then
the dimension of $\mathsf{QP}_{k+\frac{1}{2}}(x)$ is equal to the number of elements in
$\Dh_{k+\frac{1}{2}}$ since by
Proposition \ref{Prop : coeff d bar} the $\overline{d}$
for $d \in \Dh_{k+\frac{1}{2}}$ are linearly independent.

Define $F: \Ph_k \rightarrow \Dh_{k+\frac{1}{2}}$ for any set partition $d \in \Ph_k$ by
placing all singletons in $d$ in the block containing $\{k+1,\overline{k+1}\}$, that is
\[F(d) = \{  S \in d \text{ such that } |S|>1,\} \cup
\{\{ a \, |\, \{a\} \in d\} \cup \{k+1,\overline{k+1}\}\}. \]
Since $F$ is invertible, it follows that $\left|\Ph_k\right| = \left|\Dh_{k+\frac{1}{2}}\right| = B(2k)$
and it is equal to the dimension of $\mathsf{QP}_{k+\frac{1}{2}}(x)$.
\end{proof}

\section{Irreducible representations of quasi-partition algebras}\label{sec:QPirreps}

In this section we will use the representations constructed for
$\mathsf{P}_k(x)$ and $\mathsf{P}_{k+\frac{1}{2}}(x)$ in Section \ref{sec:repsPk}
and the projections $\pi^{\otimes k}$ and $\pi_{k+1}^{\otimes k}$ to construct
the representations of $\mathsf{QP}_k(x)$.

A different approach was taken by Scrimshaw \cite[Theorem 3.2]{Scrimshaw}
where he shows that that the quasi-partition
algebra as a subalgebra of the partition algebra $\mathsf{P}_k(x-1)$
is a cellular algebra and hence is semi-simple if $\mathsf{P}_k(x-1)$.

Recall that in Section \ref{sec:repsPk} we defined standard modules for every
$\nu\vdash m$, where $0\leq m \leq k$, $\Delta_k(\nu)$ with basis
\[
\mathcal{B}_k(\nu) =
\{d \otimes T \, |\, d \text{ a } (k,m)\text{-standard diagram and }
T \text{ standard tableau of shape } \nu \}.
\]
As in Section \ref{sec:repsPk}, we will assume that
the algebra $\Pa_{k}(x)$ is semisimple \cite{MS};
that is we assume that $x \notin \{0, 1, 2, \ldots, 2k-2\}$.

Recall that $V(k,m)$ contains diagrams $d$ such that $p(d)=m$ and the only propagating vertices in the bottom row are: $\overline{1}, \ldots, \overline{m}$.

\begin{lemma}\label{lem:basisQ}
    For $d\otimes T\in \mathcal{B}_k(\nu)$, we have $\pi^{\otimes k} d \otimes T = 0$ if $d$ has singleton vertices on the top row. And if $d$ has no singleton vertex in the top row, then
    \[
  \o{d\otimes T} := \pi^{\otimes k} d\otimes T = d\otimes T + \sum_{d'\otimes T' \in \mathcal{B}_k(\nu)} a_{d'}(x) d' \otimes T' ~, 
  \]
  where $a_{d'}(x) \in \mathbb{C}(x)$ for some $0\leq i \leq k$. In addition,  $d'$ has singletons in the first row.
\end{lemma}
\begin{proof}
If $d$ has isolated vertices on the top row then $\pi^{\otimes k} \cdot d =0$ by Lemma \ref{lem:odzero}. 
If $d$ has no singleton vertices in the first row, we have that
since $\pi^{\otimes k}\in \mathsf{P}_k(x)$ and $\Delta_k(\nu)$ is a $\mathsf{P}_k(x)$-module,
the $\pi^{\otimes k} \cdot d \otimes T$ is a linear combination of elements in
$\mathcal{B}_k(\nu)$. The terms of $\pi^{\otimes k} \cdot d \otimes T$
are of obtained from $\frac{1}{(-x)^J}\mathsf{p}_J d \otimes T$, where
$\mathsf{p}_J d \in V(k,m)$ or equal to zero if the propagation number of
the resulting diagram $d'$ is less than $m$. If the diagram $d'$ in the
product of $\mathsf{p}_J d$ has propagation number $m$, then
$d'= d_1 \sigma$ for some $\sigma\in S_m$ and $d_1$ has the same
top as $d'$ and is a $(k,m)$-standard diagram.
\end{proof}

\begin{example} Let $k=3$, we give an example of a nonzero element in $\pi^{\otimes k} d\otimes T$ in Lemma \ref{lem:basisQ}: 
\squaresize=14pt
\[
\overline{
\raisebox{-0.1in}{
\begin{tikzpicture}[scale=.55,line width=1.35pt]
\foreach \i in {1,...,3}
{ \path (\i,1) coordinate (T\i); \path (\i,0) coordinate (B\i); }
\filldraw[fill=gray!25,draw=gray!25,line width=4pt]  
(T1) -- (T3) -- (B3) -- (B1) -- (T1);
\draw (T2) .. controls +(.1,-.4) and +(-.1,-.4) .. (T3);
\draw (T1) -- (B1);
\foreach \i in {1,...,3}
{ \filldraw (T\i) circle (2pt); \filldraw (B\i) circle (2pt); }
\end{tikzpicture}}
\otimes
\raisebox{-0.05in}{$\young{1\cr}$}}
=
\raisebox{-0.1in}{
\begin{tikzpicture}[scale=.55,line width=1.35pt]
\foreach \i in {1,...,3}
{ \path (\i,1) coordinate (T\i); \path (\i,0) coordinate (B\i); }
\filldraw[fill=gray!25,draw=gray!25,line width=4pt]  
(T1) -- (T3) -- (B3) -- (B1) -- (T1);
\draw (T2) .. controls +(.1,-.4) and +(-.1,-.4) .. (T3);
\draw (T1) -- (B1);
\foreach \i in {1,...,3}
{ \filldraw (T\i) circle (2pt); \filldraw (B\i) circle (2pt); }
\end{tikzpicture}}
\otimes
\raisebox{-0.05in}{$\young{1\cr}$}
- \frac{1}{x}
\raisebox{-0.1in}{
\begin{tikzpicture}[scale=.55,line width=1.35pt]
\foreach \i in {1,...,3}
{ \path (\i,1) coordinate (T\i); \path (\i,0) coordinate (B\i); }
\filldraw[fill=gray!25,draw=gray!25,line width=4pt]  
(T1) -- (T3) -- (B3) -- (B1) -- (T1);
\draw (T1) -- (B1);
\foreach \i in {1,...,3}
{ \filldraw (T\i) circle (2pt); \filldraw (B\i) circle (2pt); }
\end{tikzpicture}}
\otimes
\raisebox{-0.05in}{$\young{1\cr}$}
\]
\end{example}

We define $\mathcal{Q}_k(\nu)$ to be the set of nonzero $\pi^{\otimes k} d\otimes T$, for $d\otimes T \in \mathcal{B}_k(\nu)$. 

\begin{prop} Let $k\in \mathbb{Z}_{\geq 0}$, for each $\nu\vdash m$ such that $0\leq m \leq k$ the set $\mathcal{Q}_k(\nu)$ is linearly independent. And its size is the number of $[k]$-set tableaux of shape $(n-|\nu|,\nu)$ such that no sets of size one are in the first row of the tableau. 
\end{prop}
\begin{proof}
    The linear independence follows from Lemma \ref{lem:basisQ}. 
    The leading elements in each nonzero $\pi^{\otimes k} d\otimes T$
    is in bijection with the tableaux described in the statement of the Lemma using the bijection, $\rho$, from Section \ref{sec:repsPk}.
\end{proof}

Define $\mathsf{QP}_k^\nu$ to be the $\mathbb{C}(x)$-Span of the
elements in $\mathcal{Q}_k(\nu)$ for every $\nu\vdash m$ and $0\leq m \leq k$. 

\begin{prop}\label{prop:simpleQP} Let $k\in\mathbb{Z}_{\geq 0}$. For every $\nu\vdash m$ and $0\leq m \leq k$, 
    $\mathsf{QP}_k^\nu$ is a simple module of $\mathsf{QP}_k(x)$. 
\end{prop}
\begin{proof}
It is straightforward to show that $\mathsf{QP}_k^\nu$ is a module of $\mathsf{QP}_k(\nu)$.
Every basis element in $\mathsf{QP}_k^\nu$ is of the form $\pi^{\otimes k} d'\otimes T$
for some $d'$ without singletons in the first row.  Thus, if $\overline{d}$
is a basis element in $\mathsf{QP}_k(x)$, then
\begin{eqnarray*}
\o{d} \cdot \o{d'\otimes T} & = \pi^{\otimes k} d\pi^{\otimes k} \pi^{\otimes k} d' \otimes T \\
& = \pi^{\otimes k} d \pi^{\otimes k} d' \otimes T \\
& = \pi^{\otimes k} (d \cdot \o{d'\otimes T})
\end{eqnarray*}
Notice that $d\cdot \o{d'\otimes T}$ is a linear combination of elements in $\mathcal{B}_k(\nu)$; therefore, $\pi^{\otimes k} (d \cdot \o{d'\otimes T})$ is a linear combination of elements in $\mathcal{Q}_k(\nu)$.

To show that $\mathsf{QP}_k^\nu$ is simple $\mathsf{QP}_k(x)$-module, we will show that it is generated by every nonzero element.
We know that the standard module, $\Delta_k(\nu)$, is a simple $\mathsf{P}_k(x)$-module, then it is generated by every nonzero element.
If $\pi^{\otimes k} \cdot d\otimes T\neq 0$, then it generates  $\Delta_k(\nu)$ as a $\mathsf{P}_k(x)$-module.
Therefore, for every $d'\otimes T' \in \mathcal{B}_k(\nu)$ such that $d'$ has no
singletons in its top row, there must be an $f\in \mathsf{P}_k(x)$ such that
\[ f\cdot \pi^{\otimes k} d \otimes T = d'\otimes T'~,\]
Hence the element  $\o{f} = \pi^{\otimes k} f \pi^{\otimes k}$ has the property that
\[ \overline{f}\cdot \pi^{\otimes k} d\otimes T = \o{d'\otimes T'}\in \mathcal{Q}_k(\nu).\]

Since $\mathsf{QP}_k^\nu$ is the span of the elements $\o{d'\otimes T'}$, by linearity we can generate all its elements by acting on the nonzero element $\pi^{\otimes k} d\otimes T$, and more generally on any nonzero linear combination of these elements. 
\end{proof}

Recall that $\Delta_k(\nu) \cong V(k,m)\otimes_{S_m} \mathbb{S}^\nu$.  Let $\pi^{\otimes k} V(k,m) = \{ \pi^{\otimes k} d\, |\, d\in \Dh_k\}$, where in $\pi^{\otimes k}d$ we set $\mathsf{p}_J d = 0$, whenever $p(\mathsf{p}_Jd) < m$ due to the action of $\mathsf{P}_k(x)$ on $V(k,m)$. Therefore, we have that 
\[ \mathsf{QP}_k^\nu \cong \pi^{\otimes k} V(k,m) \otimes_{S_m} \mathbb{S}^\nu ~,\]
where $\mathbb{S}^\nu$ are the simple modules of $\mathbb{S}^m$.

\begin{theorem} Let $k\in \mathbb{Z}_{\geq 0}$,
the set $\{ \mathsf{QP}_k^\nu\, |\, \nu \vdash m \text{ where } 0 \leq m \leq k\}$ forms a complete set of mutually non-isomorphic simple modules for $\mathsf{QP}_k(x)$.
\end{theorem}
\begin{proof}
We have already shown in Proposition \ref{prop:simpleQP} that the $\mathsf{QP}_k^\nu$ are simple modules.  By Proposition \ref{prop:symmetricgroup}, $\mathsf{QP}_k(x)$ contains a group isomorphic to $S_k$.  For an element $\pi^{\otimes k}\sigma$ and $\pi^{\otimes k} d \otimes T \in \pi^{\otimes k} V(k,m)\otimes_{S_m} \mathbb{S}^\nu$ we have 

\[\pi^{\otimes k} \sigma \cdot \pi^{\otimes k} d \otimes T = \pi^{\otimes k} \sigma d \otimes T = \pi^{\otimes k} d' \otimes \tau T~, \]
where $\tau \in S_m$ by Lemma \ref{V-module}. Hence if $\mu\neq \nu$, then the group generated by $\o{\mathsf{s}_i}$ act differently on $\mathsf{QP}_k^\nu$ than on $\mathsf{QP}_k^\mu$.  Therefore, these modules are not isomorphic.
\end{proof}

We note that using the bijection $\rho$ from Section \ref{sec:repsPk}
and setting $x=n\in \mathbb{Z}_{\geq 0}$ with $n\geq 2k$, we can give
an equivalent description of simple $\mathsf{QP}_k(n)$-modules using
tableaux.
Extending $\rho$ linearly
we map $\rho(\o{d\otimes T})$ to linear combination of tableaux.
In the following example we do this for $k=1$ and $k=2$.

\begin{example}\label{ex:pikM} For $k=1$, there are two simple  modules of $\mathsf{P}_1(n)$

\[ \mathsf{P}_1^{\emptyset} = \mathbb{C}\text{-Span}\left\{  \begin{array}{c}
        \scriptsize
        \young{   & & \dots & & 1\cr }
    \end{array} \right\} \qquad 
    \mathsf{P}_1^{(1)} = \mathbb{C}\text{-Span}\left\{  \begin{array}{c}
        \scriptsize
        \young{ 1 \cr  & & \dots & \cr }
    \end{array} \right\}
    \]
Then the corresponding  $\mathsf{QP}_1(n)$ simple modules are: 
\[ \mathsf{QP}_1^{\emptyset} = 0  \qquad
\mathsf{QP}_1^{(1)} = \mathbb{C}\text{-Span}\left\{  \begin{array}{c}
        \scriptsize
        \young{ 1 \cr  & & \dots & \cr }
    \end{array} \right\}
    \]         
For $k=2$:  The four simple modules for $\mathsf{P}_2(n)$ are:

\[ \mathsf{P}_2^{\emptyset} = \mathbb{C}\text{-Span}\left\{  \begin{array}{c}
        \scriptsize
        \young{   & & \dots & & 1 2\cr }
    \end{array},  \begin{array}{c}
        \scriptsize
        \young{   & & \dots & & 1& 2\cr }
    \end{array} \right\}
\]
\[    \mathsf{P}_2^{(1)} = \mathbb{C}\text{-Span}\left\{  \begin{array}{c}
        \scriptsize
        \young{ 12 \cr  & & \dots & \cr }
    \end{array},   \begin{array}{c}
        \scriptsize
        \young{ 1 \cr  & & \dots & 2 \cr }
    \end{array} ,  \begin{array}{c}
        \scriptsize
        \young{ 2 \cr  & & \dots & 1\cr }
    \end{array}   \right\}
    \]
and 
\[ \mathsf{P}_2^{(2)} = \mathbb{C}\text{-Span}\left\{  \begin{array}{c}
        \scriptsize
        \young{  1 & 2 \cr  & & \dots & & \cr }
    \end{array} \right\} \qquad 
    \mathsf{P}_2^{(1,1)} = \mathbb{C}\text{-Span}\left\{  \begin{array}{c}
        \scriptsize
        \young{ 2\cr 1 \cr    & & \dots & \cr }
    \end{array} \right\} 
    \]
The corresponding four simple modules for $\mathsf{QP}_2(n)$:
\[\mathsf{QP}_2^{\emptyset} = \mathbb{C}\text{-Span}\left\{  \begin{array}{c}
        \scriptsize
        \young{   & & \dots & & 1 2\cr }
    \end{array} - \frac{1}{n} \begin{array}{c}
        \scriptsize
        \young{   & & \dots & & 1 &2\cr }
    \end{array} \right\} 
\]    
\[\mathsf{QP}_2^{(1)} = \mathbb{C}\text{-Span}\left\{  \begin{array}{c}
        \scriptsize
        \young{ 12 \cr  & & \dots & \cr }
    \end{array} - \frac{1}{n}  \begin{array}{c}
        \scriptsize
        \young{ 1 \cr  & & \dots & 2 \cr }
    \end{array}  -\frac{1}{n}   \begin{array}{c}
        \scriptsize
        \young{ 2 \cr  & & \dots & 1 \cr }
    \end{array} \right\} 
    \]
 \[\mathsf{QP}_2^{(2)} = \mathbb{C}\text{-Span}\left\{  \begin{array}{c}
        \scriptsize
        \young{  1 & 2 \cr  & & \dots & & \cr }
    \end{array} \right\} \qquad 
 \mathsf{QP}_2^{(1,1)} = \mathbb{C}\text{-Span}\left\{  \begin{array}{c}
        \scriptsize
        \young{ 2\cr 1 \cr    & & \dots & \cr }
    \end{array} \right\} 
    \]
\end{example}


It is possible similarly construct the simple modules of $\mathsf{QP}_{k+\frac{1}{2}}(x)$
following the details in Section \ref{sec:halfPkreps} about the representations of the half-partition
algebra. However, we do not carry out these constructions here because as we will see in the next
section, just before Theorem \ref{th:QPktildecentralizer}, $\mathsf{QP}_{k+\frac{1}{2}}(n)$ is
isomorphic to $\mathsf{P}_k(n-1)$. Hence the representation theory of these
algebras is isomorphic.

\section{Quasi-partition algebras as centralizers}\label{sec:quasicentralizers}

Daugherty and the first author \cite[Section 2.4]{DO} started with the definition that
\begin{equation}\label{eq:DOresult}
\mathsf{QP}_k(n) \cong {\mathrm{End}}_{S_n}\!\left( \left(\mathbb{S}^{(n-1,1)}_{S_n}\right)^{\otimes k} \right)
\end{equation}
where $\mathbb{S}^{(n-1,1)}_{S_n}$ is an irreducible symmetric group module indexed by
the partition $(n-1,1)$.  In contrast, in this paper we start
with the generic definition of the quasi-partition algebra as a subalgebra of the
partition algebra and deduce the relationship as centralizer algebras
at specializations of the parameter $x$. The objective of this section
is to describe the centralizer algebras that correspond to the diagram algebras
$\mathsf{QP}_k(x)$, $\mathsf{QP}_{k+\frac{1}{2}}(x)$ and
$\widetilde{\mathsf{QP}}_{k+1}(x)$ for values of $x$ equal to a sufficiently large integer.

 We will use this relationship to give a combinatorial description of how the
 representation theory of these algebras are related. This will be described
 in a diagram similar to a Bratteli diagram.  Counting paths in this diagram
 will allow us to compute the dimensions of the simple modules of these algebras.

Let $V_n = \mathbb{C}{\text{-Span}}\{ v_1, v_2,\ldots, v_n\}$, then it is well known
that $\mathbb{S}^{(n-1,1)}_{S_n} \cong \mathbb{C}{\text{-Span}}\{ v_1 - v_n, v_2 - v_n,\ldots, v_{n-1}-v_n\}$
and $\mathbb{S}^{(n)}_{S_n} \cong \mathbb{C}{\text{-Span}}\{ v_1+ v_2+\cdots+ v_n\}$ and that
$V_n \cong \mathbb{S}^{(n-1,1)}_{S_n} \oplus \mathbb{S}^{(n)}_{S_n}$ as an $S_n$-module.
The action of $\mathsf{P}_k(n)$ on the $V_n^{\otimes k}$ is defined in \cite{HR} and we follow similar
notation to develop the action of the quasi-partition algebras here.

For a diagram $d$, and integers $i_r, i_{\o{r}} \in [n]$ for $1 \leq r \leq k$, define the coefficient
\begin{equation}
\delta(d)^{(i_1, i_2, \ldots, i_k)}_{(i_{\o{1}}, i_{\o{2}}, \ldots, i_{\o{k}})}
= \begin{cases}
1&\hbox{ if }i_t = i_s \hbox{ if }t\hbox{ and } s
\hbox{ are connected in }d\\
0&\hbox{ otherwise }
\end{cases}~.
\end{equation}
Note that we say that two vertices $s,t \in [k] \cup [\o{k}]$
are connected if they are in the same connected component in $d$. 

The action of a diagram $d \in \Ph_k$ on an element of $V_n^{\otimes k}$ may then be stated as
\begin{equation}\label{eq:daction}
d \cdot (v_{i_1} \otimes v_{i_2} \otimes \cdots \otimes v_{i_k})=
\sum_{1 \leq i_{\o{1}}, i_{\o{2}}, \ldots, i_{\o{k}} \leq n}
\delta(d)^{(i_1, i_2, \ldots, i_k)}_{(i_{\o{1}}, i_{\o{2}}, \ldots, i_{\o{k}})}
v_{i_{\o{1}}} \otimes v_{i_{\o{2}}} \otimes \cdots \otimes v_{i_{\o{k}}}~.
\end{equation}
With this action, each element $d \in \mathsf{P}_k(n)$
is an element of ${\mathrm{End}}_{S_n}( V_n^{\otimes k})$.
We will use the result that $\mathsf{P}_k(n) \cong {\mathrm{End}}_{S_n}( V_n^{\otimes k})$
if $n \geq 2k$ and $\mathsf{P}_{k+\frac{1}{2}}(n) \cong
{\mathrm{End}}_{S_{n-1}}( {\mathrm{Res}}^{S_n}_{S_{n-1}} V_n^{\otimes k})$
if $n \geq 2k+1$.  For a clear exposition of these results, we refer the reader to \cite{BH2,BH3, CST}.

We develop in particular the action of the projection idempotents
$\piok$ and $\piokk$ that are used in Equation \eqref{eq:threealgebras}
when these algebras act on $V_n^{\otimes k}$ and $V_n^{\otimes k+1}$.
We will consider the actions of these elements for $x=n$.

Recall from the end of Section \ref{subsec:defs}
that $\mathsf{p} \in \Ph_1$ is defined by $\mathsf{p} = \{\{1\},\{\o1\}\}$.
The action of this element on $v_i\in V_n$, for $i\in [n]$, using Equation \eqref{eq:daction} is 
$$\mathsf{p} \cdot (v_i) = v_1 + v_2 + \cdots + v_n~.$$
Therefore we have that
$$\frac{1}{n} \mathsf{p} \cdot (v_i - v_n) = 0$$
and
$$\frac{1}{n} \mathsf{p} \cdot ( v_1 + v_2 + \cdots + v_n ) = v_1 + v_2 + \cdots + v_n~.$$
Therefore, $\frac{1}{n} \mathsf{p}$ is
a projection operator from $V_n$ to the subspace
$\mathbb{S}^{(n)}_{S_n}$ and $\pi = (1-\frac{1}{n}\mathsf{p})$ is a projection operator
from $V_n$ to the subspace $\mathbb{S}^{(n-1,1)}_{S_n}$.

\begin{theorem} \label{th:QPkcentralizer} For $n,k \in \mathbb{Z}_{>0}$,
using the action of $d \in \mathsf{QP}_k(n)$ as a linear transformation
on $V_n^{\otimes k}$,
$d \in {\mathrm{End}}_{S_n}( (\mathbb{S}^{(n-1,1)}_{S_n})^{\otimes k} )$
and if $n \geq 2k$, then
\begin{equation*}
\mathsf{QP}_k(n) \cong {\mathrm{End}}_{S_n}\!\left( \left(\mathbb{S}^{(n-1,1)}_{S_n}\right)^{\otimes k} \right)~.
\end{equation*}
\end{theorem}

\begin{proof}
We see that $\piok = \overline{\mathbf{1}^{\otimes k}} \in \mathsf{QP}_k(n)$ acts as the identity
on $(\mathbb{S}^{(n-1,1)}_{S_n})^{\otimes k}$.  For any other diagram
$d \in \Dh_k$, $\overline{d} = \piok d \piok$
acts first by the identity on $(\mathbb{S}^{(n-1,1)}_{S_n})^{\otimes k}$,
next by Equation \eqref{eq:daction}
(which may be an element of $V_n^{\otimes k}$ and not $(\mathbb{S}^{(n-1,1)}_{S_n})^{\otimes k}$),
followed by a projection into the subspace $(\mathbb{S}^{(n-1,1)}_{S_n})^{\otimes k}$.
Moreover,
$$\overline{d} \in {\mathrm{End}}_{S_n}\!\left(\left(\mathbb{S}^{(n-1,1)}_{S_n}\right)^{\otimes k}\right)$$
since it is a composition of three operations which commute with the diagonal action
of $S_n$.

Now take an arbitrary element $f \in {\mathrm{End}}_{S_n}((\mathbb{S}^{(n-1,1)}_{S_n})^{\otimes k})$.
Since $(\mathbb{S}^{(n-1,1)}_{S_n})^{\otimes k} \subseteq V_n^{\otimes k}$, when $n\geq 2k$ there
is an element $d \in \mathsf{P}_k(n) \cong {\mathrm{End}}_{S_n}(V_n^{\otimes k})$ such that
$d\cdot w = f(w)$ for all $w \in (\mathbb{S}^{(n-1,1)}_{S_n})^{\otimes k}$ and
$d\cdot w' = 0$ for all $w' \in$ the orthogonal complement
of $(\mathbb{S}^{(n-1,1)}_{S_n})^{\otimes k}$ in $V_n^{\otimes k}$.
Then $f(w) = \piok d \piok\cdot w$ for all $w \in (\mathbb{S}^{(n-1,1)}_{S_n})^{\otimes k}$
and therefore $f \in \mathsf{QP}_k(n)$ and we conclude that
$\mathsf{QP}_k(n) \cong {\mathrm{End}}_{S_n}( (\mathbb{S}^{(n-1,1)}_{S_n})^{\otimes k} )$.
\end{proof}

Let the set $\{ w_i : i \in [n-1] \}$ where $w_i := v_i - v_n$ be a
basis for $\mathbb{S}^{(n-1,1)}_{S_n}$.
We note that $v_n \in V_n$ and it is invariant
under the action of $S_{n-1}$ embedded in $S_n$ where the element $n$ is fixed by all elements of $S_{n-1}$.
Consider the embedding $$\left(\mathbb{S}^{(n-1,1)}_{S_n}\right)^{\otimes k}
\hookrightarrow \left(\mathbb{S}^{(n-1,1)}_{S_n}\right)^{\otimes k} \otimes V_n$$ that maps for $i_1, \ldots, i_k \in [n-1]$
$$w_{i_1} \otimes \cdots \otimes w_{i_k} \mapsto
w_{i_1} \otimes \cdots \otimes w_{i_k} \otimes v_n$$
We assume that elements of
$\mathsf{QP}_{k+\frac{1}{2}}(n)$ acts on
the space
${\mathrm{Res}}^{S_n}_{S_{n-1}}(\mathbb{S}^{(n-1,1)}_{S_n})^{\otimes k} \otimes \mathbb{C}\text{-Span}\{v_n \}$.

\begin{theorem} \label{th:QPkhalfcentralizer} For $n,k \in \mathbb{Z}_{>0}$, using the
action of an element $d \in \mathsf{QP}_{k+\frac{1}{2}}(n)$ on $V_n^{\otimes k}$, then
$d \in {\mathrm{End}}_{S_{n-1}}( {\mathrm{Res}}^{S_n}_{S_{n-1}}(\mathbb{S}^{(n-1,1)}_{S_n})^{\otimes k} )$
and if $n \geq 2k+1$, then
\begin{equation*}
\mathsf{QP}_{k+\frac{1}{2}}(n) \cong {\mathrm{End}}_{S_{n-1}}\!\left(
{\mathrm{Res}}^{S_n}_{S_{n-1}}\!\left(\mathbb{S}^{(n-1,1)}_{S_n}\right)^{\otimes k} \right)~.
\end{equation*}
\end{theorem}

\begin{proof}
We extend the action from Equation \eqref{eq:daction} to an element
$\o{d} = \piokk d \piokk$ for
$d \in {\mathcal D}_{k+\frac{1}{2}}$ and to $\mathsf{QP}_{k+\frac{1}{2}}(n)$ by
linearity.  The action of a diagram $d \in \Ph_{k+\frac{1}{2}}$
also follows Equation \eqref{eq:daction} and because $k+1$ and $\o{k+1}$
are in a connected component of $d$, then any $d$ in
$\Ph_{k+\frac{1}{2}}$ maps an element of
$(\mathbb{S}^{(n-1,1)}_{S_n})^{\otimes k} \otimes \mathbb{C}\text{-Span}\{ v_n \}$
to an element of $V_n^{\otimes k} \otimes \mathbb{C}\text{-Span}\{ v_n \}$.
Moreover, the action of each $d \in \Pa_{k+\frac{1}{2}}(n)$
is an element of ${\mathrm{End}}_{S_{n-1}}({\mathrm{Res}}^{S_n}_{S_{n-1}}V_n^{\otimes k})$

The expression in Equation \eqref{eq:projector}
allows us to deduce that $\piokk$ acts as the identity
on $(\mathbb{S}^{(n-1,1)}_{S_n})^{\otimes k} \otimes \mathbb{C}\text{-Span}\{w \}$.
Moreover, for $i_1, \ldots, i_k \in [n]$,
$$\piokk \cdot (v_{i_1} \otimes v_{i_2} \otimes \cdots \otimes v_{i_k} \otimes w)
= w_{i_1}' \otimes w_{i_2}' \otimes \cdots \otimes w_{i_k}' \otimes v_n~.$$
where
\begin{equation}\label{eq:wp}
w'_i = \frac{1}{n}(-v_1 -v_2 - \cdots +(n-1) v_i- v_{i+1} - \cdots -v_n)
\end{equation}
and hence it is a projection from $V_n^{\otimes k} \otimes \mathbb{C}\text{-Span}\{ v_n \}$
onto $(\mathbb{S}^{(n-1,1)}_{S_n})^{\otimes k} \otimes \mathbb{C}\text{-Span}\{ v_n \}$.
With the isomorphism $(\mathbb{S}^{(n-1,1)}_{S_n})^{\otimes k} \cong
(\mathbb{S}^{(n-1,1)}_{S_n})^{\otimes k} \otimes \mathbb{C}\text{-Span}\{ v_n \}$,
$$\overline{d} \in {\mathrm{End}}_{S_{n-1}}\!\left(
{\mathrm{Res}}^{S_n}_{S_{n-1}}\!\left(\mathbb{S}^{(n-1,1)}_{S_n}\right)^{\otimes k}\right)$$
since it is a composition of three operations which commute with the diagonal action
of $S_{n-1}$.

Take an element $f \in {\mathrm{End}}_{S_{n-1}}(
{\mathrm{Res}}^{S_n}_{S_{n-1}}(\mathbb{S}^{(n-1,1)}_{S_n})^{\otimes k} )$.
Define an element $d$
such that for all $w \in (\mathbb{S}^{(n-1,1)}_{S_n})^{\otimes k}$,
we have $d \cdot (w) = f(w)$ and for all $w'$ in the orthogonal complement of $(\mathbb{S}^{(n-1,1)}_{S_n})^{\otimes k}$
in $V_n^{\otimes k}$, then $d \cdot (w') = 0$.  Since
${\mathrm{Res}}^{S_n}_{S_{n-1}}(\mathbb{S}^{(n-1,1)}_{S_n})^{\otimes k}
\subseteq {\mathrm{Res}}^{S_n}_{S_{n-1}}V_n^{\otimes k}$ is an $S_{n-1}$ invariant subspace, if $n \geq 2k+1$
we can be assured that $d \in \mathsf{P}_{k + \frac{1}{2}}(n) \cong
{\mathrm{End}}_{S_{n-1}}( {\mathrm{Res}}^{S_n}_{S_{n-1}}V_n^{\otimes k})$.
It follows that
$f(w) = \piokk d \piokk\cdot w$ for all
$w \in (\mathbb{S}^{(n-1,1)}_{S_n})^{\otimes k}$ and hence
$f \in \mathsf{QP}_{k+\frac{1}{2}}(n)$ and we conclude
$\mathsf{QP}_{k+\frac{1}{2}}(n) \cong {\mathrm{End}}_{S_{n-1}}(
{\mathrm{Res}}^{S_n}_{S_{n-1}}(\mathbb{S}^{(n-1,1)}_{S_n})^{\otimes k} )$.
\end{proof}

We remark that this isomorphism then gives us a better picture of the structure
of the half quasi-partition algebra
since we have as a corollary that $\mathsf{QP}_{k+\frac{1}{2}}(n) \cong \mathsf{P}_{k}(n-1)$~.
This is because ${\mathrm{Res}}^{S_n}_{S_{n-1}} \mathbb{S}^{(n-1,1)}_{S_n} \cong
\mathbb{S}^{(n-2,1)}_{S_{n-1}} \oplus \mathbb{S}^{(n-1)}_{S_{n-1}} \cong V_{n-1}$.

\begin{theorem} \label{th:QPktildecentralizer} For $n,k \in \mathbb{Z}_{>0}$,
using the linear action of $d \in \widetilde{\mathsf{QP}}_{k+1}(n)$,
\begin{equation*}
d \in {\mathrm{End}}_{S_{n}}( (\mathbb{S}^{(n-1,1)}_{S_n})^{\otimes k}\otimes V_n )
\end{equation*}
and if $n \geq 2k+2$, then
\begin{equation*}
\widetilde{\mathsf{QP}}_{k+1}(n) \cong {\mathrm{End}}_{S_{n}}\!\left( \left(\mathbb{S}^{(n-1,1)}_{S_n}\right)^{\otimes k}\otimes V_n \right)~.
\end{equation*}
\end{theorem}

\begin{proof}
Finally, we note that $\piokk$ acts as the identity on
the subspace $(\mathbb{S}^{(n-1,1)}_{S_n})^{\otimes k} \otimes V_n$ and it acts as
$$\piokk \cdot (v_{i_1} \otimes v_{i_2} \otimes \cdots \otimes v_{i_k} \otimes v_{i_{k+1}})
= w_{i_1}' \otimes w_{i_2}' \otimes \cdots \otimes w_{i_k}' \otimes v_{i_{k+1}}~,$$
where $w'_i$ is defined in \eqref{eq:wp}
and hence it is also a projection from $V_n^{\otimes k+1}$
onto $(\mathbb{S}^{(n-1,1)}_{S_n})^{\otimes k} \otimes V_n$.

Recall from Equation \eqref{eq:basistildeQP} that a basis of $\widetilde{\mathsf{QP}}_{k+1}(n)$
is indexed by diagrams $d \in \Ph_{k+1}$ such that $d$ has no singletons in the first $k$ positions.
Then any $\tilde{d} = \piokk d \piokk$
acts on $(\mathbb{S}^{(n-1,1)}_{S_n})^{\otimes k} \otimes V_n$ by first the identity
followed by a linear transformation to $V_n^{\otimes k+1}$ followed by a projection
to $(\mathbb{S}^{(n-1,1)}_{S_n})^{\otimes k} \otimes V_n$.
We conclude that
$${\tilde d} \in {\mathrm{End}}_{S_{n}}\!\left(\left(\mathbb{S}^{(n-1,1)}_{S_n}\right)^{\otimes k} \otimes V_n\right)$$
since it is a composition of three operators which commute with the diagonal action of $S_n$.

Now assume that $n \geq 2k+2$,
if we take any $f \in {\mathrm{End}}_{S_{n}}( (\mathbb{S}^{(n-1,1)}_{S_n})^{\otimes k}\otimes V_n )$,
then since $(\mathbb{S}^{(n-1,1)}_{S_n})^{\otimes k}\otimes V_n \subseteq V_n^{\otimes k+1}$
we can choose $d \in \mathsf{P}_{k+1}(n) \cong {\mathrm{End}}_{S_{n}}( V_n^{\otimes k+1} )$
such that for each $w \in (\mathbb{S}^{(n-1,1)}_{S_n})^{\otimes k}\otimes V_n$,
$d \cdot w = f(w)$ and for each $w'$ in the orthogonal complement
of $(\mathbb{S}^{(n-1,1)}_{S_n})^{\otimes k}\otimes V_n$,
we get $d \cdot w' = 0$.  Therefore, for all $w \in (\mathbb{S}^{(n-1,1)}_{S_n})^{\otimes k}\otimes V_n$,
$$\tilde d \cdot w
= (\piokk d \piokk) \cdot w
= (\piokk d) \cdot w
= \piokk \cdot f(w) = f(w).$$
We can therefore conclude that $f \in \widetilde{\mathsf{QP}}_{k+1}(n)$
and $\widetilde{\mathsf{QP}}_{k+1}(n)
\cong {\mathrm{End}}_{S_{n}}( (\mathbb{S}^{(n-1,1)}_{S_n})^{\otimes k}\otimes V_n )~.$
\end{proof}

\section{Formulae for the dimensions of the simple modules} \label{sec:formulaedim}

The double centralizer theorem \cite[p.158]{P} implies that since
$\mathsf{QP}_k(n)$ is isomorphic to the centralizer algebra of $S_n$
when it acts diagonally on $(\mathbb{S}^{(n-1,1)}_{S_n})^{\otimes k}$,
then the irreducible representations of $\mathsf{QP}_k(n)$
are indexed by partitions $\lambda$ of $n$.
Moreover, the dimension of the
irreducible indexed by the partition $\lambda$ is equal to the multiplicity
of the irreducible $\mathbb{S}^{\lambda}_{S_n}$ in $(\mathbb{S}^{(n-1,1)}_{S_n})^{\otimes k}$.

There is an explicit formula for the dimension in terms of Stirling numbers of the
second kind and standard tableaux given in \cite[Theorem 4.6]{DO}.  That reference also states
the dimension in terms of a combinatorial interpretation of Kronecker tableaux \cite{CG}.
In this paper we will rely on an interpretation that has been developed
more recently \cite{COSSZ, HJ, BH2} in terms of set valued tableaux
in place of Kronecker tableaux \cite{CG, BDE} or vacillating tableaux \cite{K}.
We will use the result that the irreducible
representations of $\mathsf{QP}_k(n)$ have non-zero dimension if $|\o\lambda| \leq k$
where $\o\lambda = (\lambda_2, \lambda_3, \ldots, \lambda_{\ell(\lambda)})$.

We will use the structure of the chain of algebras from Equation \eqref{eq:QPtower} in order to
determine the dimensions of the irreducible representations
of $\mathsf{QP}_{k}(n)$, $\mathsf{QP}_{k+\frac{1}{2}}(n)$
and $\widetilde{\mathsf{QP}}_{k}(n)$.
We refer the reader to the diagram in Figure \ref{fig:Bdiagram}
to view a graphic representation of the
results that we present below.

Assume that $\lambda$ is a partition of an integer
$n$, then we will use the notation $\mu \rightarrow \lambda$
to indicate that $\mu$ is a partition of $n-1$ and
$\mu$ is contained in $\lambda$.
Alternatively, if $\mu$ is a partition of $n-1$, then
$\lambda \leftarrow \mu$ will indicate that $\lambda$
is a partition of $k$ and $\mu$ is contained in $\lambda$.
If $\lambda$ is a partition of $n$ and
$\mu$ is a partition of $n-1$, then
$\mu \rightarrow \lambda$ if and only if $\o\mu = \o\lambda$ or $\o\mu \rightarrow \o\lambda$.

Each row of the diagram in Figure \ref{fig:Bdiagram} displays partitions $\overline{\lambda}$
where $\lambda$ is in the index set of the irreducible
representations of the chain algebras from Equation
\eqref{eq:QPtower}.  The rows of the diagram in that figure are color coded as follows:
\begin{itemize}
\item The irreducible representations of $\mathsf{QP}_{k}(n)$ are displayed in {\color{red}{red}}.
\item The irreducible representations of $\mathsf{QP}_{k+\frac{1}{2}}(n)$ are displayed in {\color{blue}{blue}}.
\item The irreducible representations of $\widetilde{\mathsf{QP}}_{k+1}(n)$ are displayed in {\color{darkgreen}{green}}.
\end{itemize}

We will describe the relations of the rows of this diagram in the results below,
but first summarize the relationship here.
There is an inclusion of $\mathsf{QP}_{k}(n)$ into $\mathsf{QP}_{k+\frac{1}{2}}(n)$
and an inclusion of $\mathsf{QP}_{k+\frac{1}{2}}(n)$ into $\widetilde{\mathsf{QP}}_{k}(n)$
followed by a projection of the algebra $\widetilde{\mathsf{QP}}_{k}(n)$
into $\mathsf{QP}_{k+1}(n)$.  The diagram is similar to a Bratteli diagram except
that, because of the projection operation, the dimension is no longer the
number of paths in the diagram and is instead something slightly more complex.

The relations between the irreducibles in the rows of the diagram are summarized as follows:
\begin{itemize}
\item (Equation \eqref{eq:p1}) Between the $\mathsf{QP}_{k}(n)$ and $\mathsf{QP}_{k+\frac{1}{2}}(n)$ rows
there is an edge from {\color{red}{$\o\lambda$}} to {\color{blue}{$\o\mu$}}
if {\color{blue}{$\o\mu$}} $=$ {\color{red}{$\o\lambda$}}
or {\color{blue}{$\o\mu$}} $\rightarrow$ {\color{red}{$\o\lambda$}}
(alternatively, if {\color{blue}{$\mu$}} $\rightarrow$ {\color{red}{$\lambda$}}).
\item (Equation \eqref{eq:p2}) Between the $\mathsf{QP}_{k+\frac{1}{2}}(n)$ and $\widetilde{\mathsf{QP}}_{k+1}(n)$ rows
there is an edge from {\color{blue}{$\o\mu$}} to {\color{darkgreen}{$\o\lambda$}}
if {\color{darkgreen}{$\o\lambda$}} $=$ {\color{blue}{$\o\mu$}}
or {\color{darkgreen}{$\o\lambda$}} $\leftarrow$ {\color{blue}{$\o\mu$}}
(alternatively, these two conditions may be stated as
`if {\color{darkgreen}{$\lambda$}} $\leftarrow$ {\color{blue}{$\mu$}}').
\item (Equation \eqref{eq:p3}) Between the $\widetilde{\mathsf{QP}}_{k+1}(n)$ and $\mathsf{QP}_{k}(n)$ rows
there is an edge from {\color{darkgreen}{$\o\lambda$}} to {\color{red}{$\o\lambda$}}
but the dimension of the irreducible {\color{red}{$\o\lambda$}}
is equal to the dimension of the irreducible {\color{darkgreen}{$\o\lambda$}}
minus the dimension of the irreducible {\color{red}{$\o\lambda$}} at $k-1$.
\end{itemize}

The dimensions of the irreducible representations can be calculated recursively using the
following three results.

\begin{theorem} 
Let $n \geq 2k+1$, then
for $\mu \vdash n-1$ such that $|\o\mu| < k$, then
\begin{equation} \label{eq:p1}
\dim \Big(V^{\mu}_{\mathsf{QP}_{k+\frac{1}{2}}(n)}\Big)
= \sum_{\lambda \leftarrow \mu} \dim (V^{\lambda}_{\mathsf{QP}_{k}(n)})~.
\end{equation}
\end{theorem}

\begin{proof}
Since $n \geq 2k+1$ we have from Theorem \ref{th:QPkhalfcentralizer}
$$\mathsf{QP}_{k+\frac{1}{2}}(n) \cong {\mathrm{End}}_{S_{n-1}}(({\mathrm{Res}}^{S_n}_{S_{n-1}}{\mathbb S}^{(n-1,1)}_{S_n})^{\otimes k})$$
and then from Theorem \ref{th:QPkcentralizer}
$$\mathsf{QP}_k(n)\cong{\mathrm{End}}_{S_{n}}\!\left(\left({\mathbb S}^{(n-1,1)}_{S_n}\right)^{\otimes k}\right)~.$$
There is an inclusion of $S_{n-1}$ into $S_n$ and an inclusion of
$\mathsf{QP}_k(n)$ into $\mathsf{QP}_{k+\frac{1}{2}}(n)$.
This establishes a see-saw pair \cite[Theorem 9.2.2]{GW} between
$(S_{n-1}, S_n)$ and $(\mathsf{QP}_k(n), \mathsf{QP}_{k+\frac{1}{2}}(n) )$ which implies that
$${\mathrm{Hom}}\!\left({\mathbb S}^\lambda_{S_{n-1}}
, {\mathrm{Res}}^{S_n}_{S_{n-1}} {\mathbb S}^\mu_{S_{n}}
\right) = {\mathrm{Hom}}\!\left(V^\mu_{\mathsf{QP}_k(n)}, {\mathrm{Res}}^{\mathsf{QP}_{k+\frac{1}{2}}(n)}_{\mathsf{QP}_k(n)} V^\lambda_{\mathsf{QP}_{k+\frac{1}{2}}(n)}\right)~.$$
Now the dimension of the space on the left-hand side is $1$ if $\lambda \rightarrow \mu$
and $0$ otherwise.
\end{proof}

Let $\phi$ be the $S_{n-1}$ character of some $S_{n-1}$ module $V$
and $\psi$ be the $S_n$ character of some $S_n$ module $W$.
In the following proof we will use the notation $\phi \uparrow_{S_{n-1}}^{S_n}$ to
represent the character of the character $\mathrm{Ind}_{S_{n-1}}^{S_n} V$
and $\psi \downarrow^{S_n}_{S_{n-1}}$ to represent the character of the
restriction $\mathrm{Res}^{S_{n}}_{S_{n-1}} W$.  We recall that Frobenius
reciprocity \cite[Theorem 1.12.6]{Sagan} relates these characters through the identity
\[
\langle \psi, \phi\uparrow_{S_{n-1}}^{S_n} \rangle =
\langle \psi \downarrow^{S_n}_{S_{n-1}}, \phi \rangle~.
\]

\begin{theorem}
Let $n \geq 2k+1$, then
for if $\la \vdash n$ such that $|\o\la| < k$, then
\begin{equation} \label{eq:p2}
\dim (V^{\lambda}_{\widetilde{\mathsf{QP}}_{k}(n)})
= \sum_{\mu \rightarrow \lambda} \dim \Big(V^{\mu}_{\mathsf{QP}_{k+\frac{1}{2}}(n)}\Big)~.
\end{equation}
\end{theorem}

\begin{proof}
First, we remark that
\begin{align}
{\mathrm{Ind}}_{S_{n-1}}^{S_n} {\mathrm{Res}}^{S_n}_{S_{n-1}}({\mathbb S}^{(n-1,1)}_{S_n})^{\otimes k}
&\cong {\mathrm{Ind}}_{S_{n-1}}^{S_n}
\!\left( \left({\mathrm{Res}}^{S_n}_{S_{n-1}}({\mathbb S}^{(n-1,1)}_{S_n})^{\otimes k}\right) \!\otimes
{\mathbb S}^{(n-1)}_{S_{n-1}}\right)\nonumber\\
&\cong
({\mathbb S}^{(n-1,1)}_{S_n})^{\otimes k} \otimes
{\mathrm{Ind}}_{S_{n-1}}^{S_n}\!\left({\mathbb S}^{(n-1)}_{S_{n-1}}\right)\label{eq:cong}\\
&\cong (\mathbb{S}^{(n-1,1)}_{S_n})^{\otimes k} \otimes V_n~.\nonumber
\end{align}
The isomorphism on the line labelled by Equation \eqref{eq:cong} follows from the ``tensor identity''
\cite[Equation (3.18)]{HR}.  That is, for an $S_n$ module $M$ and $S_{n-1}$ module $N$, the map
\begin{center}
\begin{tabular}{ccc}
${\mathrm{Ind}}^{S_n}_{S_{n-1}} \!\left( {\mathrm{Res}}^{S_n}_{S_{n-1}}(M) \otimes N \right)$
 & $\longrightarrow$  & $M \otimes {\mathrm{Ind}}^{S_n}_{S_{n-1}} N$
 \\
 $g \otimes_{S_n} ( m \otimes n )$  & $\longmapsto$ & $gm \otimes (g \otimes_{S_n} n)$
\end{tabular}
\end{center}
is an isomorphism of $S_n$ modules.

By the double centralizer theorem \cite[p.158]{P},
the dimension of an irreducible ${\widetilde{\mathsf{QP}}}_{k+1}(n)$ module
indexed by a partition $\lambda \vdash n$ is equal to the multiplicity
of ${\mathbb{S}}_{S_n}^\lambda$ in $(\mathbb{S}^{(n-1,1)}_{S_n})^{\otimes k} \otimes V_n$
which we have now established to be isomorphic to
$${\mathrm{Ind}}_{S_{n-1}}^{S_n} {\mathrm{Res}}^{S_n}_{S_{n-1}}\!\left({\mathbb S}^{(n-1,1)}_{S_n}\right)^{\otimes k}~.$$
Let $\phi_{k}$ represent the $S_{n-1}$ character of
${\mathrm{Res}}^{S_n}_{S_{n-1}}({\mathbb S}^{(n-1,1)}_{S_n})^{\otimes k}
\cong ({\mathrm{Res}}^{S_n}_{S_{n-1}}{\mathbb S}^{(n-1,1)}_{S_n})^{\otimes k}$, then
by Frobenius reciprocity
$${\mathrm{dim}} (V^\lambda_{{\widetilde{\mathsf{QP}}}_{k}(n)})
= \left< \chi^{\lambda}_{S_n}, \phi_{k} \uparrow_{S_{n-1}}^{S_n} \right>
= \left< \chi^{\lambda}_{S_n} \downarrow^{S_n}_{S_{n-1}}, \phi_{k} \right>
= \sum_{\mu \rightarrow \lambda} \left< \chi^{\mu}_{S_{n-1}}, \phi_{k} \right>~.$$
We have previously established by Theorem \ref{th:QPkhalfcentralizer} that
$$\left< \chi^{\mu}_{S_{n-1}}, \phi_{k} \right> =
{\mathrm{dim}} \Big(V^\mu_{\mathsf{QP}_{k+\frac{1}{2}}(n)}\Big)$$
since (again, by the double centralizer theorem \cite{P})
the dimension of an irreducible $\mathsf{QP}_{k+\frac{1}{2}}(n)$ module
indexed by a partition $\mu \vdash n-1$ in a centralizer is equal to the multiplicity
of ${\mathbb{S}}_{S_{n-1}}^\mu$ in ${\mathrm{Res}}^{S_n}_{S_{n-1}}({\mathbb S}^{(n-1,1)}_{S_n})^{\otimes k}$.
\end{proof}

\begin{theorem}
Let $n \geq 2k+2$, then
for $\lambda \vdash n$ such that $|\o\la| < k$,
\begin{equation}\label{eq:p3}
\dim (V^{\lambda}_{\mathsf{QP}_{k}(n)})
= \dim (V^{\lambda}_{\widetilde{\mathsf{QP}}_{k}(n)})
-\dim (V^{\lambda}_{\mathsf{QP}_{k-1}(n)})~.
\end{equation}
\end{theorem}

\begin{proof}
Since $n \geq 2k+2$, by Theorem \ref{th:QPktildecentralizer},
${\widetilde{\mathsf{QP}}}_{k}(n) \cong
{\mathrm{End}}_{S_n}((\mathbb{S}^{(n-1,1)}_{S_n})^{\otimes k} \otimes V_n)$,
then by the double centralizer theorem \cite{P}, for a partition $\lambda \vdash n$,
${\mathrm{dim}} (V^\lambda_{{\widetilde{\mathsf{QP}}}_{k}(n)})$
is equal to the multiplicity of
${\mathbb{S}}_{S_n}^{\lambda}$
in $(\mathbb{S}^{(n-1,1)}_{S_n})^{\otimes k} \otimes V_n$.
Since as $S_n$ modules, $V_n \cong \mathbb{S}^{(n)}_{S_n} \oplus \mathbb{S}^{(n-1,1)}_{S_n}$,
then
\begin{align*}
\left(\mathbb{S}^{(n-1,1)}_{S_n}\right)^{\otimes k} \otimes V_n &\cong
\left(\mathbb{S}^{(n-1,1)}_{S_n}\right)^{\otimes k} \otimes \left(\mathbb{S}^{(n)}_{S_n} \oplus \mathbb{S}^{(n-1,1)}_{S_n}\right)\\
&\cong \left(\left(\mathbb{S}^{(n-1,1)}_{S_n}\right)^{\otimes k} \otimes \mathbb{S}^{(n)}_{S_n}\right) \oplus
\left(\mathbb{S}^{(n-1,1)}_{S_n}\right)^{\otimes k+1}\\
&\cong \left(\mathbb{S}^{(n-1,1)}_{S_n}\right)^{\otimes k} \oplus \left(\mathbb{S}^{(n-1,1)}_{S_n}\right)^{\otimes k+1}~.
\end{align*}
Since by Theorem \ref{th:QPkcentralizer} and the
double centralizer theorem \cite{P}, ${\mathrm{dim}} (V^\lambda_{{\mathsf{QP}}_{k}(n)})$
is equal to the multiplicity of ${\mathbb{S}}_{S_n}^{\lambda}$ in
$(\mathbb{S}^{(n-1,1)}_{S_n})^{\otimes k}$, it follows that
\begin{equation}\label{eq:simplerec}
{\mathrm{dim}} (V^\lambda_{{\widetilde{\mathsf{QP}}}_{k}(n)}) =
{\mathrm{dim}} (V^\lambda_{{\mathsf{QP}}_{k}(n)}) + {\mathrm{dim}} (V^\lambda_{{\mathsf{QP}}_{k+1}(n)})~.
\end{equation}
\qedhere
\end{proof}

\begin{figure}[h]
\begin{center}
\includegraphics[width=5.25 in]{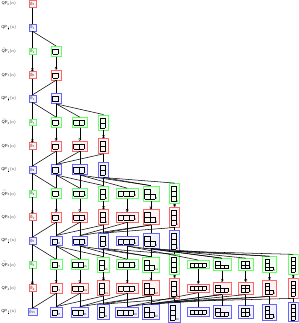}
\end{center}
\caption{A Bratteli diagram showing the relations of the
tower of algebras in Equation \eqref{eq:QPtower}.}\label{fig:Bdiagram}
\end{figure}

\subsection{A combinatorial interpretation for the dimensions of the irreducibles}
As a corollary we can give a combinatorial interpretation for the
dimension of an irreducible representation of each of the algebras.
The dimensions are expressed in terms of set valued tableaux
from Definition \ref{def:settableaux}.

\begin{cor}\label{cor:combdims}
Combinatorial interpretation for dimensions of irreducible representations.
\begin{enumerate}
\item Let $k > 0$ and $n \geq 2k+1$ and $\lambda \vdash n$,
then ${\mathrm{dim}} (V^\lambda_{{\mathsf{QP}}_{k}(n)})$
is equal to the number of $[k]$-set valued tableaux of shape $\lambda$
where the non-empty sets in the first row must be larger than $1$.
\vskip .2in

\item Let $k > 0$ and $n \geq 2k+1$ and $\mu \vdash n-1$,
then ${\mathrm{dim}} (V^\mu_{{\mathsf{QP}}_{k+\frac{1}{2}}(n)})$
is equal to the number of pairs consisting of a cell containing
a subset $S \cup \{k+1\}$ with $S \subseteq [k]$
and a $([k] - S)$-set valued tableau of shape $\mu$
where the non-empty sets in the first row must be larger than $1$.
\vskip .2in

\item Let $k > 0$ and $n \geq 2k+2$ and $\lambda \vdash n$,
${\mathrm{dim}} (V^\lambda_{{\widetilde{\mathsf{QP}}}_{k+1}(n)})$
is equal to the number of $[k+1]$-set valued tableaux of shape $\lambda$
such that if $|S_{(1,j)}|=1$, then $j=\lambda_1$ and
$S_{(1,\lambda_1)} = \{ k+1 \}$.
\end{enumerate}
\end{cor}

We will not include a detailed proof of this corollary here since
it follows by induction and using Equations
\eqref{eq:p1}, \eqref{eq:p2} and \eqref{eq:p3} and either adding a cell or
Schensted row insertion of a cell with a set consisting of the values
that are with the highest label.

\begin{example}\squaresize=18pt
Below we have displayed three set tableaux of $[9]$ and the shape of the tableaux
are $(m,2,1)$ for an appropriate integer $m$.  The first and third satisfy
Corollary \ref{cor:combdims} condition (1) for $k=9$.
The second and third satisfy Corollary \ref{cor:combdims} condition (2) for $k=8$
if the cell containing $9$ is considered separated from the rest of the tableau.
All three tableaux satisfy Corollary \ref{cor:combdims} condition (3) for $k=8$.
\begin{center}
~\young{59\cr1&6\cr&&&\cdots&37&28\cr}\hskip .3in
\young{34\cr2&578\cr&&&\cdots&16&9\cr}\hskip .3in
\young{148\cr35&7\cr&&&\cdots&&269\cr}
\end{center}
\end{example}

Using standard counting techniques, we can enumerate the objects in Corollary \ref{cor:combdims}.
For example, the dimensions in part $(1)$ of
Corollary \ref{cor:combdims} are equal to:
\begin{align}
\text{dim}(V_{\mathsf{QP}_k(n)}^\lambda) &= 
f^{\bar{\lambda}} 
\sum_{t\geq0}
\binom{k}{t}
\stirling{k-t}{|\bar{\lambda}|}
B_2(t)\label{eq:fulldim1}\\
&=\sum_{s \geq 0} (-1)^s \binom{k}{s}
\text{dim}\Big(V_{\mathsf{QP}_{k-s+\frac{1}{2}}(n)}^\mu\Big)\label{eq:mobiusformula}
\end{align}
where $B_2(t)$ is the number of set partitions of $t$ into blocks of size $\geq 2$.

Equation \eqref{eq:mobiusformula}
follows from M\"obius inversion
because when we count 
the dimensions in Corollary \ref{cor:combdims} part $(2)$
we deduce
\begin{align}
{\mathrm{dim}} \Big(V^\mu_{{\mathsf{QP}}_{k+\frac{1}{2}}(n)}\Big)
&=f^{\bar{\mu}}
\sum_{s \geq 0} \binom{k}{s}
\sum_{t\geq0}
\binom{k-s}{t}
\stirling{k-s-t}{|\bar{\mu}|}
B_2(t)\label{eq:halfdim1}\\
&=\sum_{s \geq 0} \binom{k}{s}
\text{dim}(V_{\mathsf{QP}_{k-s}(n)}^\mu)~.\label{eq:halfdim2}
\end{align}
We note that Equations \eqref{eq:fulldim1} and
\eqref{eq:mobiusformula} are equal to
the dimensions stated in Corollary
\ref{cor:dimDeltaknu} which is of a different form.

We do not include a corresponding equation for the dimension of
${\mathrm{dim}} (V^\lambda_{{\widetilde{\mathsf{QP}}}_{k}(n)})$
because counting the tableaux from Corollary \ref{cor:combdims} part $(3)$
yields a formula that also follows from
Equation \eqref{eq:simplerec}.

The formula for the dimension of the irreducible $\mathsf{QP}_k(n)$-modules
stated in Equation \eqref{eq:fulldim1}
seems significantly simpler than the formula in \cite[Proposition 2]{CG},
\cite[Corollary 3.4]{Scrimshaw}
and \cite[Theorem 4.6]{DO}
which counts ``Kronecker tableaux,'' the name that they give to paths
in the Bratteli diagram of the tower of quasi-partition algebras.

\end{document}